\documentclass[11pt,reqno]{amsart}
\usepackage[english]{babel}
\usepackage{amsmath,amssymb,amsthm,mathrsfs}
\usepackage{graphicx}
\usepackage{cite}
\usepackage{url}
\usepackage[margin=30mm]{geometry}
\providecommand{\T}{\mathbb T}
\providecommand{\av}[2]{\langle #1\rangle_{#2}}
\DeclareMathOperator{\Vol}{Vol}
\DeclareMathOperator{\Rm}{Rm}
\DeclareMathOperator{\dist}{dist}
\DeclareMathOperator{\tr}{tr}
\newtheorem{teo}{Theorem}

\newtheorem{prop}{Proposition}
\newtheorem{cor}{Corollary}
\newtheorem{tdef}{Definition}

\newtheorem{theorem}{Theorem}[section]
\newtheorem{lemma}[theorem]{Lemma}

\newtheorem{remark}{Remark}[section]
\numberwithin{equation}{section}

\title[Concentration and strong cosmic censorship in Gowdy spacetimes]{{Metric-Measure Concentration and Strong Cosmic Censorship in Vacuum Gowdy Spacetimes}}
\author[S. L. Cacciatori]{Sergio Luigi Cacciatori$^1$}
\thanks{$^1${Department of Science and High Technology, Universit\`a dell'Insubria, Via Valleggio 11, IT-22100 Como, Italy and
INFN, via Celoria 16, IT-20133 Milano Italy}}

\author[P. Ursino]{Pietro Ursino$^{3,4}$}
\thanks{$^3$ {Department of Mathematics and Informatics, Universit\`a degli Studi di Catania, Viale Andrea Doria 6, 95125 Catania, Italy}}
\thanks{$^4$  {National Group for Algebraic and Geometric Structures and their Applications (GNSAGA) Indam }}

\begin{document}

\begin{abstract}
{\emergencystretch=3em
We prove that a smooth vacuum $T^3$-Gowdy maximal globally hyperbolic
 development is $C^2$-inextendible if its asymptotic velocity differs
from one on a dense subset of the singular circle. This is the
implication posed in Section~14.2 of Ringstr\"om's 2010 review.
The proof compares the spacetime volume of a timelike boundary collar
with that of a shrinking angular cone, using Killing transport and the
Einstein constraint. 
We connect this criterion to Gromov's observable distance: the
physical spatial mm-spaces satisfy the lower bound
$d_{\rm conc}(X_\tau,*)\ge L(\tau)/(4+L(\tau))$, where $L(\tau)$ is the
length of the orbit circle. If $L(\tau)\to0$, then $v_\infty\ge1$
everywhere, the critical set $\{v_\infty=1\}$ has Lebesgue measure zero,
and the development is $C^2$-inextendible. Thus physical point
concentration is sufficient for inextendibility, but is not necessary.
We also give explicit, non-perturbative subcritical characteristic
conditions on vacuum Cauchy data ensuring quantitative Kretschmann
blow-up without using Gowdy asymptotic theorems in the curvature
estimate. The exact volume-transfer identity yields predictive
monotonicity criteria and distinguishes coordinate localization from
intrinsic metric collapse. In the transversely expanding regime,
observable rescalings with an extended transverse-length limit are
classified as a point, a measured circle, or the maximal pyramid;
a uniform spectral estimate eliminates all nontrivial toroidal Fourier
sectors, including in the presence of spikes. The geometric framework
is completed by quantitative tube estimates, a local CMC redistribution
theorem, and constrained variational descriptions of the reduced
energy current and normalized-volume flow.
}
\end{abstract}

\maketitle
\pagestyle{plain}

\section*{Introduction}
{\emergencystretch=3em
The strong cosmic censorship problem concerns the impossibility of
continuing a maximal globally hyperbolic development (MGHD) beyond its
Cauchy boundary with prescribed regularity. In the vacuum $T^3$-Gowdy
class, Ringstr\"om established generic $C^2$-inextendibility in
\cite{RingstromAnnals}. In the earlier work \cite{RingstromDense2005},
he proved curvature blow-up on a dense subset of the singularity for a
residual class of smooth initial data. Section~14.2 of his review
\cite{RingstromGowdy} asks whether the latter, weaker pointwise
condition already suffices for inextendibility. The first purpose of
the present work is to establish this implication. The second is to
relate Lorentzian inextendibility to intrinsic metric--measure
properties of the evolving spatial leaves. We distinguish throughout
curvature blow-up, geodesic incompleteness, coordinate concentration,
and metric--measure collapse.

\subsection*{Main results: from curvature to inextendibility}
Let $v_\infty(\theta)$ denote the Gowdy asymptotic velocity and set
$Z_{\rm crit}=\{\theta\in S^1:v_\infty(\theta)=1\}$. For smooth
admissible vacuum data in the global areal class, our first result is
Theorem~\ref{dense:thm:main}:
\[
 \boxed{\quad\operatorname{int}(Z_{\rm crit})=\varnothing
 \quad\Longrightarrow\quad \text{$C^2$-inextendibility of the MGHD}.\quad}
\]
Equivalently a $C^2$ extension would force an interval of velocity-one
points. The extension need not solve Einstein's equations or preserve
the Gowdy action. We combine Ringstr\"om's pointwise curvature theorem
with a new geometric argument: in an arbitrary timelike boundary chart,
the map assigning angular endpoints to timelike coordinate segments
cannot be constant. Indeed, Killing transport bounds the squared orbit
area by the collar thickness $r$, and the Einstein constraint gives an
upper bound of order $r^{5/4}$ for the volume if all these endpoints
coincide. Nondegeneracy of the extension requires a lower bound of
order $r$. Connectedness and continuity then force an angular interval.

\subsection*{From metric--measure concentration to censorship}
Let $X_\tau$ be the areal Cauchy leaf with physical spatial distance
and normalized Riemannian volume, and let $L(\tau)$ denote the
circumference of its orbit circle. The exact projection onto the
uniformly measured circle gives (Lemma~\ref{dense:lem:length})
\[
 d_{\rm conc}(X_\tau,*)\ge \frac{L(\tau)}{4+L(\tau)}.
\]
The critical-point estimates of Section~\ref{sec:concentration-censorship},
combined with Fatou's lemma and the dense-curvature theorem, yield
Theorem~\ref{dense:thm:conc} and the following chain:
\[
\begin{gathered}
 d_{\rm conc}(X_\tau,*)\to0\quad\Longrightarrow\quad L(\tau)\to0\\
 \Longrightarrow\quad v_\infty\ge1\ \text{everywhere},\quad
 \mathcal L^1(Z_{\rm crit})=0\\
 \Longrightarrow\quad\operatorname{int}(Z_{\rm crit})=\varnothing
 \quad\Longrightarrow\quad \text{$C^2$-inextendibility}.
\end{gathered}
\]
Here $\mathcal L^1$ is the fixed angular Lebesgue measure. The length-collapse condition is
weaker than physical point concentration and does not require a
spectral description of the toroidal fibres. In this direction
concentration is a sufficient geometric criterion for censorship,
not an equivalent characterization.

\subsection*{Prediction from the Cauchy data}
A complementary route uses the initial characteristic speeds of the
wave map $u=(P,Q)$ into the hyperbolic plane with
$h=dP^2+e^{2P}dQ^2$. The strict initial condition
\begin{equation}\label{eq:intro-characteristic-criterion}
 \mathcal P_0:=1-\sup_{\theta\in S^1}\max\left\{
 |u_\tau+e^{-\tau_0}u_\theta|_h^2,
 |u_\tau-e^{-\tau_0}u_\theta|_h^2\right\}_{\tau=\tau_0}>0
\end{equation}
forces exponential growth of the spatial supremum of the spacetime
Kretschmann scalar (Theorem~\ref{thm:revision-nonperturbative-K}).
The proof uses a characteristic maximum principle, derivative-energy
estimates and an integrated orthonormal-frame curvature limit, with no
asymptotic theorem of Ringstr\"om and no higher Sobolev smallness
condition. It includes a relatively open genuinely non-polarized
class. The same characteristic bound implies $Z_{\rm crit}=\varnothing$,
so the dense-curvature criterion yields global $C^2$-inextendibility
(Corollary~\ref{dense:cor:initial}). This last inference, unlike the
quantitative curvature bound, does use Ringstr\"om's blow-up theorem.
The subcritical criterion implies transverse expansion, not physical
point concentration. This distinction shows why curvature singularities
and spatial degeneration require separate criteria.

Our exact redistribution law for normalized volume relates the
logarithmic odds of a material region to the time-integrated regional
energy gap. The polarized Fourier--Bessel profile yields a preserved
initial-data chain (Theorem~\ref{thm:chain}), while a relatively open
low-velocity class of general Gowdy data has no reversal of the volume
transfer (Theorem~\ref{thm:new-predictive-low}). The latter uses
Ringstr\"om's stability theory. Without such control an initially
positive transfer gap can change sign. The high-velocity examples
show why coordinate localization alone cannot classify the intrinsic
metric limit: different torus markings give different collapse
behavior, even at identical normalized-volume transfer.

\subsection*{Observable limits and persistent modes}
In the regime $L(\tau)\to\infty$ we prove a spectral bound of order
$L(\tau)^2$ for the non-invariant toroidal Fourier sectors at the
length-normalized scale, using a hyperbolic escape estimate and a
one-dimensional exceptional-set argument. No uniform asymptotic
expansion at spikes is needed. The observable rescaling trichotomy
(Theorem~\ref{thm:gowdy-spike-rescaling-trichotomy}) is expressed in
terms of the extended limit of $a_\tau L(\tau)$: a point, a uniformly
measured circle of the corresponding finite positive length, or the
maximal pyramid in the weak pyramid topology. Persistent Fourier
observables are precisely constrained by the invariant sector in the
bounded-energy sense of Corollary~\ref{cor:gowdy-spike-persistent-modes}.
A strict initial mean-energy inequality already guarantees the
expanding hypothesis (Proposition~\ref{prop:gowdy-initial-energy-expansion}).

\subsection*{Geometric framework and organization}
The underlying quantitative geometry is applicable beyond the Gowdy
model. A fixed-dimensional tube obstruction follows from the normal
Jacobi equation; in growing dimension the exact round-sphere threshold
is $q_j=o(n_j)$ for concentration about great subspheres. The general
normalized-volume identity and its local CMC version explain the
role of the constraints, the lapse and the cumulative transfer gap.
Two constrained least-action identities identify the reduced current
and the Fisher--Rao probability flow.

Section~\ref{sec:gowdy-reduction} fixes notation and the Einstein
reduction. Sections~\ref{sec:dense-curvature-inextendibility} and
\ref{sec:concentration-censorship} contain the principal Lorentzian
and metric--measure criteria. Section~\ref{sec:revision-gowdy-curvature}
proves the characteristic Cauchy-data curvature estimate.
Section~\ref{sec:predictive-transfer} develops the predictive volume
transfer and examples; Section~\ref{sec:gowdy-observable-pyramids}
contains the observable and spectral classification. The variational
analysis is in Section~\ref{sec:revision-gowdy-variational}.
Sections~\ref{sec:ADM} and \ref{sec:vacuum-concentration}
explain the general Einstein and CMC framework. The final geometric
sections supply the normal-Jacobi, nearest-point, and tube estimates
and their exact realization in the Gowdy examples.
}

\clearpage
\setcounter{tocdepth}{1}
\tableofcontents

\section{{Gowdy reduction and geometric notation}}\label{sec:gowdy-reduction}
{
We work with the global areal vacuum $T^3$-Gowdy class, where
$t\in(0,\infty)$, the angular variables $(\theta,x,y)$ have period
$2\pi$, and $\tau=-\log t$ increases toward the singular end.
The spatial metric induced on a leaf is denoted by $\gamma_\tau$;
$\omega_\tau=dV_{\gamma_\tau}/\Vol_{\gamma_\tau}(\Sigma_\tau)$.
The unrescaled metric--measure space is
$X_\tau=(\Sigma_\tau,d_{\gamma_\tau},\omega_\tau)$.
We use $\langle f\rangle_{S,\tau}$ for the conditional mean with
respect to $\omega_\tau$ on a region of positive probability and
$\langle f\rangle_{\omega_\tau}$ for the full mean.
}

The vacuum Gowdy metric, with no polarization assumption, is
\begin{equation}\label{eq:general-gowdy-metric}
 \mathbf g=A^2(-dt^2+d\theta^2)
 +t\bigl[e^P(dx+Q\,dy)^2+e^{-P}dy^2\bigr],
 \qquad A=t^{-1/4}e^{\lambda/4}.
\end{equation}
The Einstein equations give
\begin{align}\label{eq:nonpolarized-wave}
 P_{\tau\tau}-e^{-2\tau}P_{\theta\theta}
 &=e^{2P}(Q_\tau^2-e^{-2\tau}Q_\theta^2),\\
 Q_{\tau\tau}-e^{-2\tau}Q_{\theta\theta}
 &=-2(P_\tau Q_\tau-e^{-2\tau}P_\theta Q_\theta),\nonumber
\end{align}
and the equations for the remaining metric function are
\begin{align}\label{eq:nonpolarized-lambda}
 \lambda_\tau&=-\mathcal H,\\
 \lambda_\theta&=-2\mathcal J,\nonumber
\end{align}
where
\begin{align}\label{eq:gowdy-HJ}
 \mathcal H&=P_\tau^2+e^{2P}Q_\tau^2+
   e^{-2\tau}(P_\theta^2+e^{2P}Q_\theta^2),\\
 \mathcal J&=P_\tau P_\theta+e^{2P}Q_\tau Q_\theta.\nonumber
\end{align}
Indeed, $\lambda_\theta=2t(P_tP_\theta+e^{2P}Q_tQ_\theta)$,
so replacing $t\partial_t$ by $-\partial_\tau$ gives the stated sign.
The periodic momentum constraint is
$\int_{S^1}\mathcal J\,d\theta=0$; it does not control the
sign of $\mathcal J$ at the boundary of a chosen region.

\paragraph{A wave-map balance law.}
The pair $X=(P,Q)$ is a map into the hyperbolic plane with
$g_{\mathbb H}=dP^2+e^{2P}dQ^2$. Equation
\eqref{eq:nonpolarized-wave} reads
$D_\tau X_\tau-e^{-2\tau}D_\theta X_\theta=0$.
Metric compatibility and torsion freeness then give the exact
local identity
\begin{equation}\label{eq:gowdy-local-energy}
 \boxed{\ \partial_\tau\mathcal H
 =2e^{-2\tau}\partial_\theta\mathcal J
  -2e^{-2\tau}|X_\theta|_{\mathbb H}^2\ }.
\end{equation}
Consequently
\begin{equation}\label{eq:gowdy-global-energy}
 \frac d{d\tau}\int_{S^1}\mathcal H\,d\theta
 =-2e^{-2\tau}\int_{S^1}|X_\theta|_{\mathbb H}^2\,d\theta\le0,
\end{equation}
whereas for a proper interval $I$ one retains the boundary term
$2e^{-2\tau}[\mathcal J]_{\partial I}$, with no fixed sign.

{
The positive definite orbit matrix $h_\tau$ has determinant
$e^{-2\tau}$. Set $L(\tau)=\int_{S^1} A(\tau,\theta)\,d\theta$.
Then $dV_{\gamma_\tau}=e^{-\tau}A\,d\theta\,dx\,dy$, and the
normalized angular marginal is $A\,d\theta/L(\tau)$.
We keep the physical distance unless a rescaling is explicitly stated.

The observable distance $d_{\rm conc}$ is the infimum of the Hausdorff
distance in the convergence-in-measure pseudometric between the two
pulled-back sets of real $1$-Lipschitz functions, over common measure
parametrizations (see Section~\ref{sec:mm-preliminaries} and \cite{Shi}).
In particular, $d_{\rm conc}(X,*)<\varepsilon$ implies that every
$1$-Lipschitz $f$ on $X$ is within $\varepsilon$ of a constant outside
a set of measure at most $\varepsilon$.
}

{
\section{{Dense curvature blow-up and $C^2$-inextendibility}}
\label{sec:dense-curvature-inextendibility}
Let $(M,g)$ be a smooth vacuum $T^3$-Gowdy development in the standard global areal class,
with Cauchy surfaces $\{t\}\times\T^3$, $0<t<\infty$. Let
\[
 v_\infty(\theta)=\lim_{\tau\to\infty}
 \sqrt{P_\tau^2+e^{2P}Q_\tau^2},\qquad
 Z_{\rm crit}=\{\theta\in S^1:v_\infty(\theta)=1\}.
\]
Existence of this limit and its upper semicontinuity follow from
\cite{RingstromVelocity,RingstromAnnals}. In this section we abbreviate $Z_{\rm crit}$ by $Z$, not to be confused with an orbit torus or a tubular concentration locus.
The dense $G_\delta$ set in \cite{RingstromDense2005} is precisely the class for which $S^1\setminus Z$ is dense. In Section~14.2 of \cite{RingstromGowdy}, Ringstr\"om asks whether this property suffices for $C^2$-inextendibility. The following theorem gives the implication; it does not require curvature blow-up at every angular endpoint. 

\begin{theorem}[Dense curvature criterion]\label{dense:thm:main}
Let $(M,g)$ be the maximal globally hyperbolic development of smooth, admissible vacuum $T^3$-Gowdy data in the standard areal class. If
\[
 \overline{\{\theta:v_\infty(\theta)\ne1\}}=S^1,
\]
then $(M,g)$ is $C^2$-inextendible as a Lorentzian manifold. Equivalently, any $C^2$-extension forces $Z$ to contain a nonempty open interval.
\end{theorem}
Here a $C^2$-extension is an isometric embedding into a connected Lorentzian manifold of the same dimension with a $C^2$ metric, with proper image. The extended metric need not satisfy the vacuum field equations and need not retain any symmetry.

\begin{cor}[Null critical set]\label{dense:cor:null}
If $\mathcal L^1(Z)=0$, where $\mathcal L^1$ is the angular Lebesgue measure on $S^1$, then $(M,g)$ is $C^2$-inextendible.
\end{cor}
\begin{proof}
A set of zero Lebesgue measure contains no nonempty open interval. Apply Theorem~\ref{dense:thm:main}.
\end{proof}

\paragraph{Areal coordinates and spacetime volume.}
Write $t=e^{-\tau}$ and let $(\theta,x,y)$ have period $2\pi$.
We use the sign convention of \eqref{eq:general-gowdy-metric}--\eqref{eq:gowdy-HJ}, namely
\[
 g=A^2(-dt^2+d\theta^2)+t[e^P(dx+Qdy)^2+e^{-P}dy^2],
 \qquad A=e^{(\tau+\lambda)/4},\qquad \lambda_\tau=-\mathcal H.
\]
Thus our $\lambda$ is the negative of Ringstr\"om's $\lambda$ in
\cite{RingstromAnnals}. The time orientation in the following argument may be chosen with $t$ increasing; reversing it does not change inextendibility.

Let $X=\partial_x$ and $Y=\partial_y$ be the globally defined Killing fields. The determinant of their Gram matrix is
\begin{equation}\label{dense:eq:D}
 D=g(X,X)g(Y,Y)-g(X,Y)^2=t^2.
\end{equation}
The determinant of the spacetime metric in \eqref{eq:general-gowdy-metric} is $-A^4t^2$. Hence its positive Lorentzian volume density is
\begin{equation}\label{dense:eq:vol}
 d\Vol_g=A^2t\,dt\,d\theta\,dx\,dy
         =t^{1/2}e^{\lambda/2}\,dt\,d\theta\,dx\,dy.
\end{equation}
For $0<t\le1$, \eqref{eq:nonpolarized-lambda} gives $\lambda(t,\theta)\le\max_{S^1}\lambda(1,\cdot)$. Consequently there is a uniform constant $C$ such that
\begin{equation}\label{dense:eq:volbound}
 d\Vol_g\le C t^{1/2}\,dt\,d\theta\,dx\,dy \qquad(0<t\le1).
\end{equation}

\begin{lemma}[Angular cones]\label{dense:lem:cone}
For any causal curve parametrized by $t$ one has $|d\theta/dt|\le1$, in a local lift of the angular coordinate. A causal curve with a past end at $t=0$ therefore has an angular endpoint $\psi\in S^1$ and satisfies $d_{S^1}(\theta(t),\psi)\le t$, with the standard angular distance. Moreover, uniformly in $\psi$,
\begin{equation}\label{dense:eq:conevol}
 \Vol_g\big\{0<t<T,\ d_{S^1}(\theta,\psi)\le t\big\}
 \le C_1 T^{5/2},\qquad 0<T\le1.
\end{equation}
\end{lemma}
\begin{proof}
The orbit block of \eqref{eq:general-gowdy-metric} is positive definite. The causal inequality thus implies $A^2[-1+(d\theta/dt)^2]\le0$, giving the Lipschitz bound and existence of the angular limit. At time $t$, the indicated angular arc has Lebesgue length at most $2t$. Integrating \eqref{dense:eq:volbound} over this arc and the whole orbit torus yields
\[
 C(2\pi)^2\int_0^T 2t\,t^{1/2}\,dt
 =\tfrac45 C(2\pi)^2 T^{5/2}.
\]
\end{proof}

\paragraph{Timelike boundary charts and the squared orbit area.}
Assume that a proper $C^2$-extension $\iota:(M,g)\hookrightarrow(\widetilde M,\widetilde g)$ exists and identify $M$ with its image. At least one of the timelike future and past boundary sets is nonempty:
in a small Lorentzian chart about a boundary point, nearby points can be joined to a common nearby point by timelike segments, and one of these segments must leave the open embedded spacetime. Only a local time orientation is needed in the extension. By the boundary-chart theorem for globally hyperbolic spacetimes (\cite[Proposition~2.2 and Definition~2.1]{Sbierski2018}, with time reversed if necessary), there is a chart $(s,z)$, $z\in\mathbb R^3$, and a Lipschitz graph $s=f(z)$ such that its subgraph lies in $M$, the graph lies in $\partial M$, and the coordinate lines in the $s$ direction are timelike. The curves with $s$ increasing towards $f(z)$ all have the same time orientation in $M$. Only this one-sided chart is needed; no claim about the geometry of the complementary side is required.

After shrinking the chart, choose a closed cube $B\subset\mathbb R^3$ of positive Euclidean volume, an interior starting slice $s=s_*$ with $s_*<\min_B f$, and a constant $r_0>0$ such that the region between this slice and the graph over $B$ is relatively compact in the chart, and $r_0<\min_B(f-s_*)$. The density $\sqrt{|\det\widetilde g|}$ is then bounded above and bounded below by positive constants on this compact region. Since $\widetilde g$ is $C^2$, its connection and curvature components are bounded there as well.

For $z\in B$ the curve $c_z(s)=(s,z)$, $s_*\le s<f(z)$, is timelike and has a boundary endpoint. Its end is inextendible in $M$. Since every $t$-slice is Cauchy and $t$ increases towards the expanding future, $t(c_z(s))\to0$ or $t(c_z(s))\to+\infty$, according to its orientation. Indeed, a timelike end which stopped at an interior value of $t$ would contradict the Cauchy property of the later (or earlier) slices after prolongation of the opposite end. This also follows from the uniform bounds on the spatial coordinate speeds on compact slabs of areal time.

\begin{lemma}[Uniform orbit determinant bound]\label{dense:lem:det}
On any such chart the timelike curves approaching the graph must approach $t=0$. There exists a constant $C_2>0$, uniform in $z\in B$, such that
\begin{equation}\label{dense:eq:Dlinear}
 t(c_z(s))^2\le C_2(f(z)-s),\qquad s_*\le s<f(z).
\end{equation}
\end{lemma}
\begin{proof}
For a Killing field $V$ on $M$, put $F=\nabla V$. The Killing identity expresses $\nabla F$ as a contraction of the Riemann tensor with $V$. In particular, along $c_z$ the pair $(V,F)$ solves a homogeneous linear ordinary differential system: $\nabla_{\dot c_z}V=F(\dot c_z)$ and $\nabla_{\dot c_z}F$ is linear in $V$ with curvature coefficients. In chart components, the additional coefficients are Christoffel symbols contracted with $\partial_s$. All coefficients are uniformly bounded on the compact chart region. The initial values of $(V,F)$ at $(s_*,z)$ are uniformly bounded, since this slice over $B$ is compact and contained in $M$. Gronwall's inequality gives a uniform bound on $V$ and $\nabla V$ on all the curves under consideration. This is the Killing transport argument used in \cite[Section~2]{DafermosRendall2005}; the present use only requires the one-sided bounds.

Apply it to $X$ and $Y$. The Gram determinant $D$ is uniformly bounded, as is its coordinate derivative $\partial_sD$: differentiate the three inner products in \eqref{dense:eq:D} using metric compatibility and the bounds for the fields and their covariant derivatives. By $D=t^2$, a future boundary chart with $t\to\infty$ is impossible. Thus the chart approaches $t=0$. On every vertical curve $D\to0$. The uniform derivative bound and the fundamental theorem of calculus give
\[
 |D(c_z(s))|\le\int_s^{f(z)}|\partial_sD(c_z(u))|\,du
 \le C_2(f(z)-s).
\]
This proves \eqref{dense:eq:Dlinear}. 
\end{proof}

\paragraph{The noncollapse of angular endpoints.}
\begin{prop}[Angular endpoint interval]\label{dense:prop:interval}
For any timelike boundary chart as above the angular endpoints of the curves $c_z$, $z\in B$, contain a nonempty open interval of $S^1$.
\end{prop}
\begin{proof}
For each $z$ define $\Theta(z)=\lim_{s\uparrow f(z)}\theta(c_z(s))$. By Lemma~\ref{dense:lem:cone} this limit exists, and
\[
 d_{S^1}(\theta(c_z(s)),\Theta(z))\le t(c_z(s)).
\]
For $0<r<r_0$ the maps $z\mapsto\theta(c_z(f(z)-r))$ are continuous on $B$. By \eqref{dense:eq:Dlinear} they converge uniformly, with angular error at most $\sqrt{C_2r}$, to $\Theta$. Hence $\Theta$ is continuous.

The compact cube $B$ is connected. A connected subset of a circle with more than one point contains a nonempty open interval. It therefore suffices to rule out $\Theta\equiv\psi$ on $B$.

Suppose that $\Theta$ is constant. For $0<r<r_0$ consider the one-sided collar
\[
 U_r=\{(s,z):z\in\operatorname{int} B,\ f(z)-r<s<f(z)\}\subset M.
\]
By Fubini's theorem its Euclidean coordinate volume is $r|B|$. The uniform positive lower bound on the Lorentzian volume density of the extension gives
\begin{equation}\label{dense:eq:lower}
 \Vol_g(U_r)=\Vol_{\widetilde g}(U_r)\ge c_0 r,
 \qquad c_0>0.
\end{equation}
On the other hand, for every point of the collar, Lemmas~\ref{dense:lem:cone} and \ref{dense:lem:det} give
\[
 0<t\le \sqrt{C_2r},\qquad d_{S^1}(\theta,\psi)\le t.
\]
Thus $U_r$ is contained in the angular cone of Lemma~\ref{dense:lem:cone} with $T=\sqrt{C_2r}$. For sufficiently small $r$,
\begin{equation}\label{dense:eq:upper}
 \Vol_g(U_r)\le C_1(C_2r)^{5/4}.
\end{equation}
Combining \eqref{dense:eq:lower} and \eqref{dense:eq:upper} gives $c_0\le C_1 C_2^{5/4}r^{1/4}$, a contradiction as $r\downarrow0$. The endpoint map cannot be constant, and its image contains an interval.
\end{proof}

\begin{proof}[Proof of Theorem~\ref{dense:thm:main}]
An extension yields a boundary chart by the preceding construction. All its vertical curves end at $t=0$. The Kretschmann scalar is continuous in the $C^2$ extension and therefore bounded near each boundary endpoint. Ringstr\"om's curvature blow-up result \cite[Proposition~1.19 and its proof]{RingstromVelocity} (see also \cite[Section~1.3]{RingstromAnnals} and \cite[Section~10.2]{RingstromGowdy}) states that if $v_\infty(\psi)\ne1$, then the Kretschmann scalar is unbounded along every causal curve having angular endpoint $\psi$ at $t=0$. Hence $v_\infty(\Theta(z))=1$ for every $z\in B$. Proposition~\ref{dense:prop:interval} implies that $Z$ contains a nonempty open interval. This is incompatible with the density assumption. The existence of an extension has been contradicted.
\end{proof}

}

{
\section{{Transverse collapse and metric--measure criteria for inextendibility}}\label{sec:concentration-censorship}
In this section write $Z=Z_{\rm crit}$ as in Section~\ref{sec:dense-curvature-inextendibility}.
Equip $\Sigma_\tau=\{\tau\}\times\T^3$ with its physical Riemannian metric $g_\tau$ induced by \eqref{eq:general-gowdy-metric} and its normalized volume measure $\mu_\tau$. Write $\mathsf X_\tau=(\Sigma_\tau,d_{g_\tau},\mu_\tau)$ and let $*$ be the one-point metric-measure space. This is the unscaled $X_\tau$ introduced in Section~\ref{sec:gowdy-reduction}. We use the usual observable distance $d_{\mathrm{conc}}$ of Gromov, as in \cite[Chapter~5]{Shi}. No length rescaling is performed.

\begin{lemma}[Quantitative orbit-length bound]\label{dense:lem:length}
If $L(\tau)=\int_{S^1}A(\tau,\theta)\,d\theta$, then
\begin{equation}\label{dense:eq:conc}
 d_{\mathrm{conc}}(\mathsf X_\tau,*)\ge\frac{L(\tau)}{4+L(\tau)}.
\end{equation}
\end{lemma}
\begin{proof}
The projection to the orbit circle with metric $A^2d\theta^2$ is $1$-Lipschitz. Its pushforward of the normalized volume is $A\,d\theta/L$: the determinant of the spatial orbit block is $e^{-2\tau}$, independent of $\theta$. Fix a point of the orbit circle and let $u$ be its distance function, pulled back to the leaf. For any $a\in\mathbb R$ and $\varepsilon>0$, the set $\{|u-a|\le\varepsilon\}$ has normalized measure at most $4\varepsilon/L$, since on the circle each distance value has at most two preimages and the two arc portions have total length at most $4\varepsilon$. The definition of $d_{\mathrm{conc}}$ against a point implies that when $d_{\mathrm{conc}}<\varepsilon$ a constant $a$ exists with $\mu_\tau(|u-a|\le\varepsilon)\ge1-\varepsilon$. Rearranging gives $\varepsilon\ge L/(4+L)$, and letting $\varepsilon$ decrease to $d_{\mathrm{conc}}$ proves the result.
\end{proof}

\begin{lemma}[Critical-point noncollapse]\label{dense:lem:critical}
For each $\psi\in Z$ there exists $a(\psi)\in(0,\infty)$ such that $A(\tau,\psi)\to a(\psi)$ as $\tau\to\infty$.
\end{lemma}
\begin{proof}
We give the analytic justification, which is not needed for the dense-curvature theorem. By \cite[Proposition~1]{RingstromAnnals}, at $\psi$ the potential energy $e^{-2\tau}(P_\theta^2+e^{2P}Q_\theta^2)$ tends to zero, $e^P Q_\tau\to0$, and $P_\tau\to\pm1$. After a hyperbolic inversion if necessary, we may arrange $P_\tau\to1$ without changing $\mathcal H$. Locally on a simply connected angular lift, apply the Gowdy--Ernst transformation \cite[Eq.~(15)]{RingstromAnnals},
\[
 \widehat P=\tau-P,\qquad
 \widehat Q_\tau=-e^{2(P-\tau)}Q_\theta,\qquad
 \widehat Q_\theta=-e^{2P}Q_\tau.
\]
It is a transformation of the evolution equations; no global periodicity or constraint on the auxiliary solution is required. Its asymptotic velocity at $\psi$ is zero, by the estimates just cited. Substitution yields
\begin{equation}\label{dense:eq:GEenergy}
 \mathcal H-1=\widehat{\mathcal H}-2\widehat P_\tau.
\end{equation}
We explain the applicability of the zero-velocity estimates. Upper semicontinuity of the asymptotic velocity and the local cutoff of \cite[Theorem~1 and Lemma~14]{RingstromAnnals} permit replacement of the auxiliary wave map by a smooth periodic one agreeing on a fixed interval containing $\psi$ for all sufficiently large $\tau$ and with asymptotic velocity bounded by, say, $1/8$ everywhere. The local cutoff is applied to the auxiliary solution on the angular universal cover. Choose the localization interval and its padding to have total length strictly less than $2\pi$. The resulting disk-valued solution is zero outside the padded interval at late times and hence can be periodized on the original angular circle, as in the proof of Lemma~16 of \cite{RingstromAnnals}. In the disk model let $w$ denote this solution and $\rho$ its hyperbolic distance from the disk origin.

For clarity, the condition $\rho\le\tau-2$ needed below follows from the characteristic estimate of \cite[Lemma~1]{RingstromAnnals}. In fact for each point, Proposition~1 implies that on the shrinking characteristic cone the kinetic speed tends to its asymptotic speed, and the scaled spatial speed tends to zero. At a sufficiently late time the characteristic energy $e^{-\tau}F_I$ of Lemma~1 is therefore small on a sufficiently small interval $I$ about that point (enlarge its cone section slightly by continuity). Lemma~1 propagates this upper bound on that interval. A finite cover of the periodic circle then gives a uniform eventual bound $|\partial_\tau w|_{\mathbb H^2}<1/2$. Integrating the distance from the origin yields $\rho\le\tau/2+C\le\tau-2$ for large $\tau$. More explicitly, $e^{-\tau}F_I$ is one half the sum of the squared suprema of the characteristic speeds $|\partial_\tau w\pm e^{-\tau}\partial_\theta w|_{\mathbb H^2}$. Its eventual bound can be chosen less than $1/4$ on each interval because the asymptotic speed is at most $1/8$.

By \cite[Lemma~7 and Corollary~4, with $k=0$]{RingstromAnnals}, at $\psi$ we obtain for some constants $C,m$ (enlarged when necessary)
\[
 |\partial_\tau w|_{\mathbb H^2}(\tau,\psi)
 \le C(1+\tau)^m e^{-2\tau},\qquad
 |w_\theta(\tau,\psi)|_{\mathbb R^2}\le C(1+\tau)^m.
\]
The first bound also shows that $\rho(\tau,\psi)$ remains bounded. The disk metric is uniformly comparable to the Euclidean metric on the resulting compact disk at that point. Thus $|\partial_\theta w|_{\mathbb H^2}(\tau,\psi)$ has polynomial growth. Returning to $(\widehat P,\widehat Q)$, these bounds imply
\[
 \int^{\infty}\widehat{\mathcal H}(\tau,\psi)\,d\tau<\infty,
 \qquad
 \int^{\infty}|\widehat P_\tau(\tau,\psi)|\,d\tau<\infty.
\]
Equation \eqref{dense:eq:GEenergy} and the identity $\partial_\tau\log A=(1-\mathcal H)/4$ now imply that $\log A(\tau,\psi)$ converges to a finite real number.
\end{proof}

\begin{theorem}[Transverse collapse and point concentration imply inextendibility]\label{dense:thm:conc}
If $L(\tau)\to0$ as $\tau\to\infty$ (in particular if $d_{\mathrm{conc}}(\mathsf X_\tau,*)\to0$), then
\[
 v_\infty(\theta)\ge1\quad\text{for every }\theta\in S^1,
 \qquad \mathcal L^1(Z)=0,
 \qquad (M,g)\text{ is }C^2\text{-inextendible}.
\]

\end{theorem}
\begin{proof}
The assertion under point concentration follows from the stronger hypothesis of transverse length collapse, by Lemma~\ref{dense:lem:length}. Assume $L(\tau)\to0$. By \cite[Proposition~1]{RingstromAnnals}, $\mathcal H(\tau,\theta)\to v_\infty(\theta)^2$ at every $\theta$. Hence $\tau^{-1}\log A(\tau,\theta)\to(1-v_\infty(\theta)^2)/4$. If $v_\infty<1$ at a point, upper semicontinuity implies that it is less than one on an interval. There $A\to\infty$ pointwise, contradicting $L\to0$ by Fatou's lemma. Thus $v_\infty\ge1$ everywhere. On $Z$, Lemma~\ref{dense:lem:critical} and Fatou give
\[
 0=\liminf_{\tau\to\infty}L(\tau)\ge\int_Z a(\theta)\,d\theta.
\]
The positive pointwise limit $a$ on $Z$ is measurable, so $\mathcal L^1(Z)=0$. Corollary~\ref{dense:cor:null} gives inextendibility.
\end{proof}

\begin{remark}[Distinction between the criteria]
The condition $\mathcal P_0>0$ implies transverse expansion, not
unscaled point concentration. Thus the two sufficient inextendibility
criteria single out different regimes. Moreover, the measure-zero
conclusion concerns fixed angular Lebesgue measure, not the evolved
normalized Riemannian volume. 
\end{remark}
}

{
\section{{A direct initial-data criterion for curvature blow-up}}
\label{sec:revision-gowdy-curvature}
Introduce the real variables
\begin{equation}\label{eq:revision-first-order-vars}
 v=P_\tau,\qquad a=e^{-\tau}P_\theta,\qquad
 w=e^PQ_\tau,\qquad b=e^{P-\tau}Q_\theta,
 \qquad U=(v,a,w,b).
\end{equation}
The Einstein wave-map equations are equivalent to
\begin{equation}\label{eq:revision-gowdy-system}
 \begin{aligned}
 v_\tau&=e^{-\tau}a_\theta+w^2-b^2,\\
 a_\tau&=e^{-\tau}v_\theta-a,\\
 w_\tau&=e^{-\tau}b_\theta-vw+ab,\\
 b_\tau&=e^{-\tau}w_\theta-(1-v)b-aw.
 \end{aligned}
\end{equation}
All the quantities in the following criterion are measured on
one initial areal Cauchy slice $\tau_0$:
\begin{equation}\label{eq:revision-initial-functional}
 \boxed{\begin{aligned}
 \mathcal P_0&=1-\sup_{\theta\in S^1}
                  \max\{E_+(\theta),E_-(\theta)\},\\
 E_\pm(\theta)&=(P_1\pm e^{-\tau_0}P'_0)^2
   +e^{2P_0}(Q_1\pm e^{-\tau_0}Q'_0)^2.
 \end{aligned}}
\end{equation}

\begin{theorem}[Direct non-polarized Cauchy-data criterion]
\label{thm:revision-nonperturbative-K}
Let $(P_0,Q_0,P_1,Q_1)$ be smooth periodic Gowdy vacuum data
satisfying the periodic momentum constraint.  If
$\mathcal P_0>0$, set $m=\sqrt{1-\mathcal P_0}<1$.
For the associated areal development, there is a finite
$T$ (depending on the initial data) such that, for $\tau\ge T$,
\begin{equation}\label{eq:revision-K-lower}
 \boxed{\qquad
 \sup_{\theta\in S^1}\mathcal K(\tau,\theta)
 \ge\frac{3(1-m^2)^2}{8}
 \exp\bigl(3\tau-\sup_{\theta}\lambda(\tau_0,\theta)\bigr).
 \qquad}
\end{equation}
Here $\mathcal K=R_{\alpha\beta\gamma\delta}
R^{\alpha\beta\gamma\delta}$ is the spacetime Kretschmann scalar.
Consequently $\sup_\theta|\mathcal K(\tau,\theta)|\to\infty$.
The initial condition defines a relatively open set of smooth Gowdy
data and imposes no smallness on any higher Sobolev norm.

\end{theorem}

\begin{lemma}[Characteristic maximum principle]
\label{lem:revision-characteristic}
Under the preceding initial condition, for all $\tau\ge\tau_0$,
\begin{equation}\label{eq:revision-pointwise-H}
 v^2+a^2+w^2+b^2\le m^2,
 \qquad \mathcal H\le m^2.
\end{equation}
\end{lemma}
\begin{proof}
Let $V=u_\tau$ and $W=e^{-\tau}u_\theta$ be hyperbolic tangent
vectors and put $Z_\pm=V\pm W$.
The covariant wave-map equations give
\[
 (D_\tau-e^{-\tau}D_\theta)Z_+=-W,
 \qquad(D_\tau+e^{-\tau}D_\theta)Z_-=W.
\]
Hence
\begin{align*}
 (\partial_\tau-e^{-\tau}\partial_\theta)|Z_+|_h^2
 &=-|Z_+|_h^2+\langle Z_+,Z_-\rangle_h,\\
 (\partial_\tau+e^{-\tau}\partial_\theta)|Z_-|_h^2
 &=-|Z_-|_h^2+\langle Z_+,Z_-\rangle_h.
\end{align*}
The characteristic maximum principle gives, at all later times,
\[
\sup_{\theta}\max\{|Z_+|_h,|Z_-|_h\}\le m.
\]
Now $|V|_h^2+|W|_h^2=
\frac12(|Z_+|_h^2+|Z_-|_h^2)$, proving the assertion.
\end{proof}

\begin{lemma}[Decay of normalized derivatives and nonlinear fields]
\label{lem:revision-derivative}
Let $\delta=1-m>0$.  Then
\begin{equation}\label{eq:revision-rescaled-gradient}
 \|e^{-\tau}\partial_\theta U(\tau)\|_{L^2(S^1)}
 \le C e^{-\delta(\tau-\tau_0)}.
\end{equation}
Moreover $a,b,w\to0$ in $L^2(S^1)$, and
$v(\tau,\cdot)\to\bar v$ in $L^1(S^1)$ for some measurable
$|\bar v|\le m$ almost everywhere. Here $\bar v$ is the signed
limit of $P_\tau$, not the nonnegative invariant velocity $v_\infty$
of Section~\ref{sec:dense-curvature-inextendibility}.
\end{lemma}
\begin{proof}
Differentiate \eqref{eq:revision-gowdy-system} in $\theta$;
set $(\xi,\alpha,\zeta,\beta)=\partial_\theta U$.
After integration by parts, the transport terms cancel and
\begin{align}\label{eq:revision-gradient-energy}
 \frac12\frac{d}{d\tau}
 \int(\xi^2+\alpha^2+\zeta^2+\beta^2)\,d\theta
 =\int\bigl[&-\alpha^2-\beta^2+v(\beta^2-\zeta^2)\\
 &+w(\xi\zeta-\alpha\beta)
 +b(\alpha\zeta-\xi\beta)\bigr]d\theta.\notag
\end{align}
The vector
$F=(\beta^2-\zeta^2,\xi\zeta-\alpha\beta,
\alpha\zeta-\xi\beta)$ satisfies
$|F|\le \xi^2+\alpha^2+\zeta^2+\beta^2$:
indeed, writing $A_1=\xi^2+\alpha^2$ and $B_1=\zeta^2+\beta^2$,
each of the last two squared entries is at most $A_1 B_1$ by
Cauchy--Schwarz, while the first is at most $B_1^2$. Hence
$|F|^2\le B_1^2+2A_1B_1\le(A_1+B_1)^2$.
By \eqref{eq:revision-pointwise-H},
$|(v,w,b)|\le m$; dropping the negative terms yields
\[
 \frac{d}{d\tau}\|\partial_\theta U\|_2^2
 \le2m\|\partial_\theta U\|_2^2.
\]
Gronwall implies \eqref{eq:revision-rescaled-gradient}.

For $Y=(a,b)$, the symmetric part of the damping matrix is
\[
 D=\begin{pmatrix}1&w/2\\w/2&1-v\end{pmatrix},
 \qquad
 \lambda_{\min}(D)
 =1-\tfrac12(v+\sqrt{v^2+w^2})\ge1-m=\delta.
\]
The forcing of its equations is
$e^{-\tau}(v_\theta,w_\theta)$, which has $L^2$ norm bounded
by $Ce^{-\delta(\tau-\tau_0)}$.
An $L^2$ energy estimate therefore gives
\[
 \|Y(\tau)\|_2\le
 C(1+\tau-\tau_0)e^{-\delta(\tau-\tau_0)}.
\]
The two exact integral identities
\[
 \frac d{d\tau}\int_{S^1}\mathcal H\,d\theta
 =-2\int_{S^1}(a^2+b^2)\,d\theta,
 \qquad
 \frac d{d\tau}\int_{S^1}v\,d\theta
 =\int_{S^1}(w^2-b^2)\,d\theta
\]
show that $\int_{\tau_0}^\infty\|w(\tau)\|_2^2\,d\tau<\infty$.
The $w$-equation and the bounds above imply that
$\frac d{d\tau}\|w(\tau)\|_2^2$ is bounded.  Thus
$\|w(\tau)\|_2\to0$.
Finally $v_\tau=e^{-\tau}a_\theta+w^2-b^2$ has
integrable $L^1(S^1)$ norm, proving the $L^1$ limit of $v$.
Since $w\to0$ in $L^2$ and $v\to\bar v$ in $L^1$, along a
common subsequence these convergences hold almost everywhere.
The pointwise existence of the invariant speed
$v_\infty=\lim_{\tau\to\infty}\sqrt{v^2+w^2}$ then gives
$v_\infty=|\bar v|$ almost everywhere. This identification is
not needed for the curvature estimate below.
\end{proof}

\begin{lemma}[Integrated Kasner curvature limit]
\label{lem:revision-K-limit}
Write $N=e^{(\lambda-3\tau)/4}$ for the areal lapse.
Under the initial characteristic criterion,
\begin{equation}\label{eq:revision-K-L1}
 N^4\mathcal K(\tau,\cdot)
 \longrightarrow \frac{(1-\bar v^2)^2(3+\bar v^2)}4
 \quad\text{in }L^1(S^1).
\end{equation}
\end{lemma}
\begin{proof}
Use the orthonormal coframe
\[
 e^0=N\,d\tau,\quad e^1=A\,d\theta,\quad
 e^2=e^{(P-\tau)/2}(dx+Q\,dy),\quad
 e^3=e^{-(P+\tau)/2}dy.
\]
The nonzero structure expressions can be written
\begin{align*}
 N\,de^0&=\tfrac12(va+wb)e^0\wedge e^1,\\
 N\,de^1&=\tfrac14(1-\mathcal H)e^0\wedge e^1,\\
 N\,de^2&=\tfrac12(v-1)e^0\wedge e^2
   +\tfrac12a\,e^1\wedge e^2+w\,e^0\wedge e^3+b\,e^1\wedge e^3,\\
 N\,de^3&=-\tfrac12(v+1)e^0\wedge e^3
                  -\tfrac12a\,e^1\wedge e^3.
\end{align*}
Differentiating these structure coefficients and using
\eqref{eq:revision-gowdy-system} and the Einstein constraints
expresses every $N^2R_{abcd}$ as a universal polynomial in
$(U,e^{-\tau}\partial_\theta U)$, of degree at most one in
$e^{-\tau}\partial_\theta U$, with bounded coefficients for bounded $U$.
No uncontrolled second derivatives remain.  Setting
$(a,w,b,e^{-\tau}\partial_\theta U)=(0,0,0,0)$
recovers the diagonal Kasner curvature for velocity $v$.
With
\[
 \kappa=(3+v^2)/4,\qquad
 (f_1,f_2,f_3)=((1-v^2)/4,(v-1)/2,-(v+1)/2),
\]
its components in the present convention satisfy, up to signs,
$N^2R_{0i0i}=-(\kappa f_i+f_i^2)$ and
$N^2R_{ijij}=f_if_j$; thus
\[
 N^4\mathcal K\big|_{(v,0,0,0)}
 =4\left(\sum_i(\kappa f_i+f_i^2)^2
       +\sum_{i<j}(f_if_j)^2\right)
 =\frac{(1-v^2)^2(3+v^2)}4.
\]
Since $U$ is bounded and
$(a,w,b,e^{-\tau}\partial_\theta U)\to0$ in $L^2$,
each polynomial term in the curvature contraction converges
in $L^1$.  Also $v\to\bar v$ in $L^1$, and the
Kasner polynomial is Lipschitz on $[-m,m]$.
This proves \eqref{eq:revision-K-L1}.
\end{proof}

\begin{proof}[Proof of Theorem~\ref{thm:revision-nonperturbative-K}]
The limiting function in \eqref{eq:revision-K-L1} is bounded below
by $3(1-m^2)^2/4$.  Hence for all sufficiently large $\tau$,
\[
 \sup_\theta N^4\mathcal K(\tau,\theta)
 \ge \frac1{2\pi}\int N^4\mathcal K(\tau,\theta)\,d\theta
 \ge\frac{3(1-m^2)^2}{8}.
\]
The constraint $\lambda_\tau=-\mathcal H\le0$ gives
$N^{-4}=e^{3\tau-\lambda}\ge
 e^{3\tau-\sup_\theta\lambda(\tau_0,\theta)}$.
At a maximizer of $N^4\mathcal K$ this proves
\eqref{eq:revision-K-lower}.  Relative openness follows
from continuity of \eqref{eq:revision-initial-functional} on
constraint-compatible smooth Cauchy data.
\end{proof}

\begin{cor}[Non-polarized data of unbounded high-frequency size]
\label{cor:revision-highfrequency}
Let $v_*\in(0,1)$ and take fixed $\tau_0$.
For all small enough $\varepsilon>0$ and every integer $n\ge1$,
\[
 P_0=v_*\tau_0+\frac{\varepsilon}{n}\cos(n\theta),\quad
 P_1=v_*,\quad Q_0=\frac{\varepsilon}{n}\sin(n\theta),\quad Q_1=0.
\]
These data obey the global momentum constraint, are genuinely
non-polarized (the initial map and its velocity are not tangent
to a single geodesic in $\mathbb H^2$), and satisfy
$\mathcal P_0>0$ independently of $n$.  For each $s>1$,
$\|(P_0,Q_0)\|_{H^s}\to\infty$ as $n\to\infty$.
\end{cor}
\begin{proof}
The momentum integral is $v_*\int P_0'=0$.
The characteristic squares are bounded by
$(v_*+C\varepsilon e^{-\tau_0})^2+
 C\varepsilon^2e^{2v_*\tau_0-2\tau_0}$,
which is smaller than one for $\varepsilon$ small independently of
$n$.  At generic $\theta$, the two target-space vectors
$u_\tau$ and $u_\theta$ are linearly independent, excluding any
common geodesic reduction.  Finally, the $n$th Fourier coefficient
has amplitude $\varepsilon/n$.
\end{proof}

\begin{cor}[Initial-data criterion for global $C^2$-inextendibility]
\label{dense:cor:initial}
The strict initial characteristic inequality \eqref{eq:revision-initial-functional},
$\mathcal P_0>0$, implies $C^2$-inextendibility of the whole MGHD.
\end{cor}
\begin{proof}
Lemma~\ref{lem:revision-characteristic} gives $|u_\tau|_h\le m<1$
for every $\tau\ge\tau_0$ and every $\theta$. Hence
$v_\infty(\theta)\le m<1$ and $Z_{\rm crit}=\varnothing$.
Theorem~\ref{dense:thm:main} applies. Unlike the direct quantitative
Kretschmann estimate of Theorem~\ref{thm:revision-nonperturbative-K},
this inference uses Ringstr\"om's curvature blow-up theorem.
\end{proof}
}

\section{{Predictive volume transfer and Gowdy end geometries}}\label{sec:predictive-transfer}
{
We now turn from the inextendibility criteria to the evolution of
normalized volume, and to the prediction of its localization from
Cauchy data. The exact orbital tube formula specifies when that
localization also corresponds to intrinsic metric tubes.
}

\subsection{{Exact transfer and orbital tube geometry}}\label{sec:gowdy-nonpolarized}
For any transported orbital region $S=I\times\mathbb T^2$ define
$p(\tau)=\omega_\tau(S)$ and set
$D(\tau)=\langle\mathcal H\rangle_{S^c,\tau}
-\langle\mathcal H\rangle_{S,\tau}$.
Since $\det(h_{ab})=t^2$ for the orbit metric, the volume form is
$tA\,d\theta\,dx\,dy$ and hence
\begin{equation}\label{eq:nonpolarized-transport}
 \frac{dp}{d\tau}=\frac14p(1-p)D(\tau),\qquad
 \log\frac{p(\tau)}{1-p(\tau)}
 =\log\frac{p(\tau_0)}{1-p(\tau_0)}
 +\frac14\int_{\tau_0}^{\tau}D(s)\,ds.
\end{equation}
Thus $p(\tau)\to1$ for a fixed nontrivial region if and only if
$\int_{\tau_0}^{\tau}D(s)\,ds\to+\infty$.
For a family of initially small orbital regions,
Proposition~\ref{prop:mass-threshold} applies with
$I_\delta=\frac14\int D_\delta$.

\begin{theorem}[Exact orbital tube formula without polarization]
\label{thm:nonpolarized-tubes}
Let $Z_{\theta_0}=\{\theta=\theta_0\}\simeq\mathbb T^2$
be any orbit torus in a smooth compact Gowdy slice. Its normal
injectivity radius $\rho_t$ and normalized tubular measure satisfy
\begin{equation}\label{eq:nonpolarized-tube-formula}
 \rho_t=\frac12\int_0^{2\pi}A(t,\theta)\,d\theta,
 \qquad
 \omega_t(T_{\gamma_t}(Z_{\theta_0},r))=\frac{r}{\rho_t}
 \quad(0<r<\rho_t).
\end{equation}
The mean curvature of each orbit torus \emph{inside the spatial
slice} vanishes. The full second fundamental form need not vanish.
There exist radii $r_t\downarrow0$ with tube probability tending
to one if and only if $\rho_t\to0$.
\end{theorem}
\begin{proof}
For the positive definite orbit matrix $h_{ab}(t,\theta)$,
$\det h=t^2$. Write $ds=A(t,\theta)\,d\theta$ on the transverse
circle. Then $\gamma_t=ds^2+h_{ab}(s)dx^a dx^b$ and
$dV_{\gamma_t}=t\,ds\,dx\,dy$.
Every path to $Z_{\theta_0}$ has length at least the length of its
projection to the $s$-circle, with equality for curves at fixed
$(x,y)$. The normal geodesics are these curves; the positive
definiteness of $h$ prevents a focal degeneracy before opposite
normal geodesics meet at half the circumference. The tube of radius
$r<\rho_t$ covers exactly a transverse arc of length $2r$ and
therefore has relative volume $2r/(2\rho_t)$.
The orbit mean curvature is proportional to
$\partial_s\log\sqrt{\det h}=\partial_s\log t=0$,
although $\partial_s h$ need not vanish.
{If $\rho_t\to0$, one may take a monotone majorant
$r_t=2\sup_{0<s\le t}\rho_s+t$, which tends to zero and covers the
whole slice;} if
$\limsup\rho_t>0$, a subsequence has positive normal radius,
and any $r_t\to0$ has vanishing tube probability there.
\end{proof}

\subsection{{A preserved chain from polarized vacuum Cauchy data}}\label{sec:gowdy-waves}

\noindent The following construction gives a global \emph{areal-time}
example, not an example of concentration along CMC leaves. Its inputs are
polarized vacuum Gowdy constraint data on $\T^3$, and its evolution is
controlled by a conserved spectral function constructed at $t=t_0$.
Existence of the distinct global CMC foliation is classical \cite{IM};
the areal and CMC volume distributions need not agree. 

Let $\theta,x,y$ be $2\pi$-periodic. A polarized Gowdy metric in the areal gauge is
\begin{equation}\label{eq:metric}
g=A(t,\theta)^2(-dt^2+d\theta^2)
+t\bigl(e^{U(t,\theta)}dx^2+e^{-U(t,\theta)}dy^2\bigr),\quad
A=t^{-1/4}e^{\lambda/4},
\end{equation}
with the future orientation toward decreasing $t>0$. The vacuum Einstein equations are equivalent to
\begin{equation}\label{eq:gowdy}
U_{tt}+t^{-1}U_t-U_{\theta\theta}=0,\quad
\lambda_t=t(U_t^2+U_\theta^2),\quad
\lambda_\theta=2tU_tU_\theta.
\end{equation}
Choose a Cauchy slice $t=t_0$, with positive $t_0$, and prescribe two real-valued, periodic analytic functions
\[
u(\theta)=U(t_0,\theta),\qquad r(\theta)=t_0U_t(t_0,\theta)
\]
subject to the \emph{initial momentum condition}
\begin{equation}\label{eq:momentum}
\int_0^{2\pi}r(\theta)u'(\theta)\,d\theta=0.
\end{equation}
Choose $\lambda_0$ with $\lambda_0'=2ru'$; it exists periodically by \eqref{eq:momentum}, and its additive constant is free. Put
\begin{align}\label{eq:initialmetric}
\gamma_0&=A_0^2d\theta^2+t_0e^u dx^2+t_0e^{-u}dy^2,\qquad A_0=t_0^{-1/4}e^{\lambda_0/4},\\
\lambda_{t,0}&=t_0^{-1}r^2+t_0(u')^2,\nonumber\\
K_0&=\frac{A_0}{4}(\lambda_{t,0}-t_0^{-1})d\theta^2
+\frac{e^u}{2A_0}(1+r)dx^2+\frac{e^{-u}}{2A_0}(1-r)dy^2.\label{eq:initialK}
\end{align}
These are \emph{genuine vacuum constraint data}. Indeed the solution of the linear equation for $U$ exists for $t>0$, the quantity $t\int U_tU_\theta d\theta$ is conserved, so the two equations for $\lambda$ are compatible and periodic. The resulting metric \eqref{eq:metric} is Ricci flat and induces precisely \eqref{eq:initialmetric}--\eqref{eq:initialK}; the Gauss--Codazzi equations therefore give
\begin{equation}\label{eq:constraints}
R_{\gamma_0}+(\tr_{\gamma_0}K_0)^2-|K_0|^2_{\gamma_0}=0,
\quad\nabla^i(K_{0ij}-\tr_{\gamma_0}K_0\,\gamma_{0ij})=0.
\end{equation}

\paragraph{A conserved function computed from the Cauchy data.}
For $n\in\mathbb Z$, let $u_n,r_n$ be the Fourier coefficients of $u,r$. Define
\begin{equation}\label{eq:vn}
v_n=J_0(|n|t_0)r_n+|n|t_0J_1(|n|t_0)u_n,
\qquad v(\theta)=\sum_{n\in\mathbb Z}v_ne^{in\theta}.
\end{equation}
Here $J_0,J_1$ are Bessel functions of the first kind, and for $n=0$ the formula reads $v_0=r_0$.
\begin{lemma}[Exact propagation of the initial-data profile]\label{lem:v}
The function $v$ in \eqref{eq:vn} depends only on $(\gamma_0,K_0)$ and the polarized Gowdy action. It is unchanged when computed on any later areal slice $t\in(0,t_0]$. For analytic data, there exist analytic functions $w,B$ such that, uniformly with any fixed finite number of $\theta$ derivatives,
\begin{align}\label{eq:asym}
U(t,\theta)&=v(\theta)\log t+w(\theta)+O(t^2(1+|\log t|)),\\
\lambda(t,\theta)&=v(\theta)^2\log t+B(\theta)
+O(t^2(1+|\log t|)^2),\label{eq:lambdaasym}\\
t\lambda_t(t,\theta)&=v(\theta)^2+O(t^2(1+|\log t|)^2).\label{eq:lambdat}
\end{align}
On sufficiently small fixed analytic neighborhoods of initial data, the remainder constants may be chosen uniformly.
\end{lemma}
\begin{proof}
The Fourier coefficient $U_n$ obeys
$U_n''+t^{-1}U_n'+n^2U_n=0$. The Bessel identities
$J_0'=-J_1$ and $J_1'=J_0-J_1/z$ imply
\[
\frac d{dt}\{t[J_0(|n|t)U_n'(t)+|n|J_1(|n|t)U_n(t)]\}=0.
\]
This conserved expression at $t_0$ is precisely $v_n$. For $n\ne0$ the general mode is
$\alpha_nJ_0(|n|t)+\beta_nY_0(|n|t)$; the Wronskian gives $v_n=2\beta_n/\pi$. The zero mode is $v_0\log t+w_0$. Since $Y_0(z)=\frac2\pi(\log(z/2)+\gamma_E)+O(z^2(1+|\log z|))$ and $J_0(z)=1+O(z^2)$, the first expansion follows modewise. Analytic Fourier decay permits uniform summation of the expansions and of any prescribed finite number of spatial derivatives: split the modes at $|n|t=1$, estimate the low modes by their Taylor remainders and the exponentially small high-frequency tail by Bessel bounds. The same argument gives a uniform estimate on a fixed analytic ball. Substituting $U_t=v/t+O(t(1+|\log t|))$ and $U_\theta=v'\log t+w'+O(t^2(1+|\log t|))$ into \eqref{eq:gowdy} gives \eqref{eq:lambdat}, and integration in $t$ gives \eqref{eq:lambdaasym} with a periodic $B$.
\end{proof}

\paragraph{A preserved chain, localized concentration, and a curvature singularity.}
We fix the material region
\[
S=\{(\theta,x,y):\cos\theta<0\}\subset\T^3.
\]
The reference family is given by constants $a,b,C$ with $b>0$, $a>1+2b/\pi$, and
\begin{align}\label{eq:refU}
U_*(t,\theta)&=a\log t+bY_0(t)\cos\theta,\\
\lambda_*(t,\theta)&=C+a^2\log t+\frac{b^2t^2}{2}(Y_0^2+Y_1^2)
+2abY_0\cos\theta-b^2tY_0Y_1\cos^2\theta.\label{eq:reflambda}
\end{align}
All Bessel functions in \eqref{eq:reflambda} have argument $t$.
\begin{theorem}[Initial-data chain, localization and curvature blow-up]\label{thm:chain}
Fix $a,b$ as above and an analytic norm with exponential Fourier weights. There exist $t_0>0$, $\varepsilon>0$ and a nonempty relatively open neighborhood $\mathcal U$ of the Cauchy data induced by \eqref{eq:refU}--\eqref{eq:reflambda} at $t=t_0$ in the polarized Gowdy analytic \emph{vacuum constraint class}, such that the following hold for every $(\gamma_0,K_0)\in\mathcal U$.
\begin{enumerate}
\item The data satisfy both vacuum constraints \eqref{eq:constraints}, and their development admits the areal Cauchy foliation $0<t\le t_0$ without assuming a future continuation hypothesis.
\item The conserved initial-data function $v$ given by \eqref{eq:vn} satisfies $v>1+\eta$ for some common $\eta>0$ and $v^2$ has a unique nondegenerate global minimum at $\theta_*=\theta_*(\gamma_0,K_0)\in(\pi/2,3\pi/2)$. Consequently $\mathcal P$ --- membership of the conserved spectral data in the specified open neighborhood, together with $t\le t_0$ --- is preserved on every areal slice.
\item If $\omega_t=dV_{\gamma_t}/\Vol_{\gamma_t}(\T^3)$ and $p(t)=\omega_t(S)$, then
\begin{equation}\label{eq:monotonic}
p'(t)<0\quad (0<t\le t_0),\qquad p(t)\longrightarrow1\quad(t\downarrow0).
\end{equation}
Thus the normalized volume of the material region increases at \emph{every future step}.
\item The concentration locus predicted by the data is the torus $Z=\{\theta=\theta_*\}\simeq\T^2$. With the natural areal identification of leaves,
\begin{equation}\label{eq:weak}
\omega_t\rightharpoonup\delta_{\theta_*}\otimes\frac{dx\,dy}{(2\pi)^2}.
\end{equation}
Writing $m=\min v^2$ and $\Delta=\min_{S^c}v^2-m>0$, for some $C_1,\zeta>0$,
\begin{equation}\label{eq:rates}
1-p(t)\le C_1\sqrt{|\log t|}\,t^{\Delta/4},\qquad
\sup_{q\in\Sigma_t}\dist_{\gamma_t}(q,Z)\le C_1t^\zeta\longrightarrow0.
\end{equation}
In particular there are intrinsic metric tubes about $Z$ whose radii tend to zero and whose normalized measures equal one.
\item The spacetime Kretschmann scalar diverges uniformly toward $t=0$: for some $c_*>0$ and sufficiently small $t$,
\begin{equation}\label{eq:curvature}
\Rm_{\alpha\beta\gamma\delta}\Rm^{\alpha\beta\gamma\delta}\ge c_*t^{-(m+3)}.
\end{equation}
Every future-directed timelike curve starting on $\Sigma_{t_0}$ has finite proper length inside the future development; in particular future-inextendible timelike geodesics are incomplete. The curvature blow-up prevents any $C^2$ extension through $t=0$.
\end{enumerate}
This is an open-set result \emph{only relative to the polarized Gowdy constraint class}.
\end{theorem}
\begin{proof}
The constraints and global existence in areal time follow from \eqref{eq:gowdy} and the initial construction. To ensure $\mathcal U$ is nonempty, compute $\lambda_*$ from \eqref{eq:reflambda} by direct differentiation, using $Y_0'=-Y_1$ and $Y_1'=Y_0-Y_1/t$; it satisfies both equations for $\lambda$. The reference profile is
\[
v_*(\theta)=a+\frac{2b}{\pi}\cos\theta.
\]
It is strictly larger than one and $v_*^2$ has a unique nondegenerate minimum at $\pi$, a positive distance from $S^c$. These properties persist under sufficiently small analytic perturbations. The Fourier--Bessel constants form a fixed sequence conserved by evolution, so an open condition on this sequence is preserved on each new Cauchy slice. The reference coefficients and their small neighborhoods satisfy the periodic momentum constraint whenever the data are restricted to the vacuum constraint class.

The volume form of \eqref{eq:metric} is $dV_{\gamma_t}=tA\,d\theta\,dx\,dy$, hence, after normalization, its $\theta$ density is proportional to $e^{\lambda(t,\theta)/4}$. Let $\mu_t$ denote this probability on the $\theta$-circle. For the fixed set $S$, differentiating its probability yields
\begin{equation}\label{eq:derivative}
\frac{d}{dt}p(t)=\frac14p(t)(1-p(t))
\left(\av{\lambda_t}{S,t}-\av{\lambda_t}{S^c,t}\right).
\end{equation}
By \eqref{eq:lambdaasym}, the $\theta$ density is
\begin{equation}\label{eq:weight}
\frac{e^{B(\theta)/4}t^{v(\theta)^2/4}(1+O(t^2(1+|\log t|)^2))}
{\int_0^{2\pi}e^{B(\varphi)/4}t^{v(\varphi)^2/4}(1+O(t^2(1+|\log t|)^2))\,d\varphi}.
\end{equation}
The nondegenerate minimum $\theta_*$ gives, by the elementary one-dimensional Laplace method,
\begin{equation}\label{eq:laplace}
\int_0^{2\pi}e^{\lambda/4}d\theta\asymp \frac{t^{m/4}}{\sqrt{|\log t|}},
\quad\av{v^2}{S,t}\longrightarrow m,
\quad\av{v^2}{S^c,t}\ge m+\Delta.
\end{equation}
The first two estimates hold uniformly on a sufficiently small analytic neighborhood of the reference, with a uniform positive $\Delta$. From \eqref{eq:lambdat} and \eqref{eq:laplace}, upon decreasing the common initial $t_0$ if necessary, we obtain
\[
\av{\lambda_t}{S^c,t}-\av{\lambda_t}{S,t}\ge\frac{\Delta}{2t}>0
\qquad(0<t\le t_0).
\]
Equation \eqref{eq:derivative} yields strict monotonicity. Since $v^2\ge m+\Delta$ on $S^c$, the numerator of $1-p(t)$ in \eqref{eq:weight} is $O(t^{(m+\Delta)/4})$, while the denominator is bounded below by a constant multiple of $t^{m/4}/\sqrt{|\log t|}$. This proves the first estimate in \eqref{eq:rates}, the weak convergence \eqref{eq:weak} and $p(t)\to1$.

The spatial distance to $Z$ is at most the length of an arc at fixed $(x,y)$, so
\[
\sup_{q\in\Sigma_t}\dist_{\gamma_t}(q,Z)\le 2\pi\sup_{\theta}A(t,\theta).
\]
By \eqref{eq:lambdaasym}, $\sup A\le Ct^{(m-1)/4}$, and $m>1$. Choose $\zeta=(m-1)/4$ to conclude the second part of \eqref{eq:rates}.

For the curvature claim, set
\[
\beta(\theta)=\frac{v(\theta)^2+3}{4},\qquad
\tau(t,\theta)=\int_0^t A(s,\theta)\,ds\sim\frac{e^{B(\theta)/4}}{\beta(\theta)}t^{\beta(\theta)}.
\]
The derivative bounds in Lemma~\ref{lem:v} show that orthonormal-frame curvature terms involving $\theta$ derivatives are smaller than the velocity-dominated terms by factors $O(t(1+|\log t|)^k)$. The leading-order time-dependent metric is pointwise Kasner, with exponents
\[
 p_\theta=\frac{v^2-1}{v^2+3},\quad
 p_x=\frac{2(1+v)}{v^2+3},\quad
 p_y=\frac{2(1-v)}{v^2+3},\quad
 \sum p_i=\sum p_i^2=1.
\]
The resulting uniform curvature expansion is
\begin{align}\label{eq:Kexpansion}
\Rm_{\alpha\beta\gamma\delta}\Rm^{\alpha\beta\gamma\delta}
&=4\tau^{-4}\!\left[\sum_i p_i^2(p_i-1)^2+\sum_{i<j}p_i^2p_j^2+o(1)\right].
\end{align}
This is consistent with the velocity-dominated Kasner asymptotics for polarized Gowdy; see \cite{RingstromGowdy}.
 Since $v>1+\eta$, we have $p_x>0$, $p_y<0$, and the bracket has a uniform strictly positive lower bound. Because $\tau\asymp t^{(v^2+3)/4}$, uniformly in $\theta$, \eqref{eq:curvature} follows. Finally, for each timelike curve with decreasing $t$, its proper-time element is no more than $A(t,\theta)|dt|\le Ct^\zeta|dt|$, and $\int_0^{t_0}t^\zeta dt<\infty$. Compact Cauchy slices and smoothness for each $t>0$ ensure any future-inextendible timelike geodesic approaches $t=0$, proving timelike incompleteness. Uniform scalar blow-up excludes a $C^2$ extension.
\end{proof}

\paragraph{Explicit example.}
For $a=3$ and $b=\pi/2$, one has $v_*(\theta)=3+\cos\theta$,
$\theta_*=\pi$, and $m=4$. With $S=\{\cos\theta<0\}$,
$\min_{S^c}v_*^2=9$, hence $\Delta=5$. The preceding estimates read
\[
1-\omega_t(S)\le C\sqrt{|\log t|}\,t^{5/4},\qquad
\sup_{q\in\Sigma_t}\dist_{\gamma_t}(q,Z)\le Ct^{3/4},
\qquad \Rm_{\alpha\beta\gamma\delta}\Rm^{\alpha\beta\gamma\delta}
\ge ct^{-7}.
\]
The constants depend on the initial data and the chosen $t_0$.

\subsection{{Conditional and asymptotic propagation for general Gowdy data}}
{
The polarized model has the conserved Fourier--Bessel profile of
Lemma~\ref{lem:v}. For general Gowdy data the wave-map nonlinearity
does not yield that explicit invariant. The exact redistribution and
tube identities continue to hold, but propagation needs additional
control.
}

For conditional averages with respect to the evolving normalized
volume one also has
\begin{equation}\label{eq:nonpolarized-D-derivative}
 \frac{dD}{d\tau}=
 \langle\partial_\tau\mathcal H\rangle_{S^c}
 -\langle\partial_\tau\mathcal H\rangle_S
 +\frac14\bigl(\operatorname{Var}_{S}(\mathcal H)
 -\operatorname{Var}_{S^c}(\mathcal H)\bigr).
\end{equation}
Neither the local energy balance nor the global monotonic energy
by itself makes the positive-gap condition $D>0$ invariant.

\begin{remark}[An initially positive gap may change sign]
Even in the polarized class, positivity of the transfer gap is not
an invariant consequence of the constraints. On any compact interval
of positive areal times consider the exact solution
$U(t,\theta)=a\log t+\epsilon J_0(t)\cos\theta$ with $a>0$
and $0<\epsilon\ll1$. Its periodic momentum constraint holds.
For $S=\{\cos\theta>0\}\times\mathbb T^2$, the energy and the gap
have the expansions, uniformly for $6\le t\le8$,
\[
 \mathcal H=a^2-2a\epsilon tJ_1(t)\cos\theta+O(\epsilon^2),
 \qquad
 D(t)=\frac{8a\epsilon}{\pi}tJ_1(t)+O(\epsilon^2).
\]
Since $J_1(8)>0$ and $J_1(6)<0$, the transfer changes sign as
$t$ decreases. Thus even the \emph{polarized} constraint class does
not preserve an arbitrary initially favorable gap. The monotonic
chain in Theorem~\ref{thm:chain} relies on its additional conserved
profile and spectral separation conditions.
\end{remark}

\begin{theorem}[Conditional propagation from an initial energy profile]
\label{thm:nonpolarized-energy-barrier}
Assume the smooth Gowdy evolution exists for all
$\tau\ge\tau_0$. Write
$f_0(\theta)=\mathcal H(\tau_0,\theta)$ and let
$S=I\times\mathbb T^2$ be an orbital region whose complement
has positive length. Suppose
\begin{equation}\label{eq:initial-energy-gap}
 m:=\min_{S^1}f_0<\min_{I^c}f_0=:m+\Delta,
 \qquad \Delta>0,
\end{equation}
and define the accumulated spatial-error bound
\begin{equation}\label{eq:nonpolarized-error-bound}
 \varepsilon:=2\int_{\tau_0}^{\infty}e^{-2s}
 \bigl(\|\partial_\theta\mathcal J(s)\|_{L^\infty}
       +\|\,|X_\theta(s)|_{\mathbb H}^2\,\|_{L^\infty}\bigr)ds.
\end{equation}
If $\varepsilon<\Delta/2$, then for some $c,C>0$ and all
sufficiently large $\tau$,
\begin{equation}\label{eq:nonpolarized-propagation-rate}
 1-p(\tau)\le C e^{-c(\tau-\tau_0)},\qquad p(\tau)\to1.
\end{equation}
If in addition $\varepsilon<\Delta/4$, then $D(\tau)>0$ for
all sufficiently large $\tau$. The assumptions do not by themselves
imply concentration in shrinking \emph{metric} tubes.
\end{theorem}
\begin{proof}
By \eqref{eq:gowdy-local-energy},
$\|\mathcal H(\tau)-f_0\|_\infty\le\varepsilon$.
Choose a compact interval $U\Subset I$ containing a minimizer
of $f_0$ and $\delta>0$ so small that
$\sup_U f_0\le m+\delta$ and
$\delta+2\varepsilon<\Delta$. On $I^c$ we have
$\mathcal H\ge m+\Delta-\varepsilon$, while on $U$ we have
$\mathcal H\le m+\delta+\varepsilon$. Integrating
$\lambda_\tau=-\mathcal H$ and comparing the integrals
of $e^{\lambda/4}$ gives
\[
 \frac{\int_{I^c}e^{\lambda(\tau,\theta)/4}\,d\theta}
      {\int_Ie^{\lambda(\tau,\theta)/4}\,d\theta}
 \le C\exp\!\left(-\frac{\Delta-\delta-2\varepsilon}{4}
           (\tau-\tau_0)\right),
\]
which proves the first assertion.
If $\varepsilon<\Delta/4$, choose
$2\varepsilon<\delta_1<\Delta-2\varepsilon$.
A second comparison with a still smaller neighborhood of the
minimizer shows that the conditional probability of
$\{f_0\ge m+\delta_1\}\cap I$ tends to zero.
Consequently
$\langle\mathcal H\rangle_I\le m+\delta_1+\varepsilon+o(1)$,
whereas $\langle\mathcal H\rangle_{I^c}\ge m+\Delta-\varepsilon$.
Thus $D(\tau)\ge\Delta-\delta_1-2\varepsilon+o(1)>0$.
\end{proof}

\begin{theorem}[Areal localization and the metric tube dichotomy]
\label{thm:nonpolarized-asymptotic-classification}
Suppose a (possibly non-polarized) smooth Gowdy solution satisfies,
uniformly on the transverse circle,
\begin{equation}\label{eq:gowdy-uniform-lambda-assumption}
 \lambda(\tau,\theta)=-\tau f(\theta)+B(\theta)+o(1)
 \qquad(\tau\to\infty),
\end{equation}
where $f\in C^2(S^1)$ and $B\in C(S^1)$, and $f$ has a
unique nondegenerate global minimum at $\theta_*$, of value
$m=f(\theta_*)$. Then the normalized areal marginal satisfies
$\mu_\tau\rightharpoonup\delta_{\theta_*}$.
For $\delta_\tau\downarrow0$ and
$I_{\delta_\tau}=\{|\theta-\theta_*|<\delta_\tau\}$,
\begin{equation}\label{eq:gowdy-gaussian-scale}
 \mu_\tau(I_{\delta_\tau})
 =\operatorname{erf}\!\left(
 \delta_\tau\sqrt{\frac{\tau f''(\theta_*)}{8}}\right)+o(1).
\end{equation}
In particular, $\mu_\tau(I_{\delta_\tau})\to1$ if and only if
$\delta_\tau\sqrt\tau\to\infty$.
Furthermore the orbital normal radius has the sharp asymptotic
\begin{equation}\label{eq:gowdy-general-rho}
 \rho_\tau\sim \frac{e^{B(\theta_*)/4}}{2}
 \sqrt{\frac{8\pi}{\tau f''(\theta_*)}}
 e^{(1-m)\tau/4}.
\end{equation}
Thus shrinking \emph{intrinsic} metric tubes around an orbit torus
carry asymptotically full normalized measure precisely when
$m\ge1$; at the critical value $m=1$ one has
$\rho_\tau\asymp\tau^{-1/2}\to0$.
For $m<1$ the transverse normal radius diverges, even though the
areal-coordinate measure localizes.
\end{theorem}
\begin{proof}
The areal marginal is proportional to $e^{\lambda/4}d\theta$.
Laplace's method at the nondegenerate minimum of $f$ gives a Gaussian
profile with variance $4/(\tau f''(\theta_*))$ and proves
\eqref{eq:gowdy-gaussian-scale}, including shrinking intervals.
Since $A=e^{(\tau+\lambda)/4}$, the same Laplace estimate for
$\rho_\tau=\tfrac12\int A\,d\theta$ proves
\eqref{eq:gowdy-general-rho}. The exact tube formula
\eqref{eq:nonpolarized-tube-formula} then gives the criterion.
\end{proof}

\begin{remark}[Curvature growth need not select the mass-attracting region]
In the explicit polarized reference family
$v(\theta)=a+c\cos\theta$ with $a-c>1$, the areal measure
localizes where $v$ is minimal, at $\theta=\pi$, while the leading
Kretschmann power is $t^{-(v(\theta)^2+3)}$ and is steeper
where $v$ is maximal. Hence even within one family of vacuum
solutions the point of fastest curvature growth need not be the
point selected by normalized-volume redistribution. An assertion
of a ``concentration singularity'' therefore requires a separate
metric--measure definition and cannot be inferred from the
Kretschmann blow-up alone.
\end{remark}

\begin{remark}[Comparison with generic Gowdy curvature singularities]
Ringstr\"om's theorem \cite{RingstromAnnals} proves curvature
blow-up and $C^2$-inextendibility for generic Gowdy data.
These are not consequences of our volume-transfer identity. On any fixed
regular interval where the asymptotic velocity obeys
$v_\infty^2\le1-\delta$ and the local expansion
$\lambda=-\tau v_\infty^2+B+o(\tau)$ holds uniformly,
$A=e^{(\tau+\lambda)/4}$ grows exponentially on that interval.
It follows that $\rho_\tau\to\infty$, so there is no shrinking
metric-tube concentration around a single orbit torus, despite
curvature blow-up. A separate problem is whether the areal marginal
localizes in the presence of spikes, where a uniform global
expansion of the form \eqref{eq:gowdy-uniform-lambda-assumption}
is not automatic. More generally, curvature blow-up, causal geodesic
incompleteness, coordinate localization, and intrinsic tubular
concentration are four distinct assertions.
\end{remark}

\subsection{{Non-polarized predictive criteria and the geometry of the orbits}}
\label{sec:gowdy-predictive-geometry}
The exact transfer law in Section~\ref{sec:gowdy-nonpolarized}
leaves two distinct questions: whether an initial strict transfer gap
can be preserved and whether the resulting spatial geometry collapses.
We give a Cauchy-data criterion for the first question in a relatively
open, genuinely non-polarized low-velocity class, and an explicit
high-velocity classification showing that the second question depends
also on the marking of the orbit lattice.  A point-collapse example
may coexist with spacetime curvature blow-up.  
\paragraph{An open initial-data set with no reversal.}
Fix
\begin{equation}\label{eq:new-low-parameters}
 0<c<a,\qquad a+c<1,\qquad 0<\alpha<\pi,
 \qquad S_\alpha=\{|\theta|<\alpha\}\times\mathbb T^2.
\end{equation}
Write $t=e^{-\tau}$ and put $b=\pi c/2$.  The explicitly
polarized reference wave is
\begin{equation}\label{eq:new-low-reference}
 P_*(\tau,\theta)=a\tau+bY_0(e^{-\tau})\cos\theta,
 \qquad Q_*=0,
\end{equation}
where $Y_0$ is the Bessel function of the second kind.  Its
asymptotic velocity is $v_*(\theta)=a-c\cos\theta$.
For initial data $\mathcal D=(P_0,P_1,Q_0,Q_1)$ on $\tau=\tau_0$,
we impose the Gowdy momentum constraint
\begin{equation}\label{eq:new-initial-momentum}
 \int_{S^1}\bigl(P_1P_0'+e^{2P_0}Q_1Q_0'\bigr)\,d\theta=0.
\end{equation}
The additive constant of $\lambda$ is immaterial for normalized
volume.  Let $\mathcal D_*$ denote the data of
\eqref{eq:new-low-reference} on $\tau_0$.

\begin{theorem}[Open Cauchy-data obstruction to transfer reversal]
\label{thm:new-predictive-low}
For the parameters in \eqref{eq:new-low-parameters}, there exist
$\tau_0>0$, an integer $k\ge6$, $\epsilon>0$ and $d>0$ such that
any smooth periodic Gowdy data satisfying
\eqref{eq:new-initial-momentum} and
\begin{equation}\label{eq:new-predictive-ball}
 \|P_0-P_*(\tau_0)\|_{H^k}
 +\|P_1-\partial_\tau P_*(\tau_0)\|_{H^{k-1}}
 +\|Q_0\|_{H^k}+\|Q_1\|_{H^{k-1}}<\epsilon
\end{equation}
produce a general (not necessarily polarized) Gowdy development on every
$\tau\ge\tau_0$ for which
\begin{equation}\label{eq:new-low-gap}
 D_{\mathcal D}(\tau)
 =\langle\mathcal H\rangle_{S_\alpha^c,\tau}
  -\langle\mathcal H\rangle_{S_\alpha,\tau}\ge d>0.
\end{equation}
Consequently $p_{\mathcal D}(\tau)=\omega_\tau(S_\alpha)$
is strictly increasing and satisfies
\begin{equation}\label{eq:new-low-prob-bound}
 \frac{1-p_{\mathcal D}(\tau)}{p_{\mathcal D}(\tau)}
 \le\frac{1-p_{\mathcal D}(\tau_0)}{p_{\mathcal D}(\tau_0)}
 e^{-d(\tau-\tau_0)/4},\qquad
 p_{\mathcal D}(\tau)\longrightarrow1.
\end{equation}
The set of such initial data is relatively open in the smooth
vacuum Gowdy constraint class; it includes waves with genuinely
nonconstant $Q$. The class includes the polarized reference, but no
polarization hypothesis is imposed on the perturbations.
\end{theorem}
\begin{proof}
Put $x=\cos\theta$ and $t=e^{-\tau}$.  For the reference,
\[
 P_{*,\tau}=a+btY_1(t)x,\qquad
 P_{*,\theta}=-bY_0(t)\sin\theta,
\]
so
\[
 \mathcal H_*=(a+btY_1(t)x)^2
                 +b^2t^2Y_0(t)^2(1-x^2).
\]
The Bessel limits $tY_1(t)\to-2/\pi$ and $tY_0(t)\to0$
give, uniformly in $-1\le x\le1$,
\begin{equation}\label{eq:new-reference-monotonicity}
 \partial_x\mathcal H_*
 \longrightarrow -2c(a-cx)\le-2c(a-c)<0.
\end{equation}
Choose $\tau_0$ so large that $\partial_x\mathcal H_*<0$
for all $\tau\ge\tau_0$.  On $S_\alpha$ one has
$x>\cos\alpha$, while on its complement $x\le\cos\alpha$.
The strict pointwise ordering implies $D_*(\tau)>0$ at every
such time, independently of the volume weights.

Moreover $\mathcal H_*\to f_*:=(a-c\cos\theta)^2$ uniformly,
and the constraint equation yields
$\lambda_*=-\tau f_*+B_*+o(1)$ uniformly.  Applying Laplace's
method to the two conditional measures gives
\begin{equation}\label{eq:new-reference-positive-limit}
 D_*(\tau)\longrightarrow
 \Delta_\alpha:=(a-c\cos\alpha)^2-(a-c)^2>0.
\end{equation}
As $D_*$ is continuous and positive on the finite initial
interval, it follows that $\inf_{\tau\ge\tau_0}D_*(\tau)>0$.

We invoke the \emph{initial-data} theorem of
Ringstr\"om~\cite[Theorem~9.1]{RingstromInitial2004} (numbering
as in arXiv:gr-qc/0204044). Choose $0<\gamma\le1/4$ with
$2\gamma<a-c<a+c<1-2\gamma$, and fix the initial slice
sufficiently late as follows. In spatial dimension one the energy
index $m_d$ in that theorem equals $2$. In its notation the relevant
initial quantities, for $j=0,1,2$, are
\begin{align*}
 \mathcal E_j&=\tfrac12\int_{S^1}
 \bigl[(\partial_\theta^j P_\tau)^2+
 e^{-2\tau}(\partial_\theta^{j+1}P)^2\bigr]d\theta,\\
 F_j&=\tfrac12\int_{S^1}(\partial_\theta^{j+1}P)^2d\theta,\\
 E_j&=\tfrac12\int_{S^1}e^{2P}
 \bigl[(\partial_\theta^j Q_\tau)^2+
 e^{-2\tau}(\partial_\theta^{j+1}Q)^2\bigr]d\theta.
\end{align*}
For the polarized reference, $E_j=0$, the $\mathcal E_j$
are bounded as $\tau_0\to\infty$, and $F_j^{1/2}=O(1+\tau_0)$.
Thus the one-dimensional condition (9.1) of that theorem is
strict for sufficiently large $\tau_0$, and the velocity
condition (9.4) has a strict margin. Fix this $\tau_0$.
Choose common upper bounds $\epsilon_j,\nu_j$ for
$\mathcal E_j^{1/2}(\tau_0),F_j^{1/2}(\tau_0)$, preserving
(9.1), and common positive upper bounds $e_j$ for
$E_j^{1/2}(\tau_0)$ sufficiently small for the smallness
condition (9.3). This is possible because the reference
$Q$-energies vanish. By Sobolev embedding and continuity of
these finitely many initial norms, a sufficiently small ball
in \eqref{eq:new-predictive-ball} has these same bounds and
satisfies (9.4), together with the imposed momentum constraint.

For clarity, the uniformity needed below follows from the
energy estimates in the proof of Ringstr\"om's theorem rather
than merely from its pointwise-in-data asymptotic conclusion.
With $s=\tau-\tau_0$, the bounds (6.3), (7.1), (7.2) and
Lemma~9.1 in \cite{RingstromInitial2004} give, with constants
controlled by the fixed initial energy bounds and $\gamma$,
\begin{align*}
 \mathcal E_j^{1/2}(\tau)&\le\epsilon_j+1,
       &&j=0,1,2,\\
 \|e^P Q_\tau\|_{C^0}+\|e^{P-\tau}Q_\theta\|_{C^0}
       &\le C(1+s)^N e^{-\gamma s}.
\end{align*}
The latter polynomials in (7.1)--(7.2) can be bounded by a
single polynomial with nonnegative coefficients using the
common energy upper bounds. In particular these estimates
require no bound on arbitrarily high derivatives of the data.
The first estimate and the one-dimensional Sobolev embedding
yield $\|P_\tau\|_{C^1}\le C$, and integration of
$P_{\theta\tau}$ gives $\|P_\theta\|_{C^0}\le C(1+s)$.
The same energy estimate bounds
$\|e^{-\tau}P_{\theta\theta}\|_{H^1}\le C$; hence the
$P$ equation implies
\[
 \|P_{\tau\tau}\|_{C^0}\le
 C e^{-\tau}+C(1+s)^{2N}e^{-2\gamma s}.
\]
Integrating this common integrable bound from $\tau$ to
infinity gives a uniform convergence rate of $P_\tau$ to
$v_{\mathcal D}$, where
$\gamma\le v_{\mathcal D}\le1-\gamma$ by Theorem~9.1.
Combining it with the bounds on the three remaining terms in
$\mathcal H$ gives
\begin{equation}\label{eq:new-ringstrom-uniform}
 \mathcal H_{\mathcal D}(\tau,\theta)
 =v_{\mathcal D}(\theta)^2+R_{\mathcal D}(\tau,\theta),
 \quad
 \|R_{\mathcal D}(\tau)\|_{C^0}
 \le C(1+\tau-\tau_0)^N e^{-\kappa(\tau-\tau_0)}
\end{equation}
with common $C,N,\kappa>0$ (enlarging $C,N$ or decreasing
$\kappa$ if necessary). Moreover, comparing $P_\tau$ for the
reference and the perturbed data at a large fixed time by
finite-time continuous dependence, and then using the common
tail bound, proves continuity of $\mathcal D\mapsto v_{\mathcal D}$
in $C^0$ at the reference. In particular
$\|v_{\mathcal D}-v_*\|_{C^0}$ can be made uniformly small.
After fixing the additive constant of $\lambda$ on the initial
slice, its initial $C^0$ norm is uniformly bounded by the
momentum constraint and the initial Sobolev bounds. Integrating
$\lambda_\tau=-\mathcal H$ gives
\[
 \lambda_{\mathcal D}(\tau,\theta)
 =-(\tau-\tau_0)v_{\mathcal D}(\theta)^2
   +B_{\mathcal D}(\theta)+O(e^{-\kappa'(\tau-\tau_0)}),
\]
where
$B_{\mathcal D}=\lambda_{\mathcal D}(\tau_0,\cdot)
-\int_{\tau_0}^{\infty}R_{\mathcal D}(t,\cdot)\,dt$
is uniformly bounded, and $0<\kappa'<\kappa$ absorbs the
polynomial factor. All these bounds hold on the same
initial-data neighborhood, independently of the particular
solution.

Choose nested intervals $U_0\Subset U_1\Subset S_\alpha$
around $\theta=0$ so that
\[
 \sup_{U_1}f_*<\inf_{S_\alpha^c}f_* ,\qquad
 \sup_{U_0}f_*<\inf_{S_\alpha\setminus U_1}f_*.
\]
These strict separations persist for $v_{\mathcal D}^2$.
Comparison of $e^{\lambda_{\mathcal D}/4}$ over these sets
shows that the conditional mass inside $S_\alpha$
concentrates on $U_1$, uniformly within the initial-data
neighborhood.  Thus \eqref{eq:new-ringstrom-uniform} implies
$D_{\mathcal D}(\tau)\ge d_1>0$ for all $\tau\ge T$, for
one $T$ and $d_1$ independent of the data in that neighborhood.
On the compact interval $[\tau_0,T]$, continuous dependence
of the evolution and of the volume densities gives
$\sup_{[\tau_0,T]}|D_{\mathcal D}-D_*|<
\tfrac12\min_{[\tau_0,T]}D_*$ after shrinking the neighborhood.
Together the estimates give \eqref{eq:new-low-gap}; integrating
\eqref{eq:nonpolarized-transport} proves
\eqref{eq:new-low-prob-bound}.

Finally, non-polarized data occur inside this neighborhood:
for instance, $Q_0=\delta\sin\theta$, $Q_1=0$,
$P_0=P_*(\tau_0)$ and $P_1=P_{*,\tau}(\tau_0)$ satisfy the
periodic constraint, and are genuinely nonconstant in $Q$ when
$0<|\delta|\ll1$.  At every point with $Q_0'\ne0$, the vectors
$u_\tau=(P_1,0)$ and $u_\theta=(P_0',Q_0')$ are linearly
independent, since $P_1\ne0$. Thus the wave-map image cannot
lie in a fixed target geodesic; the example is genuinely
non-polarized even up to a fixed target isometry.
\end{proof}

\paragraph{Measure localization does not decide the metric limit.}
For a Gowdy spatial slice, write
\begin{equation}\label{eq:new-orbit-metric}
 \gamma_t=A(t,\theta)^2d\theta^2+h_{ab}(t,\theta)dx^a dx^b,
 \qquad
 h=t\begin{pmatrix}e^P&e^PQ\\e^PQ&e^{-P}+e^PQ^2\end{pmatrix},
 \qquad\det h=t^2.
\end{equation}
Let $L_\theta(t)=\int_{S^1}A(t,\theta)\,d\theta$.
The orbit projection to the circle endowed with arclength
$ds=A\,d\theta$ is $1$-Lipschitz, and the projection of the
normalized spatial measure is \emph{exactly} normalized
arclength on that circle.

\begin{theorem}[The low-velocity metric--measure limit]
\label{thm:new-low-mgh}
For every development in a sufficiently small initial-data
neighborhood of Theorem~\ref{thm:new-predictive-low},
\begin{equation}\label{eq:new-low-L-growth}
 L_\theta(t)\longrightarrow\infty,\qquad
 \sup_\theta\operatorname{diam}(\mathbb T^2,h(t,\theta))
 \longrightarrow0.
\end{equation}
Consequently $\operatorname{diam}(\mathbb T^3,\gamma_t)
\to\infty$, whereas the rescaled metric--measure spaces
\begin{equation}\label{eq:new-low-mgh-statement}
 \bigl(\mathbb T^3,L_\theta(t)^{-1}d_{\gamma_t},\omega_t\bigr)
 \ \longrightarrow\ (S^1_{\mathrm{length}=1},d_{S^1},ds)
\end{equation}
in measured Gromov--Hausdorff topology.  Thus monotone
transfer and areal localization coexist with \emph{transverse
expansion}, not unrescaled collapse to a point.
\end{theorem}
\begin{proof}
Ringstr\"om's asymptotics in the stated neighborhood give,
uniformly, $P=v(\theta)\tau+O(1)$, with
$0<v_-\le v\le v_+<1$, bounded $Q$, and
$\lambda=-v^2\tau+O(1)$.  Hence
\[
 A=\exp\bigl((1-v^2)\tau/4+O(1)\bigr)\to\infty
\]
uniformly, so $L_\theta\to\infty$.  The orbit metric has
coordinate lengths bounded by constants times
$\sqrt t\,e^{P/2}$ and $\sqrt t\,e^{-P/2}$, both tending
uniformly to zero because $v_+<1$ and $v_->0$.
The orbit diameters thus tend uniformly to zero.
The projection to the transverse circle is a metric
approximation with distortion at most twice the maximal orbit
diameter: a minimizing base path can be lifted with constant
$(x,y)$, followed by a path in one fiber.  After rescaling by
$L_\theta^{-1}$ the base circle has length one, and the orbit
diameters vanish.  The pushforward measure is exactly
normalized arclength since $dV_{\gamma_t}=t\,ds\,dx\,dy$.
This proves the measured convergence.  The unrescaled
diameter is at least $L_\theta/2$ by the same $1$-Lipschitz
projection, and therefore diverges.
\end{proof}

\paragraph{An arithmetic diameter lemma for the orbit tori.}
The toroidal orbit metric can be written as
\begin{equation}\label{eq:new-flat-torus-lattice}
 h=R^2(dx+q\,dy)^2+\eta^2dy^2,
 \qquad R=\sqrt t\,e^{P/2},\quad
 \eta=\sqrt t\,e^{-P/2},\quad R\eta=t.
\end{equation}
Here $x,y$ have period $2\pi$.  The marking, and not merely
the determinant $\det h=t^2$, determines the diameter.

\begin{lemma}[Bounded-type and rational slopes]
\label{lem:new-arithmetic-lattice}
Let $P\to+\infty$ along a family of orbit tori of the form
\eqref{eq:new-flat-torus-lattice}.
\begin{enumerate}
\item If $q$ is a fixed irrational number with bounded
continued-fraction coefficients, then
$\operatorname{diam}(\mathbb T^2,h)\le C_q\sqrt t$.
\item If $q=r/s$ in lowest terms, then
$\operatorname{diam}(\mathbb T^2,h)\ge\pi R/s$.
\item The bound in (1) remains true if the actual slope
$Q_t$ obeys $|Q_t-q|\le C e^{-P}$ at the orbit under
consideration, with $q$ as in (1).
\end{enumerate}
\end{lemma}
\begin{proof}
The orbit is the Euclidean quotient by the lattice
\[
 2\pi\{(R(m+qn),\eta n):m,n\in\mathbb Z\}.
\]
Set $H=e^{P/2}$; then $R=\sqrt t\,H$ and
$\eta=\sqrt t\,H^{-1}$.  If $q$ has bounded partial quotients,
its consecutive convergent denominators $k_j,k_{j-1}$ can
be chosen so that $k_{j-1}\asymp k_j\asymp H$ as $H\to\infty$.
The convergents have determinant $\pm1$ and satisfy
$|qk_j-p_j|\le k_j^{-1}$, with analogous estimates for
$j-1$.  The corresponding lattice vectors
\[
 2\pi\bigl(R(-p_j+qk_j),\eta k_j\bigr)
\]
and the predecessor therefore form a basis of length
$O_q(\sqrt t)$.  The covering radius, and hence the
quotient diameter, is at most the sum of the lengths of
these two basis vectors.  This proves (1).
If $q=r/s$, the function
$\chi=sx+ry\pmod{2\pi}$ is well-defined and
$h\ge(R^2/s^2)d\chi^2$.  A circle of length $2\pi R/s$
is then a $1$-Lipschitz quotient, so the orbit diameter is
at least $\pi R/s$, giving (2).
For (3), the first component of each convergent basis
vector changes by at most
$R|Q_t-q|k_j\le C\sqrt t\,H\cdot H^{-2}\cdot H=O(\sqrt t)$;
the same holds for the preceding vector.  The basis remains
unimodular, and the diameter estimate follows.
\end{proof}

\begin{theorem}[Identical transfer, different global metric limits]
\label{thm:new-high-arithmetic}
Fix $a>1+c>1$ with $c>0$, take $b=\pi c/2$, and
choose the reference wave
\begin{equation}\label{eq:new-high-reference}
 P_*(\tau,\theta)=a\tau+bY_0(e^{-\tau})\cos\theta,
 \qquad Q_*(\tau,\theta)=q\quad\text{constant}.
\end{equation}
Let $S_\alpha=\{|\theta|<\alpha\}\times\mathbb T^2$,
$0<\alpha<\pi$.  There exists $\tau_0$ such that for all
$\tau\ge\tau_0$ and every constant $q$:
\begin{equation}\label{eq:new-high-monotone}
 p_q'(\tau)>0,\qquad p_q(\tau)\to1,\qquad
 \omega_\tau\rightharpoonup\delta_{\theta=0}
 \otimes\frac{dx\,dy}{(2\pi)^2},\qquad
 L_\theta(t)\to0.
\end{equation}
The normalized measures, the transfer gaps, and the intrinsic
local curvature expressions are \emph{identical} for every $q$.
Nevertheless the global metric types differ:
\begin{enumerate}
\item If $q=r/s$ is rational in lowest terms, then
$\operatorname{diam}(\mathbb T^3,\gamma_t)\to\infty$.
\item If $q$ is irrational with bounded continued-fraction
coefficients, then
\begin{equation}\label{eq:new-high-point-collapse}
 \operatorname{diam}(\mathbb T^3,\gamma_t)
 \le L_\theta(t)+C_q\sqrt t\to0.
\end{equation}
Thus the compact spatial metric--measure spaces converge to
one point in Gromov--Hausdorff and measured
Gromov--Hausdorff topology.
\end{enumerate}
In both cases the spacetime Kretschmann scalar blows up
uniformly and prevents a $C^2$ extension across $t=0$.
\end{theorem}
\begin{proof}
Because $Q=q$ is constant, the Einstein wave equation reduces
exactly to the polarized equation; its $\lambda$ equation is
independent of $q$.  The momentum constraint holds by parity.
As in the proof of Theorem~\ref{thm:new-predictive-low},
$\partial_{\cos\theta}\mathcal H_*<0$ for all sufficiently
large $\tau$, now since $a>c>0$.
Thus the transfer toward $S_\alpha$ is strictly monotone.
Its asymptotic profile is
$v_*(\theta)=a-c\cos\theta>1$, with unique minimum at
$\theta=0$.  Uniform Bessel asymptotics imply
$\lambda=-v_*^2\tau+B+o(1)$, giving the areal measure
localization and $L_\theta\to0$ by Laplace's method.

For rational $q=r/s$, the map
$\chi=sx+ry\pmod{2\pi}$ is defined over the \emph{whole}
spatial torus.  With $R(t,\theta)=\sqrt t\,e^{P_*/2}$,
$\gamma_t\ge (\min_\theta R(t,\theta)^2/s^2)d\chi^2$.
Since $\min_\theta v_*(\theta)>1$,
$\min_\theta R(t,\theta)\to\infty$, so
$\operatorname{diam}(\Sigma_t)\ge
\pi\min_\theta R(t,\theta)/s\to\infty$.

For bounded-type irrational $q$, Lemma~\ref{lem:new-arithmetic-lattice}
provides $\operatorname{diam}(Z_{\theta=0})\le C_q\sqrt t$.
Any two points of the spatial torus can be joined to that orbit
by transverse curves of total length at most $L_\theta(t)$,
followed by a path inside the orbit.  Thus
$\operatorname{diam}(\Sigma_t)\le L_\theta(t)+C_q\sqrt t\to0$.
Since $Q=q$ can be eliminated by the local transformation
$x\mapsto x+qy$, the local spacetime curvature is exactly that
of the polarized reference solution.  The curvature expansion of
Theorem~\ref{thm:chain}, applied to the same Bessel wave with its
uniform speed $v_*>1$, therefore gives uniform Kretschmann blow-up
and $C^2$-inextendibility, independently of $q$.
\end{proof}

\begin{remark}
\label{rem:new-arithmetic-not-open}
The condition that the fixed slope $q$ have bounded partial
quotients is \emph{not open} in $\mathbb R$.  Indeed rational
slopes occur arbitrarily close to it, and the two metric limits
in Theorem~\ref{thm:new-high-arithmetic} differ.  

\end{remark}

\paragraph{Genuinely non-polarized point-collapse examples.}
A further conclusion is possible without asserting an open
high-velocity Cauchy-data neighborhood.
\begin{cor}[High-velocity Fuchsian examples]
\label{cor:new-fuchsian-high}
Let $v(\theta)=a-c\cos\theta>1$ be analytic with $a>1+c$,
and take a constant bounded-type irrational $q$.  In the
high-velocity Fuchsian class with $q$ constant, choose
analytic asymptotic data $\phi(\theta)=0$ and an odd,
nonconstant $\psi(\theta)$, for instance
$\psi(\theta)=\delta\sin\theta$ with $\delta\ne0$.
The corresponding sufficiently regular Gowdy solution has
\begin{equation}\label{eq:new-high-fuchsian}
 {P=v\tau+o(1)},\qquad
 Q=q+e^{-2v\tau}(\psi+o(1)),\qquad
 \lambda=-v^2\tau+B+o(1),
\end{equation}
with the derivative control supplied by the Fuchsian
construction~\cite{KichenassamyRendall1998,RendallFuchsian2000}.
There is a regular late Cauchy slice $\tau_0$ such that its
future development satisfies $D(\tau)>0$ for every
$\tau\ge\tau_0$, $p(\tau)\to1$, and
$\operatorname{diam}(\Sigma_t)\to0$.
{The solution is genuinely non-polarized, in the sense that the wave-map image is not contained in a single target geodesic.}
\end{cor}
\begin{proof}
The high-velocity Fuchsian existence result applies when the
limiting $Q$-value $q$ is constant; this removes the leading
spatial $e^{2P-2\tau}Q_\theta^2$ obstruction to high-velocity
expansions.  The evenness of $v,\phi$ and oddness of $\psi$
about $\theta=0$ imply, by uniqueness of the Fuchsian
solution, $P(\tau,-\theta)=P(\tau,\theta)$ and
$Q(\tau,-\theta)-q=-(Q(\tau,\theta)-q)$.
Consequently the momentum density $\mathcal J$ is odd,
the global momentum constraint holds, and
$Q(\tau,0)=q$ for all sufficiently large $\tau$.
The uniform expansions give
$\mathcal H\to v^2$, $\lambda=-v^2\tau+B+o(1)$,
and thus $D(\tau)\to
(a-c\cos\alpha)^2-(a-c)^2>0$ by Laplace's method.
Choose the late initial slice beyond the time where $D>0$.
Since $\min v>1$, $L_\theta\to0$, and since
$Q(\tau,0)=q$ and $P(\tau,0)\to+\infty$, Lemma~
\ref{lem:new-arithmetic-lattice} gives the diameter of the
orbit at $\theta=0$ bounded by $C_q\sqrt t$.
The global diameter consequently tends to zero as in
\eqref{eq:new-high-point-collapse}.  {Finally, in the upper half-plane coordinates $(Q,e^{-P})$,
a geodesic ending at the ideal point $q$ has a fixed finite limit of
$(Q-q)e^{2P}$ along approaches to $q$ (zero on the vertical geodesic).
Here that limit is $\psi(\theta)$ since $\phi=0$. Its nonconstancy
excludes containment of the wave-map image in a single geodesic.}
\end{proof}

\paragraph{Intermediate forms and regularity obstructions.}
The simplest Kasner family makes the possible unrescaled and
rescaled limits explicit.  For constant $v>0$ and $Q=q$,
\begin{equation}\label{eq:new-kasner-family}
 \gamma_\tau=e^{(1-v^2)\tau/2}d\theta^2
 +e^{(v-1)\tau}(dx+q\,dy)^2
 +e^{-(v+1)\tau}dy^2.
\end{equation}
For $0<v<1$ both orbital directions collapse, while the
transverse circle expands; after normalization by its length,
the limit is a circle.  At $v=1$, the transverse circle has
constant length: for $q=r/s$ the orbit collapses to a circle
and the spatial limit is a flat two-torus, whereas for
bounded-type irrational $q$ the orbit collapses to a point
and the spatial limit is the transverse circle.  For $v>1$
a rational slope produces unbounded diameter (and a circle
after rescaling by the long orbital scale), while a
bounded-type irrational slope gives an unrescaled point limit.
These assertions follow by direct inspection of
\eqref{eq:new-kasner-family} and
Lemma~\ref{lem:new-arithmetic-lattice}; the uniform spatial
volume density provides the corresponding measured limits.
They are \emph{model geometries}, not a Perelman-type universal
classification of all Gowdy singularities.

There is also an exact geometric diagnostic valid without
polarization.  On a fixed areal leaf write
$\gamma=ds^2+h_{ab}(s)dx^a dx^b$, where $ds=A\,d\theta$.
The orbit shape operator is $B=\tfrac12h^{-1}\partial_sh$.
Since $\det h=t^2$ does not depend on $s$, $\operatorname{tr}B=0$,
and a direct two-by-two calculation gives the principal
curvatures and spatial scalar curvature
\begin{equation}\label{eq:new-orbit-curvatures}
 \kappa_{\pm}=\pm\frac1{2A}
 \sqrt{P_\theta^2+e^{2P}Q_\theta^2},
 \qquad
 R_{\gamma_t}=-\frac1{2A^2}
 (P_\theta^2+e^{2P}Q_\theta^2)\le0.
\end{equation}
Thus every orbital torus is minimal in its areal leaf, but
need not be totally geodesic.  These identities provide
geometric quantities whose growth can be followed before
collapse.  For every $t>0$ the constructed spatial metrics
remain smooth; loss of differentiability is not implied by
metric collapse.  In the high-velocity Bessel examples of
Theorem~\ref{thm:new-high-arithmetic}, the independent
Kretschmann blow-up supplies an obstruction to any Lorentzian
$C^2$ extension, but not by itself to an extension of lower
regularity.

\section{{Observable limits, spectral elimination and pyramids}}
\label{sec:gowdy-observable-pyramids}
We now regard the evolving areal Cauchy slices as metric--measure spaces
rather than simply using the fixed areal identification to compare
probability densities.  The distinction is essential: weak localization
in the coordinate $\theta$ need not be concentration to a point in
Gromov's observable metric.  Our use of pyramids and infinite
dissipation follows Shioya~\cite{Shi}; the geometric collapse literature
provides further context~\cite{CheegerFukayaGromov1992,
BettiolDerdzinskiPiccione2018}.

Let
\[
 X_\tau=(\mathbb T^3,d_{\gamma_\tau},\omega_\tau),\qquad
 B_\tau=(\mathbb S^1,ds_\tau^2),\qquad
 ds_\tau=A(\tau,\theta)\,d\theta,
\]
with normalized Riemannian volume $\omega_\tau$, and put
\[
 L(\tau)=\int_{\mathbb S^1}A(\tau,\theta)\,d\theta,
 \quad
 F_\tau(\theta)=\operatorname{diam}
   (\mathbb T^2_\theta,h(\tau,\theta)).
\]
The normalized measure on $B_\tau$ is $m_\tau=ds_\tau/L(\tau)$.

\begin{lemma}[Exact transverse observables and length evolution]
\label{lem:gowdy-base-projection}
For every smooth polarized or non-polarized Gowdy areal leaf, the orbit
projection $\pi_\tau:X_\tau\to B_\tau$ is $1$-Lipschitz and
\begin{equation}\label{eq:new-projection-exact}
 (\pi_\tau)_\#\omega_\tau=m_\tau,
 \qquad
 \frac{d}{d\tau}\log L(\tau)
 =\frac14\bigl(1-\langle\mathcal H\rangle_{\omega_\tau}\bigr).
\end{equation}
\end{lemma}
\begin{proof}
The Riemannian volume is
$dV_{\gamma_\tau}=e^{-\tau}A\,d\theta\,dx\,dy$,
since $\det h=e^{-2\tau}$ and each torus has coordinate area
$(2\pi)^2$.  This gives the asserted exact pushforward and the
$1$-Lipschitz statement follows from the orthogonal splitting
$\gamma_\tau=ds_\tau^2+h$.  Finally,
$\partial_\tau\log A=(1+\lambda_\tau)/4=(1-\mathcal H)/4$;
integration against $A\,d\theta$ proves the second identity.
\end{proof}

\subsection{{Hyperbolic escape and an initial-data expansion criterion}}
\paragraph{Uniform control outside a small-probability set, including spikes.}
A consequence of the Einstein constraint $\lambda_\tau=-\mathcal H$
will be used below to control all toroidal Fourier modes without any
pointwise asymptotic expansion at exceptional spike points.  Write
$h_\tau(\theta)$ for the $2\times2$ metric of the orbit $\mathbb T^2$
and $L(\tau)=\int_{S^1}A(\tau,\theta)\,d\theta$.  Fix an initial
areal slice $\tau_0$ and let $L_0=L(\tau_0)$.

\begin{lemma}[Hyperbolic escape versus negligible probability]
\label{lem:gowdy-hyperbolic-escape}
For every smooth polarized or non-polarized Gowdy evolution on
$\tau\ge\tau_0$, put $T=\tau-\tau_0$ and
\[
 r_\tau(\theta)=d_{\mathbb H^2}
 \bigl((P,Q)(\tau,\theta),(P,Q)(\tau_0,\theta)\bigr),
 \qquad E_\tau=\{\theta:r_\tau(\theta)>T\}.
\]
Then the following assertions hold without excluding true or false
spikes:
\begin{align}
 h_\tau(\theta)&\preceq h_{\tau_0}(\theta)
       &&(\theta\notin E_\tau),
       \label{eq:gowdy-spike-good-orbits}\\
 \omega_\tau(E_\tau\times\mathbb T^2)
       &\le\frac{L_0}{L(\tau)}.
       \label{eq:gowdy-spike-exceptional-measure}
\end{align}
Here $\preceq$ denotes the ordering of positive definite matrices
in the fixed marked toroidal coordinates $(x,y)$.
\end{lemma}
\begin{proof}
Write $h_\tau=e^{-\tau}M(P,Q)$, where
\[
 M(P,Q)=\begin{pmatrix}e^P&e^PQ\\
 e^PQ&e^{-P}+e^PQ^2\end{pmatrix},\qquad \det M=1.
\]
The upper half-plane metric $dP^2+e^{2P}dQ^2$ identifies
$\mathbb H^2$ with the positive definite symmetric $2\times2$
matrices of determinant one, with
\[
 \cosh d_{\mathbb H^2}(X,Y)
 =\tfrac12\operatorname{tr}\bigl(M(X)^{-1}M(Y)\bigr).
\]
The eigenvalues of $M(X)^{-1/2}M(Y)M(X)^{-1/2}$ are
$e^{\pm d_{\mathbb H^2}(X,Y)}$; hence
$M(Y)\preceq e^{d_{\mathbb H^2}(X,Y)}M(X)$.
By Cauchy--Schwarz along the wave-map trajectory at fixed $\theta$,
\[
 r_\tau(\theta)^2
 \le T\int_{\tau_0}^{\tau}|X_s|_{\mathbb H^2}^2\,ds
 \le T\int_{\tau_0}^{\tau}\mathcal H(s,\theta)\,ds
 =T\bigl(\lambda(\tau_0,\theta)-\lambda(\tau,\theta)\bigr).
\]
If $r_\tau\le T$, then
$h_\tau=e^{-\tau}M(\tau)\preceq
 e^{-\tau+r_\tau}M(\tau_0)\preceq h_{\tau_0}$.
If $r_\tau>T$, the displayed energy estimate gives
$\lambda(\tau,\theta)<\lambda(\tau_0,\theta)-T$, so
\[
 A(\tau,\theta)=e^{(\tau+\lambda(\tau,\theta))/4}
 \le e^{(\tau_0+\lambda(\tau_0,\theta))/4}=A(\tau_0,\theta).
\]
Because $\omega_\tau$ projects to $A(\tau,\theta)d\theta/L(\tau)$,
\[
 \omega_\tau(E_\tau\times\mathbb T^2)
 =\frac1{L(\tau)}\int_{E_\tau}A(\tau,\theta)\,d\theta
 \le\frac1{L(\tau)}\int_{S^1}A(\tau_0,\theta)\,d\theta
 =\frac{L_0}{L(\tau)}.
\]
No asymptotic regularity or spike-width estimate is used.
\end{proof}

\begin{prop}[Explicit initial-energy condition for transverse expansion]
\label{prop:gowdy-initial-energy-expansion}
Let the smooth periodic Gowdy Cauchy data at $\tau_0$ satisfy the
vacuum momentum constraint and define
\begin{equation}\label{eq:gowdy-initial-energy-condition}
 \overline{\mathcal H}_0
 :=\frac1{2\pi}\int_0^{2\pi}\bigl[
 P_1^2+e^{2P_0}Q_1^2+e^{-2\tau_0}
 ((\partial_\theta P_0)^2+e^{2P_0}(\partial_\theta Q_0)^2)
 \bigr]d\theta<1.
\end{equation}
Then the associated global areal Gowdy evolution satisfies
\[
 L(\tau)\ge 2\pi\exp\left[
 \frac{\tau+\overline\lambda_0
 -(\tau-\tau_0)\overline{\mathcal H}_0}{4}\right]
 \longrightarrow+\infty,
\]
where $\overline\lambda_0$ is the coordinate mean of the initial
periodic $\lambda$.  The strict inequality
\eqref{eq:gowdy-initial-energy-condition} is relatively open in the
smooth Gowdy vacuum constraint data.  No polarization or absence of
spikes is assumed.
\end{prop}
\begin{proof}
The integrated wave-map energy identity
\eqref{eq:gowdy-global-energy} gives
$\int_{S^1}\mathcal H(s,\theta)d\theta\le
 2\pi\overline{\mathcal H}_0$ for every $s\ge\tau_0$.
Using $\lambda_\tau=-\mathcal H$ and Jensen's inequality,
\begin{align*}
 L(\tau)
 &=e^{\tau/4}\int_{S^1}e^{\lambda(\tau,\theta)/4}\,d\theta\\
 &\ge 2\pi\exp\bigl((\tau+\overline\lambda(\tau))/4\bigr)\\
 &\ge 2\pi\exp\bigl((\tau+\overline\lambda_0
     -(\tau-\tau_0)\overline{\mathcal H}_0)/4\bigr).
\end{align*}
Since $\overline{\mathcal H}_0<1$, the last expression diverges
exponentially.  Openness follows from continuity of the integral
in the indicated Cauchy-data topology.
\end{proof}
\subsection{{Infinite dissipation and observable convergence}}
Recall that $\operatorname{Sep}(X;\kappa_0,\ldots,\kappa_N)$
is the supremum of the minimum pairwise distances among Borel
subsets of $X$ carrying probabilities at least
$\kappa_0,\ldots,\kappa_N$.  A family is \emph{infinitely
dissipating} if this quantity tends to infinity for every
finite tuple of strictly positive masses with sum less than one.
Write $\mathcal P_X$ for the pyramid associated with an mm-space
and $\mathcal X$ for the maximal pyramid of all mm-spaces.

\begin{theorem}[Transverse expansion forces infinite dissipation]
\label{thm:gowdy-infinite-dissipation}
If an arbitrary smooth Gowdy development satisfies
$L(\tau)\to+\infty$ as $\tau\to\infty$, then
$(X_\tau)_\tau$ is infinitely dissipating.  Consequently,
\begin{equation}\label{eq:gowdy-max-pyramid}
 \mathcal P_{X_\tau}\longrightarrow\mathcal X
 \quad\text{in the weak topology of pyramids}.
\end{equation}
In particular, the unscaled family cannot converge in observable
distance $d_{\mathrm{conc}}$ to an ordinary mm-space.
\end{theorem}
\begin{proof}
Fix $\kappa_0,\ldots,\kappa_N>0$ with sum smaller than one.
The circle $B_\tau$ of length $L(\tau)$ contains $N+1$
disjoint arcs with normalized arclength masses $\kappa_i$,
separated by gaps each bounded below by
$c_{\kappa,N}L(\tau)$ for a constant $c_{\kappa,N}>0$.
Their inverse images under $\pi_\tau$ have exactly the
prescribed masses by~\eqref{eq:new-projection-exact} and have
pairwise distances bounded below by the corresponding base
distances because $\pi_\tau$ is $1$-Lipschitz.  Thus
\[
 \operatorname{Sep}(X_\tau;\kappa_0,\ldots,\kappa_N)
 \ge c_{\kappa,N}L(\tau)\to+\infty.
\]
The pyramid conclusion follows from the equivalence between
infinite dissipation and convergence to the maximal pyramid
\cite[Proposition~8.5]{Shi}.  An ordinary observable limit
would have a principal-pyramid limit instead, which is impossible.
\end{proof}

\begin{cor}[Generic non-polarized Gowdy: a maximal pyramid]
\label{cor:gowdy-generic-pyramid}
Consider a Gowdy solution possessing a nonempty regular spatial
interval $I$ on which the asymptotic expansions hold uniformly and
\[
 0<v_\infty(\theta)^2\le 1-\delta
 \qquad(\theta\in I)
\]
for some $\delta>0$.  Then~\eqref{eq:gowdy-max-pyramid}
holds.  In particular, it applies to the generic non-polarized
$T^3$-Gowdy class of Ringstr\"om
\cite{RingstromAnnals,RingstromGowdy}.
\end{cor}
\begin{proof}
On the regular interval, the known asymptotics give
$\lambda(\tau,\theta)=-v_\infty(\theta)^2\tau+O(1)$
uniformly.  Hence $A(\tau,\theta)\ge
c_Ie^{\delta\tau/4}$ there and
$L(\tau)\ge c'_Ie^{\delta\tau/4}\to\infty$.
The generic class admits such a regular interval with
$0<v_\infty<1$; the previous theorem applies.  The argument
uses no uniform control of the metric at exceptional spike points.
\end{proof}

\begin{prop}[A probabilistic criterion for observable collapse]
\label{prop:gowdy-conc-probabilistic}
Let $X_\tau$ be any smooth Gowdy areal family and set
\[
 \widehat X_\tau:=
 (\mathbb T^3,L(\tau)^{-1}d_{\gamma_\tau},\omega_\tau),
 \qquad E_\tau(\theta):=F_\tau(\theta)/L(\tau).
\]
If, for every $\varepsilon>0$,
\begin{equation}\label{eq:gowdy-fiber-probability}
 m_\tau\{\theta:E_\tau(\theta)>\varepsilon\}
 \longrightarrow0,
\end{equation}
then
\begin{equation}\label{eq:gowdy-observable-circle}
 d_{\mathrm{conc}}
 \bigl(\widehat X_\tau,
 (\mathbb S^1_{\mathrm{length}=1},d_{\mathrm{arc}},m_{\mathrm{Leb}})
 \bigr)\longrightarrow0.
\end{equation}
No uniform bound on the sectional curvature or on
$\sup_\theta F_\tau(\theta)/L(\tau)$ is assumed.
\end{prop}
\begin{proof}
Use normalized arclength $u=s/L(\tau)$ on the base circle.
For any real-valued $1$-Lipschitz observable $f$ on
$\widehat X_\tau$, define the fibrewise average
\[
 \overline f(u)=\frac1{(2\pi)^2}
 \int_{\mathbb T^2}f(u,x,y)\,dx\,dy.
\]
The horizontal curves $(s,x,y)$ with fixed $(x,y)$ have
length equal to their base length, so $\overline f$ is
$1$-Lipschitz on the unit circle.  On each fibre,
\[
 |f(u,x,y)-\overline f(u)|\le E_\tau(\theta(u)).
\]
By~\eqref{eq:new-projection-exact}, the right-hand side
tends to zero in probability with respect to $\omega_\tau$.
This yields a uniform Ky Fan error, simultaneously for all
$1$-Lipschitz $f$.  Conversely, every $1$-Lipschitz observable
on the base lifts through $\pi_\tau$ to a $1$-Lipschitz
observable of $\widehat X_\tau$, with exactly the same
probability distribution.  The definition of Gromov's
observable distance as Hausdorff distance between pulled-back
$1$-Lipschitz observables in the convergence-in-measure metric
then proves~\eqref{eq:gowdy-observable-circle}; see~\cite{Shi}.
\end{proof}

\subsection{{Fourier reduction and elimination of non-invariant modes}}
\begin{lemma}[Fourier reduction of the normalized Gowdy Laplacian]
\label{lem:gowdy-fourier-schrodinger}
Let $\widehat\gamma_\tau=L(\tau)^{-2}\gamma_\tau$ and use
normalized transverse arclength
\[
 u=\frac{1}{L(\tau)}\int_0^\theta A(\tau,\xi)\,d\xi
 \pmod{1}.
\]
For each $\tau$, the normalized volume is
$du\,dx\,dy/(2\pi)^2$, and the orthogonal Fourier
expansion in $L^2(\widehat X_\tau)$ is
\[
 f(u,x,y)=\sum_{(k,\ell)\in\mathbb Z^2}
 f_{k,\ell}(u)e^{i(kx+\ell y)}.
\]
The nonnegative Laplace--Beltrami operator $-\Delta_{\widehat\gamma_\tau}$
preserves every toroidal Fourier sector, and on that sector is
unitarily equivalent to the periodic one-dimensional operator
\begin{equation}\label{eq:gowdy-schrodinger-sector}
 \mathscr H_{k,\ell,\tau}
 =-\frac{d^2}{du^2}+V_{k,\ell,\tau}(u),\qquad
 V_{k,\ell,\tau}(u)=L(\tau)^2e^\tau
 \left[e^{-P}k^2+e^P(\ell-Qk)^2\right],
\end{equation}
where $P$ and $Q$ are evaluated at $(\tau,\theta(u))$.
In particular $\mathscr H_{0,0,\tau}=-d^2/du^2$.
\end{lemma}
\begin{proof}
The inverse of the orbit metric
$h=e^{-\tau}\bigl(\begin{smallmatrix}
 e^P&e^P Q\\ e^P Q&e^{-P}+e^P Q^2
\end{smallmatrix}\bigr)$ is
\[
 h^{-1}=e^\tau\begin{pmatrix}
 e^{-P}+e^P Q^2&-e^P Q\\
 -e^P Q&e^P
 \end{pmatrix}.
\]
In the coordinate $u$, the normalized spatial metric is
$du^2+L^{-2}h_{ab}\,dx^a dx^b$.
Because $\det h=e^{-2\tau}$ is independent of $u$,
its volume density is constant, and no first-order transverse
term occurs in the scalar Laplacian.  Applying $-\Delta$ to
$f_{k,\ell}(u)e^{i(kx+\ell y)}$ gives precisely
\eqref{eq:gowdy-schrodinger-sector}.  The Fourier modes are
orthogonal for the normalized product measure.

\end{proof}

\subsection{{Rescaling trichotomy and persistent observables}}
\begin{theorem}[Spike-robust spectral extinction and rescaling trichotomy]
\label{thm:gowdy-spike-rescaling-trichotomy}
Let a smooth non-polarized or polarized $T^3$-Gowdy vacuum
solution be defined for all $\tau\ge\tau_0$.  Suppose that
$L(\tau)\to+\infty$, and denote by $\widehat X_\tau$ the
$L(\tau)^{-1}$-rescaling of the areal spatial mm-space.
There are constants $c>0$ and $\tau_1$, depending only on the
initial slice, such that for $\tau\ge\tau_1$:
\begin{equation}\label{eq:gowdy-spike-spectral-lower}
 \boxed{\Lambda_\tau\ge cL(\tau)^2.}
\end{equation}
Let $F_0=\sup_\theta\operatorname{diam}
(\mathbb T^2,h_{\tau_0}(\theta))$ and $L_0=L(\tau_0)$.
Then the following stronger geometric estimate also holds:
\begin{equation}\label{eq:gowdy-spike-conc-rate}
 \boxed{d_{\mathrm{conc}}(\widehat X_\tau,\mathbb S^1_1)
 \le\frac{\max\{F_0,L_0\}}{L(\tau)}.}
\end{equation}
More generally, let $a_\tau>0$ be any family of distance
multipliers and suppose that $a_\tau L(\tau)\to\ell$ in the
extended half-line $[0,+\infty]$.  Then:
\begin{enumerate}
 \item if $\ell=0$, the $a_\tau$-rescaled spatial mm-spaces
 concentrate in $d_{\mathrm{conc}}$ to the one-point space;
 \item if $0<\ell<\infty$, they concentrate in $d_{\mathrm{conc}}$
 to the circle of circumference $\ell$ with uniform measure;
 \item if $\ell=+\infty$, their pyramids converge weakly to
 the maximal pyramid $\mathcal X$.
\end{enumerate}
All statements hold without uniform metric or curvature estimates
at exceptional true or false spikes.  The classification concerns
observable limits in the transversely expanding regime only.
\end{theorem}
\begin{proof}
Use the good set $G_\tau=S^1\setminus E_\tau$ from
Lemma~\ref{lem:gowdy-hyperbolic-escape}.  On $G_\tau$ the
marked orbit metric satisfies $h_\tau\preceq h_{\tau_0}$,
so its intrinsic diameter is at most $F_0$, while
$\omega_\tau(E_\tau\times\mathbb T^2)\le L_0/L(\tau)$.
For each $1$-Lipschitz observable $f$ of $\widehat X_\tau$,
its fibre average $\bar f$ is $1$-Lipschitz on the unit
transverse circle, and
$|f-\bar f\circ\pi_\tau|\le F_0/L(\tau)$ on $G_\tau$.
The exceptional probability estimate yields the Ky Fan bound
\eqref{eq:gowdy-spike-conc-rate}.  The converse comparison
holds because every base observable lifts isometrically in the
Lipschitz order.  For a general scale $a_\tau$ the same argument gives
\begin{equation}\label{eq:gowdy-rescaled-ky-fan}
 d_{\mathrm{conc}}(a_\tau X_\tau,\mathbb S^1_{a_\tau L(\tau)})
 \le \max\{a_\tau F_0,L_0/L(\tau)\}.
\end{equation}
If $a_\tau L(\tau)\to\ell<\infty$, then
$a_\tau\to0$ and the right-hand side tends to zero.
The uniform circles with lengths $a_\tau L(\tau)$ converge
in $d_{\mathrm{conc}}$ to the circle of length $\ell$,
interpreted as a point when $\ell=0$.  If the lengths diverge,
the measure-preserving $1$-Lipschitz orbit projection onto
these circles and the separation argument in
Theorem~\ref{thm:gowdy-infinite-dissipation} show that the
pyramids tend to $\mathcal X$.

It remains to establish the \emph{uniform} spectral bound,
which is not a direct consequence of the small probability of
$E_\tau$.  In the normalized arclength coordinate $u$ the
probability of the exceptional set is $\varepsilon_\tau\le L_0/L$.
Set
\[
 \beta=\min_{\theta\in S^1}
 \lambda_{\min}(h_{\tau_0}(\theta)^{-1})>0.
\]
On $G_\tau$, positive matrix inversion and
$h_\tau\preceq h_{\tau_0}$ imply
$h_\tau^{-1}\succeq h_{\tau_0}^{-1}$.
Consequently, for every $(k,\ell)\in\mathbb Z^2\setminus\{0\}$,
\[
 V_{k,\ell,\tau}(u)
 =L(\tau)^2(k,\ell)h_\tau^{-1}(k,\ell)^T
 \ge \beta L(\tau)^2\qquad (u\notin E_\tau).
\]
For a periodic $H^1$ function $\varphi$ with
$\|\varphi\|_2=1$ and $D=\|\varphi'\|_2^2$,
the one-dimensional Sobolev inequality gives
\[
 \|\varphi\|_\infty^2\le1+2\sqrt D,
 \qquad
 \int_{E_\tau}|\varphi|^2\,du
 \le\frac{L_0}{L(\tau)}(1+2\sqrt D).
\]
Choose
$\eta=\min\{\beta/2,1/[64(1+L_0^2)]\}>0$
and take $\tau$ large enough that $L(\tau)\ge4L_0$.
If $D\ge\eta L(\tau)^2$, the Rayleigh quotient of
$\mathscr H_{k,\ell,\tau}$ is at least $\eta L(\tau)^2$.
Otherwise the preceding inequality gives
\[
 \int_{E_\tau}|\varphi|^2\,du
 \le\tfrac14+2L_0\sqrt\eta\le\tfrac12.
\]
The potential term in the Rayleigh quotient is therefore at
least $\frac\beta2 L(\tau)^2\ge\eta L(\tau)^2$.
By the min--max principle this estimate is uniform over
all nonzero toroidal indices, and proves
\eqref{eq:gowdy-spike-spectral-lower} with $c=\eta$.
\end{proof}

\begin{cor}[Complete classification of persistent toroidal sectors]
\label{cor:gowdy-spike-persistent-modes}
Under the assumptions of
Theorem~\ref{thm:gowdy-spike-rescaling-trichotomy},
identify $L^2(\widehat X_\tau)$ with the fixed Hilbert
space $L^2(S^1\times\mathbb T^2,du\,dx\,dy/(2\pi)^2)$
using normalized transverse arclength.  Write $\Pi_0$ for
the orthogonal projection onto functions independent of $(x,y)$
and $H_\tau=-\Delta_{\widehat\gamma_\tau}$.
Then the entire invariant sector has eigenvalues
$\{(2\pi j)^2:j\in\mathbb Z\}$, independent of $\tau$, while
all non-invariant sectors have spectra contained in
$[cL(\tau)^2,\infty)$ for large $\tau$.  In particular, for
$s>0$ and $z>0$,
\begin{align}
 \big\|e^{-sH_\tau}-e^{s\partial_u^2}\Pi_0\big\|_{2\to2}
 &\le e^{-scL(\tau)^2},
 \label{eq:gowdy-spike-heat-gap}\\
 \big\|(H_\tau+z)^{-1}
 -(-\partial_u^2+z)^{-1}\Pi_0\big\|_{2\to2}
 &\le\frac1{cL(\tau)^2+z}.
 \label{eq:gowdy-spike-resolvent-gap}
\end{align}
Thus, at the scale $L(\tau)^{-1}$, \emph{exactly} the
zero-toroidal Fourier modes persist at bounded spectral energy;
no exceptional mode remains because of a true spike.
\end{cor}
\begin{proof}
The direct-sum Fourier decomposition in
Lemma~\ref{lem:gowdy-fourier-schrodinger} makes
$H_\tau=-\partial_u^2$ exactly on the range of $\Pi_0$.
Theorem~\ref{thm:gowdy-spike-rescaling-trichotomy} bounds the
spectrum of its orthogonal complement below by
$cL(\tau)^2$.  Spectral calculus proves both inequalities.
\end{proof}

\begin{cor}[Cauchy-data conditions and Ringstr\"om true spikes]
\label{cor:gowdy-spike-cauchy-data}
The conclusions of
Theorem~\ref{thm:gowdy-spike-rescaling-trichotomy} and
Corollary~\ref{cor:gowdy-spike-persistent-modes} hold under either
of the following independent sufficient hypotheses:
\begin{enumerate}
 \item the explicit strict Cauchy-data energy inequality
 \eqref{eq:gowdy-initial-energy-condition};
 \item the Cauchy data lie in Ringstr\"om's generic $T^3$-Gowdy
 class, with its usual regular low-velocity interval and possibly
 finitely many nondegenerate true or false spikes
 \cite{RingstromAnnals,RingstromGowdy}.
\end{enumerate}
The first hypothesis is directly testable on one initial slice;
the second uses Ringstr\"om's generic initial-data theorem.
No initial-data condition excluding spikes is imposed in either
spectral argument.  
\end{cor}
\begin{proof}
The first condition implies $L(\tau)\to\infty$ by
Proposition~\ref{prop:gowdy-initial-energy-expansion}.
For the second, Ringstr\"om's estimates supply a nonempty
regular interval $I\subset S^1$ and $\delta>0$ on which
$\lambda(\tau,\theta)=-v_\infty(\theta)^2\tau+o(\tau)$
uniformly with $v_\infty(\theta)^2\le1-\delta$.
Thus $A=e^{(\tau+\lambda)/4}$ grows exponentially on $I$,
and $L(\tau)\to\infty$.  The two cited results apply
without requiring any estimate at the spikes.
\end{proof}

\subsection{{Additional quantitative spectral criteria}}
\begin{prop}[Spectral suppression of non-invariant observables]
\label{prop:gowdy-spectral-concentration}
For an arbitrary smooth polarized or non-polarized Gowdy areal
family, define
\begin{equation}\label{eq:gowdy-spectral-gap-definition}
 \Lambda_\tau=
 \inf_{(k,\ell)\in\mathbb Z^2\setminus\{(0,0)\}}
 \inf\operatorname{Spec}(\mathscr H_{k,\ell,\tau}).
\end{equation}
If $\Lambda_\tau\to+\infty$ as $\tau\to\infty$, then
\begin{equation}\label{eq:gowdy-spectral-conc}
 d_{\mathrm{conc}}\bigl(\widehat X_\tau,
 (\mathbb S^1_{\mathrm{length}=1},d_{\mathrm{arc}},m_{\mathrm{Leb}})
 \bigr)\longrightarrow0.
\end{equation}
More quantitatively, if $\Lambda_\tau\ge1$, the definition
of $d_{\mathrm{conc}}$ via the Ky Fan metric yields
\begin{equation}\label{eq:gowdy-spectral-rate}
 d_{\mathrm{conc}}(\widehat X_\tau,\mathbb S^1_1)
 \le \Lambda_\tau^{-1/3}.
\end{equation}

\end{prop}
\begin{proof}
Let $f$ be any real $1$-Lipschitz function on
$\widehat X_\tau$ and let
\[
 \overline f(u)=\frac{1}{(2\pi)^2}
 \int_{\mathbb T^2}f(u,x,y)\,dx\,dy
\]
be its zero toroidal Fourier mode.  The straight horizontal
curves at fixed $(x,y)$ show that $\overline f$ is
$1$-Lipschitz on the unit transverse circle.  By the
orthogonal sector decomposition, the variational definition
of $\Lambda_\tau$, and Parseval's identity for the
Dirichlet form,
\begin{align*}
 \Lambda_\tau\|f-\overline f\|_{L^2(\omega_\tau)}^2
 &\le\sum_{(k,\ell)\ne(0,0)}
  \big\langle f_{k,\ell},
   \mathscr H_{k,\ell,\tau}f_{k,\ell}\big\rangle_{L^2(du)}\\
 &\le\int_{\widehat X_\tau}|\nabla f|^2\,d\omega_\tau
 \le1.
\end{align*}
The calculation extends from smooth functions to Lipschitz
functions by density in the Dirichlet-form domain.  Thus
$\|f-\overline f\|_2^2\le\Lambda_\tau^{-1}$.
Chebyshev's inequality gives
\[
 \omega_\tau\{|f-\overline f\circ\pi_\tau|>\varepsilon\}
 \le\frac{1}{\Lambda_\tau\varepsilon^2}.
\]
For $\varepsilon=\Lambda_\tau^{-1/3}$, the right-hand side
is at most $\varepsilon$.  This is a Ky Fan estimate uniform
in the $1$-Lipschitz observable $f$.  Conversely, every
$1$-Lipschitz observable of the unit base circle lifts
through $\pi_\tau$ to a $1$-Lipschitz observable on
$\widehat X_\tau$.  Using a common measure parametrization
for $\omega_\tau$ and its base pushforward, the Hausdorff
comparison defining $d_{\mathrm{conc}}$ proves
\eqref{eq:gowdy-spectral-rate} and hence
\eqref{eq:gowdy-spectral-conc}.
\end{proof}

\begin{cor}[A geometric bound for the Fourier spectral gap]
\label{cor:gowdy-gap-fibre-diameter}
With $F(\tau)=\sup_\theta F_\tau(\theta)$, we have
\[
 \Lambda_\tau\ge\pi^2\left(\frac{L(\tau)}{F(\tau)}\right)^2.
\]
In particular, the uniform fibre-collapse criterion
$F(\tau)/L(\tau)\to0$ implies the spectral condition
\eqref{eq:gowdy-spectral-gap-definition}, and therefore
observable convergence to the unit circle.
\end{cor}
\begin{proof}
On each rescaled flat orbit torus, a nonconstant character
$e^{i(kx+\ell y)}$ has eigenvalue $V_{k,\ell,\tau}(u)$.
The corresponding angular coordinate defines a surjective
$1$-Lipschitz map onto a circle of circumference
$2\pi/V_{k,\ell,\tau}(u)^{1/2}$, with the rescaled
metric on the target.  Thus the diameter $F_\tau(\theta)/L$
of the rescaled orbit is at least
$\pi/V_{k,\ell,\tau}(u)^{1/2}$; consequently
\[
 V_{k,\ell,\tau}(u)\ge
 \pi^2\left(\frac{L(\tau)}{F(\tau)}\right)^2
 \qquad((k,\ell)\ne(0,0)).
\]
Taking the infimum of the quadratic forms of
$-d^2/du^2+V_{k,\ell,\tau}$ proves the claim.
\end{proof}

\begin{remark}
\label{rem:gowdy-spectral-spikes-scope}
Theorem~\ref{thm:gowdy-spike-rescaling-trichotomy} now verifies the
spectral hypothesis directly from the initial geometry whenever the
transverse length diverges; in particular it requires no local
uniform control at true spikes.  The heat-semigroup extinction
\eqref{eq:gowdy-spike-heat-gap} is analogous to the use of spectral
regularization and observable reconstruction in the Gaussian
Fock-space approach of
\cite{CacciatoriShioyaUrsinoYokotaGaussian}.  
\end{remark}

\begin{cor}[One Cauchy-data criterion, two different limits]
\label{cor:gowdy-predictive-two-limits}
For every Gowdy development from the relatively open Cauchy-data
set of Theorem~\ref{thm:new-predictive-low}, the selected region
$S_\alpha$ satisfies $\omega_\tau(S_\alpha)\to1$ with strictly
positive transfer derivative at every $\tau\ge\tau_0$.  At the same
time,
\begin{equation}\label{eq:gowdy-two-limits}
 \mathcal P_{X_\tau}\longrightarrow\mathcal X,
 \qquad
 d_{\mathrm{conc}}(\widehat X_\tau,\mathbb S^1_1)
 \longrightarrow0.
\end{equation}
\end{cor}
\begin{proof}
Theorem~\ref{thm:new-predictive-low} gives the monotone transfer.
Theorem~\ref{thm:new-low-mgh} gives $L(\tau)\to\infty$ and
uniform collapse of the fibre diameters; applying
Theorem~\ref{thm:gowdy-infinite-dissipation} and
Proposition~\ref{prop:gowdy-conc-probabilistic}
proves the two limits.
\end{proof}

\begin{cor}[Opposite pyramid limits at identical transfer]
\label{cor:gowdy-arithmetic-pyramids}
In the explicit high-velocity constant-slope family of
Theorem~\ref{thm:new-high-arithmetic}, rational $q$ gives
$\mathcal P_{X_\tau}\to\mathcal X$, while bounded-type
irrational $q$ gives
$d_{\mathrm{conc}}(X_\tau,\{\mathrm{point}\})\to0$
and hence convergence to the minimal one-point pyramid.
\end{cor}
\begin{proof}
For bounded-type irrational $q$, the theorem proves that the
full diameter tends to zero.  For $q=r/s$ rational, the map
$\chi=sx+ry\pmod{2\pi}$ defines a measure-preserving
$1$-Lipschitz quotient to a circle equipped with its
induced angular metric scaled by
$\min_\theta R(\tau,\theta)/s$.
Its circumference tends to infinity, so the separation
argument from Theorem~\ref{thm:gowdy-infinite-dissipation}
applies without a transverse expansion assumption.
\end{proof}

{
\section{{Constrained least-action principles for the Gowdy flow}}
\label{sec:revision-gowdy-variational}
For the wave map $u=(P,Q)$ into
$(\mathbb H^2,h=dP^2+e^{2P}dQ^2)$ set
\[
 G=|u_\theta|_h^2,\qquad
 \mathcal H=|u_\tau|_h^2+e^{-2\tau}G,\qquad
 j_{\mathcal H}=-2e^{-2\tau}\langle u_\tau,u_\theta\rangle_h,
 \qquad \rho=\frac{A}{L}.
\]
The reduced equations and the Einstein constraints yield
\begin{equation}\label{eq:revision-current-balance}
 \partial_\tau\mathcal H+\partial_\theta j_{\mathcal H}
    =-2e^{-2\tau}G,\qquad
 j_{\mathcal H}=4e^{-2\tau}\partial_\theta\log\rho,
 \qquad
 \partial_\tau\rho=-\frac14\rho(\mathcal H-\langle\mathcal H\rangle_\rho).
\end{equation}
The reduced energy $\mathcal H$ is a gauge-dependent wave-map energy,
not a covariant local gravitational energy density.

\begin{theorem}[Hierarchical transfer and exact global-arrest criterion]
\label{thm:revision-maximal-gap}
For fixed measurable orbital regions $S'\subsetneq S$, assume
$0<\omega_\tau(S')<\omega_\tau(S)$ (equivalently, both
$S'$ and $S\setminus S'$ have positive angular Lebesgue measure,
since $\rho_\tau>0$). Define
\[
 p_S=\omega_\tau(S),\quad
 q_{S'\mid S}=\frac{\omega_\tau(S')}{\omega_\tau(S)},\quad
 D_{S'\mid S}=\langle\mathcal H\rangle_{S\setminus S'}
              -\langle\mathcal H\rangle_{S'}.
\]
Then
\begin{equation}\label{eq:revision-hierarchical}
 \dot q_{S'\mid S}=\tfrac14 q_{S'\mid S}(1-q_{S'\mid S})D_{S'\mid S}.
\end{equation}
At each fixed time let $\mathfrak D(\tau)$ be the supremum of
$D_I=\langle\mathcal H\rangle_{I^c}-\langle\mathcal H\rangle_I$
over proper orbital arcs $I\subset S^1$ and their complements.
For continuous $\mathcal H$ and positive $\rho$,
\begin{equation}\label{eq:revision-maxgap-bound}
 \tfrac12\operatorname{osc}_{S^1}\mathcal H
 \le \mathfrak D(\tau)\le\operatorname{osc}_{S^1}\mathcal H,
 \qquad
 \mathfrak D(\tau)=0\iff \dot\rho_\tau=0.
\end{equation}
The maximal instantaneous regional transfer is
\begin{equation}\label{eq:revision-maxflux}
 \sup_{S\ \mathrm{measurable}}\dot p_S(\tau)
 =\frac18\int_{S^1}\left|
   \mathcal H-\langle\mathcal H\rangle_\rho\right|\rho\,d\theta.
\end{equation}
These are instantaneous statements: they assert neither monotonicity
of $\mathfrak D$ nor permanent stationarity of the spacetime.
\end{theorem}
\begin{proof}
Divide the replicator identity for the mass of $S'$ by the one
for $S$ to obtain \eqref{eq:revision-hierarchical}.
Write $h_- =\min\mathcal H$, $h_+=\max\mathcal H$ and
$\bar h=\langle\mathcal H\rangle_\rho$.
Every conditional mean lies in $[h_-,h_+]$, giving the upper bound
in \eqref{eq:revision-maxgap-bound}.
Shrinking an arc about a minimizer yields
$D_I\to\bar h-h_-$, whereas the complement of an arc shrinking
about a maximizer yields $D_I\to h_+-\bar h$.
Their maximum is at least $(h_+-h_-)/2$.
The replicator identity proves the equivalence with $\dot\rho=0$.
Lastly, $\dot p_S=\frac14\int_S(\bar h-\mathcal H)\rho\,d\theta$;
its supremum is one quarter of the positive part of a mean-zero
function, equal to one eighth of its $L^1(\rho)$ norm.
\end{proof}

\begin{theorem}[Minimal current with prescribed energy history]
\label{thm:revision-minimal-current}
Fix a smooth Gowdy solution on a compact interval $I=[a,b]$.
Among the periodic currents $j$ satisfying
$\partial_\tau\mathcal H+\partial_\theta j=-2e^{-2\tau}G$
with the same prescribed $\mathcal H,G$ and $\rho$,
$j_{\mathcal H}$ uniquely minimizes
\[
 \mathcal C_I[j]=\frac12\int_a^b\int_{S^1}\frac{j^2}{\rho}
 \,d\theta\,d\tau.
\]
Every admissible current is $j=j_{\mathcal H}+c(\tau)$, and
\begin{equation}\label{eq:revision-pythagoras}
 \mathcal C_I[j]-\mathcal C_I[j_{\mathcal H}]
 =\frac12\int_a^b c(\tau)^2
       \left(\int_{S^1}\rho^{-1}\,d\theta\right)d\tau\ge0.
\end{equation}
The minimal current also obeys $|j_{\mathcal H}|\le e^{-\tau}\mathcal H$.
\end{theorem}
\begin{proof}
Subtraction of the two balance laws implies
$\partial_\theta(j-j_{\mathcal H})=0$.
The cross term in the expansion of the weighted square vanishes since
\[
 \int_{S^1}\frac{j_{\mathcal H}}{\rho}\,d\theta
 =-4e^{-2\tau}\int_{S^1}\partial_\theta(\rho^{-1})\,d\theta=0.
\]
This proves \eqref{eq:revision-pythagoras} and uniqueness.  Finally,
$2|\langle u_\tau,e^{-\tau}u_\theta\rangle_h|
\le |u_\tau|_h^2+e^{-2\tau}|u_\theta|_h^2$.
\end{proof}

\begin{theorem}[Fisher--Rao action and the concentration cost]
\label{thm:revision-fisher-rao}
Keep the energy field $\mathcal H(\tau,\theta)$ fixed, and denote
$\mathcal E_\tau(r)=\int\mathcal H_\tau r\,d\theta$,
$\operatorname{Var}_r(\mathcal H)=
\int(\mathcal H-\mathcal E_\tau(r))^2r\,d\theta$.
The positive probability density $\rho$ is the unique solution with
initial density $\rho_a$ of
\[
 \partial_\tau\rho=-\tfrac14\rho
 (\mathcal H-\mathcal E_\tau(\rho)).
\]
For a smooth positive probability curve $r_\tau$ with $r_a=\rho_a$,
put
\begin{align*}
 \mathscr A_{\mathcal H}[r]
 :={}&\int_a^b\left[\frac12\int\frac{|\dot r|^2}{r}
       \,d\theta+\frac1{32}\operatorname{Var}_r(\mathcal H)\right]d\tau\\
 &+\frac14[\mathcal E_b(r_b)-\mathcal E_a(r_a)]
   -\frac14\int_a^b\int(\partial_\tau\mathcal H)r\,d\theta\,d\tau.
\end{align*}
Then
\begin{equation}\label{eq:revision-fr-square}
 \mathscr A_{\mathcal H}[r]
 =\frac12\int_a^b\int_{S^1}
  \frac{|\dot r+\frac14r(\mathcal H-\mathcal E_\tau(r))|^2}{r}
  \,d\theta\,d\tau\ge0,
\end{equation}
with equality only for the geometric density $r=\rho$.
For each fixed orbital set $S$, writing $p_S(\tau)=\int_S\rho_\tau\,d\theta$,
\begin{equation}\label{eq:revision-fr-concentration}
 \int_a^b\operatorname{Var}_{\rho_\tau}(\mathcal H_\tau)\,d\tau
 \ge\frac{64}{b-a}
 \left(\arcsin\sqrt{p_S(b)}-\arcsin\sqrt{p_S(a)}\right)^2.
\end{equation}
\end{theorem}
\begin{proof}
The Fisher--Rao norm of a tangent variation $\sigma$ is
$\int\sigma^2/r$; the gradient of $\mathcal E_\tau$ is
$r(\mathcal H-\mathcal E_\tau(r))$.
Expanding the square in \eqref{eq:revision-fr-square}, its mixed term
is $\frac14\int\mathcal H\dot r
=\frac14\frac{d}{d\tau}\mathcal E_\tau(r)
 -\frac14\int(\partial_\tau\mathcal H)r$, proving the action
identity.  Equality yields the prescribed scalar ODE for $r$ at every
$\theta$, whose explicit exponential-tilt solution is unique.
Moreover
\[
 |p'_S|\le\tfrac14\sqrt{p_S(1-p_S)}
 \sqrt{\operatorname{Var}_{\rho}(\mathcal H)}.
\]
Apply this to $\partial_\tau(\arcsin\sqrt{p_S})$ and use
Cauchy--Schwarz in time.
\end{proof}

}

\section{{Normalized volume along Einstein evolution and CMC propagation}}
\label{sec:ADM}
Let $(\mathcal M,\mathbf g)$ be a globally hyperbolic spacetime,
with a smooth foliation by spacelike hypersurfaces $\Sigma_s$.
In ADM variables,
\[
 \mathbf g=-N^2ds^2+\gamma_{ij}(dx^i+\beta^i ds)
 (dx^j+\beta^j ds),\qquad N>0.
\]
We use the convention
$K_{ij}=-\frac12\mathcal L_n\gamma_{ij}$, with $n$ future
unit normal, and set $H=\operatorname{tr}_{\gamma}K$.
With vanishing shift (or after identifying the leaves by its flow),
\begin{equation}\label{eq:ADM-evolution-new}
 \partial_s\gamma_{ij}=-2NK_{ij},\qquad
 \partial_s dV_{\gamma_s}=-NH\,dV_{\gamma_s}.
\end{equation}
In geometrized units the Einstein constraints are
\begin{equation}\label{eq:ADM-constraints-new}
 R_{\gamma}+H^2-|K|_{\gamma}^2=16\pi\rho,
 \qquad \nabla^j(K_{ij}-H\gamma_{ij})=8\pi j_i.
\end{equation}
Here $\rho=T(n,n)$ and $j$ is the energy-momentum flux seen by
the normal observer; in vacuum both vanish.  In this section the
measure of interest is the normalized volume
\[
 d\omega_s=\frac{dV_{\gamma_s}}{\Vol_{\gamma_s}(\Sigma_s)}
\]
on compact leaves.  It is distinct from a normalized matter-mass
measure.

\subsection{Local redistribution and a cumulative criterion}
For a material region $S$ with $0<\omega_s(S)<1$, set
$p(s)=\omega_s(S)$ and write $\langle f\rangle_{A,s}$ for the
conditional mean on a region $A$ with respect to spatial volume.
Differentiating the volumes gives the exact identity
\begin{equation}\label{eq:generic-transfer-new}
 p'(s)=p(s)(1-p(s))
 \Bigl(\langle NH\rangle_{S^c,s}-\langle NH\rangle_{S,s}\Bigr).
\end{equation}
Thus a strict separation of the conditional means produces an
infinitesimal transfer.  A global concentration conclusion requires
control of the time integral of the separation; positivity at each
instant alone is insufficient.

\begin{prop}[Logarithmic odds of a material region]
For every regular evolution interval on which $0<p(s)<1$,
\begin{equation}\label{eq:generic-logodds-new}
 \log\frac{p(s)}{1-p(s)}-
 \log\frac{p(s_0)}{1-p(s_0)}
 =\int_{s_0}^s
 \Bigl(\langle NH\rangle_{S^c,\xi}-
 \langle NH\rangle_{S,\xi}\Bigr)\,d\xi.
\end{equation}
In particular, $p(s)\to1$ at a future endpoint if and only if
the integral on the right tends to $+\infty$.
\end{prop}
\begin{proof}
Divide \eqref{eq:generic-transfer-new} by $p(1-p)$ and
integrate.
\end{proof}

\subsection{Initial mass versus accumulated relative expansion}
\label{sec:amplification-threshold}
The initial probability is an essential part of the transfer criterion
when the candidate concentration region itself shrinks. The following
statement is a direct consequence of the log-odds identity; it does not
require the Einstein equations.

\begin{prop}[Exact initial-mass amplification threshold]
\label{prop:mass-threshold}
Let $S_\delta\subset\Sigma_{s_0}$ be a family of transported
material regions with $0<p_{0,\delta}:=\omega_{s_0}(S_\delta)<1$.
For each fixed $\delta$ define
\begin{equation}\label{eq:relative-amplification}
 I_\delta(s)=\int_{s_0}^{s}
 \bigl(\langle NH\rangle_{S_\delta^c,\xi}
       -\langle NH\rangle_{S_\delta,\xi}\bigr)\,d\xi.
\end{equation}
Then, on any common interval of regular evolution,
\begin{equation}\label{eq:exact-amplification}
 \omega_s(S_\delta)=
 \frac{p_{0,\delta}e^{I_\delta(s)}}
      {1-p_{0,\delta}+p_{0,\delta}e^{I_\delta(s)}}.
\end{equation}
For an arbitrary selection $\delta=\delta(s)$,
\begin{equation}\label{eq:amplification-threshold}
 \omega_s(S_{\delta(s)})\longrightarrow1
 \quad\Longleftrightarrow\quad
 I_{\delta(s)}(s)-
 \log\frac{1-p_{0,\delta(s)}}{p_{0,\delta(s)}}
 \longrightarrow+\infty.
\end{equation}
In particular, if $p_{0,\delta}\to0$, the cumulative relative
amplification has to dominate the logarithmic penalty
$\log(1/p_{0,\delta})$.
\end{prop}
\begin{proof}
Apply \eqref{eq:generic-logodds-new} to each fixed $S_\delta$,
solve for $p_\delta(s)$, and only afterwards substitute
$\delta=\delta(s)$. The equivalence follows by taking logarithmic
odds. No moving-boundary differentiation is used.
\end{proof}

\subsection{Concentration in shrinking tubes}
Let $Z_s\subset\Sigma_s$ be proper closed subsets and
$r_s\downarrow0$ along the parameter limit.  We speak of a
\emph{strong tubular concentration locus} when
\begin{equation}\label{eq:strong-tubular-new}
 \omega_s(T_{\gamma_s}(Z_s,r_s))\longrightarrow1.
\end{equation}
This differs both from the increase of probability of a fixed
material region and from weak convergence of measures after a
particular identification of different leaves.  In particular,
\eqref{eq:strong-tubular-new} makes an intrinsic statement about
the metric on each leaf, but the choice of leaves can matter.

Section~\ref{sec:vacuum-concentration} treats the local and
conditional CMC implications of \eqref{eq:generic-transfer-new}.
The polarized Gowdy model of Section~\ref{sec:gowdy-waves} supplies a separate global example
with a geometrically distinguished areal foliation.  The tube
obstruction of Section~\ref{sec:dimensional-obstruction} supplies a
common method of measuring the geometric degeneration in either
setting.

\subsection{{Vacuum Cauchy data and CMC propagation criteria}}\label{sec:vacuum-concentration}

We distinguish three questions: an infinitesimal change of normalized volume on a
compact Cauchy hypersurface, persistence of that change along a prescribed
foliation, and genuine concentration in shrinking metric neighborhoods before a
curvature singularity.  The first question admits a general answer for vacuum
CMC data; the latter two require additional geometric control.  In
particular, neither a nonconstant lapse nor a local increase of normalized
volume implies by itself a finite-time singularity or a global concentration
locus.

\subsubsection{A local theorem from the vacuum constraint data}
Let $(\Sigma_0,\gamma_0,K_0)$ be smooth vacuum initial data on a connected,
closed three-manifold. We use $K=-\frac12\mathcal L_n\gamma$ and
$h=\operatorname{tr}_{\gamma}K$, so that the initial constraints are
\begin{equation}\label{eq:vacuumconstraints-new}
R_{\gamma_0}+h_0^2-|K_0|_{\gamma_0}^2=0,
\qquad \nabla^i(K_{0ij}-h_0\gamma_{0ij})=0.
\end{equation}
Assume $h_0>0$ is constant and
$|\sigma_0|_{\gamma_0}^2$ is nonconstant, where
$K_0=\sigma_0+\frac13h_0\gamma_0$.
The potential
$|K_0|_{\gamma_0}^2=|\sigma_0|_{\gamma_0}^2+h_0^2/3$
is strictly positive. We parametrize the local CMC foliation by its mean
curvature $h$, increasing toward the chosen future direction, and transport
material regions by the normal flow (zero shift). The CMC lapse satisfies
\begin{equation}\label{eq:CMC-lapse-revised}
(-\Delta_{\gamma_h}+|K_h|_{\gamma_h}^2)N_h=1,\qquad N_h>0.
\end{equation}
Here $\Delta=\nabla^i\nabla_i$, and the normalization $dh/dh=1$
fixes the source in \eqref{eq:CMC-lapse-revised}.

\begin{theorem}[Local normalized-volume transfer]\label{thm:local-lapse-revised}
Under the preceding assumptions there exists a smooth local vacuum CMC
foliation for $h\in[h_0,h_0+\delta_0)$, with $N_h>0$. Its initial lapse
$N_{h_0}$ is nonconstant. There exist a material open region
$S\Subset\Sigma_0$ (with positive-volume complement) and $\delta,c>0$ such
that, for $h\in[h_0,h_0+\delta]$,
\begin{equation}\label{eq:local-gap-new}
D(h):=\langle N_h\rangle_{S^c,h}-\langle N_h\rangle_{S,h}
\geq \frac{c}{h^2}>0.
\end{equation}
Consequently, if
$p(h)=\operatorname{Vol}_{\gamma_h}(S)/\operatorname{Vol}_{\gamma_h}(\Sigma_h)$,
then
\begin{equation}\label{eq:logistic-revised}
p'(h)=hD(h)p(h)(1-p(h))>0.
\end{equation}
This conclusion is stable under sufficiently small perturbations within the
CMC vacuum constraint class. 
\end{theorem}
\begin{proof}
The standard local CMC implicit-function argument applies because the
Jacobi operator $-\Delta_{\gamma_0}+|K_0|^2$ has strictly positive
potential; its inverse maps the constant function $1$ to a positive lapse by
the strong maximum principle. If $N_{h_0}$ were constant,
\eqref{eq:CMC-lapse-revised} would force $|K_0|^2$, hence $|\sigma_0|^2$,
to be constant, a contradiction. Choose a regular value
$r\in(\min N_{h_0},\max N_{h_0})$ and take the smoothly bounded sublevel set
$S=\{N_{h_0}<r\}$. Its complement satisfies $N_{h_0}\ge r$,
whereas $N_{h_0}<r$ on $S$, so
$\langle N_{h_0}\rangle_S<\langle N_{h_0}\rangle_{S^c}$.
The smooth dependence of the CMC solution and its elliptic lapse on $h$
and on the constraint data preserves this strict gap on a sufficiently
short interval and in a neighborhood of the initial data; multiplying by
$h^2$ yields \eqref{eq:local-gap-new} after decreasing $c$.
Finally $\partial_h dV_{\gamma_h}=-hN_h\,dV_{\gamma_h}$, and differentiating
the two volumes yields \eqref{eq:logistic-revised}.
\end{proof}

\subsubsection{The cumulative chain and its exact obstruction}
For a fixed transported region $S$ define
\begin{equation}\label{eq:E-def-revised}
E(h)=h^2D(h),\qquad u_h=h^2N_h,\qquad \tau=\log h.
\end{equation}
The logistic identity integrates exactly to
\begin{equation}\label{eq:log-odds-revised}
\log\frac{p(h)}{1-p(h)}
=\log\frac{p(h_0)}{1-p(h_0)}
+\int_{h_0}^{h}\frac{E(s)}s\,ds.
\end{equation}
It follows that $p(h)\to1$ if and only if the integral on the right
diverges to $+\infty$.  Merely retaining $D(h)>0$ is not sufficient.

\begin{prop}[Exact evolution of the lapse gap]\label{prop:E-identity-revised}
On every smooth CMC vacuum interval, with the preceding material
identification,
\begin{equation}\label{eq:E-identity-revised}
\frac{dE}{d\tau}=
\langle\partial_\tau u\rangle_{S^c,h}
-\langle\partial_\tau u\rangle_{S,h}
+\operatorname{Var}_{S,h}(u)-\operatorname{Var}_{S^c,h}(u).
\end{equation}
Here variances and conditional means are taken with respect to the
normalized spatial volume on each material region. In particular, the
sign of $dE/d\tau$ is not determined by the vacuum constraints.
\end{prop}
\begin{proof}
From $\partial_h dV=-hN_h dV$ we obtain
$\partial_\tau dV=-u\,dV$. For a transported region $A$ and a function
$f$,
\[
\frac{d}{d\tau}\langle f\rangle_{A,h}
=\langle\partial_\tau f\rangle_{A,h}
-\langle fu\rangle_{A,h}
+\langle f\rangle_{A,h}\langle u\rangle_{A,h}.
\]
Put $f=u$ and subtract the identities for $S^c$ and $S$.
\end{proof}

\begin{prop}[Conditional global chain criterion]\label{prop:chain-criterion-revised}
Assume, in addition, that a smooth compact CMC vacuum foliation exists for
all $h\ge h_0$ and that for some $A,\delta>0$,
\begin{equation}\label{eq:E-assumption-revised}
\frac{dE}{dh}\ge
-\frac{A}{h}\left(\frac{h_0}{h}\right)^\delta,
\qquad E(h_0)>\frac{A}{\delta}.
\end{equation}
Then $E(h)\ge\varepsilon:=E(h_0)-A/\delta>0$ and
\begin{equation}\label{eq:chain-estimate-revised}
p(h)\ge
\left[1+\frac{1-p(h_0)}{p(h_0)}
\left(\frac{h_0}{h}\right)^\varepsilon\right]^{-1}
\longrightarrow1.
\end{equation}
This is a sufficient criterion for concentration of normalized volume in
the \emph{fixed material region} $S$, not yet for concentration in shrinking
metric tubes.
\end{prop}
\begin{proof}
Integrating \eqref{eq:E-assumption-revised} from $h_0$ to $h$ gives
$E(h)\ge E(h_0)-A/\delta\ge\varepsilon$.
Insert this lower bound in \eqref{eq:log-odds-revised} and exponentiate.
\end{proof}

For clarity, \eqref{eq:E-assumption-revised} is \emph{not} deduced here
from the data in Theorem~\ref{thm:local-lapse-revised}. In particular, the
unrestricted vacuum constraints do not guarantee the needed global CMC
continuation or a curvature singularity.

\subsubsection{Infinitesimal invariance and the limits of the constraint equations}
\label{sec:cmc-invariance-barrier}
A positive gap on one CMC slice persists locally by continuity; that
fact alone is not an iteration principle. The desired chain must
follow from a \emph{closed differential barrier} satisfied on every
regular step. The following conditional lemma makes the distinction
precise.

\begin{prop}[Conditional infinitesimal barrier for the CMC chain]
\label{prop:cmc-infinitesimal-barrier}
Suppose the smooth compact CMC foliation exists for
$\tau=\log h\in[\tau_0,\infty)$, and let $S$ be one fixed transported
region. Assume $E(\tau_0)>0$ and
\begin{equation}\label{eq:cmc-invariance-barrier}
 \frac{dE}{d\tau}\ge-a(\tau)E(\tau),
 \qquad a\ge0,\quad
 \int_{\tau_0}^{\infty}a(s)\,ds<\infty.
\end{equation}
Then $E(\tau)\ge c>0$ for all $\tau\ge\tau_0$ and
$\omega_{h}(S)\to1$ as $h\to\infty$.
\end{prop}
\begin{proof}
Multiplying \eqref{eq:cmc-invariance-barrier} by
$\exp(\int_{\tau_0}^{\tau}a)$ yields
$E(\tau)\ge E(\tau_0)\exp(-\int_{\tau_0}^{\tau}a)\ge c>0$.
Insert this bound in \eqref{eq:log-odds-revised}.
\end{proof}

The assumptions in Proposition~\ref{prop:cmc-infinitesimal-barrier}
are \emph{not} a consequence of the vacuum constraint equations alone.
Indeed, differentiating the CMC lapse equation schematically gives
\[
 \bigl(-\Delta_{\gamma_h}+|K_h|^2\bigr)\partial_\tau N_h
 =-\partial_\tau\bigl(-\Delta_{\gamma_h}+|K_h|^2\bigr)N_h,
\]
where the derivative of the operator must also be computed using
metric and second fundamental form evolution. In particular,
$\partial_\tau K$ involves the full spatial Ricci tensor, not merely
the scalar curvature constrained by the Hamiltonian equation.
Constraint propagation thus preserves the constraint surface, but
not automatically the positive-gap cone $E>0$ or its lower barrier.
Likewise, local reapplication does not imply global continuation:
the steps may accumulate at an endpoint where the required geometric
norms lose control. Even a preserved sign $E>0$ is not enough unless
$\int E\,d\tau=+\infty$.

\begin{prop}[An obstruction in a quasi-Kasner CMC regime]\label{prop:kasner-no-gap-revised}
Suppose a compact vacuum CMC foliation exists for $h\in[h_0,\infty)$
and, for some $C,\delta>0$ and sufficiently large $h$,
\[
\left\|\frac{|K_h|^2}{h^2}-1\right\|_{L^\infty(\Sigma_h)}
\le Ch^{-\delta}.
\]
Then $E(h)=O(h^{-\delta})$, $\int_{h_0}^\infty|E(h)|\,dh/h<\infty$,
and $p(h)$ converges to a value strictly between $0$ and $1$ for every
fixed material region with nontrivial initial volume. Thus the global
gap $E(h)\ge\varepsilon>0$ cannot hold in this regime.
\end{prop}
\begin{proof}
Put $\bar\gamma_h=h^2\gamma_h$ and $Q_h=|K_h|^2/h^2$.
The lapse equation becomes
$(-\Delta_{\bar\gamma_h}+Q_h)u_h=1$.
If $\epsilon_h=Ch^{-\delta}<1$, the maximum principle gives
$1/(1+\epsilon_h)\le u_h\le1/(1-\epsilon_h)$.
Therefore $|E(h)|\le2\epsilon_h/(1-\epsilon_h^2)$ and the integral
converges. The log-odds identity \eqref{eq:log-odds-revised} then has a
finite real limit, which implies $p_\infty\in(0,1)$.
\end{proof}

\section{{Geometric foundations: concentration, tubes and transport}}\label{sec:geometry-foundations}
\subsection{{Metric--measure preliminaries}}\label{sec:mm-preliminaries}
Given a set $X$, we denote by $l^{\infty}(X)$ the unital
Banach algebra of all bounded real-valued functions on $X$ equipped with the supremum norm.
Let $X$ be a topological space. If the topology of $X$ is generated by a metric $d$, then we call
$d$ a compatible metric on $X$.\\
We will denote by $C(X)$ the set of all continuous real-valued functions on $X$, and we set $CB(X)=l^{\infty}(X)\cap C(X)$. Let us denote by $\mathcal{B}(X)$ the
Borel $\sigma$-algebra of $X$ and by $P(X)$ the set of all Borel probability measures on $X$.
The weak topology on $P(X)$ is defined to be the initial topology on $P(X)$ generated by the maps of the
form $P(X) \rightarrow \mathbb{R}, \mu\rightarrow \int f d\mu$ where $f \in CB(X) $.\\
The support of a measure $\mu\in P(X)$ is defined as
\begin{align*}
spt\mu=\{x\in X\mid \forall U\subseteq X\mbox{ open: }x\in U\Rightarrow \mu (U)>0\}.
\end{align*}
Given $\mu\in P(X)$ and a Borel subset $B \subseteq X$
with $\mu(B) = 1$, we let $\mu|_B :=\mu|_{\mathcal{B}(B)}\in P(B)$. The push-forward of a measure $\mu\in P(X)$ along
a Borel map $f : X\rightarrow Y$ into another topological space $Y$ is defined to be
\begin{align*}
 \sharp f\mu:\mathcal{B}(Y)\rightarrow [0,1], B\mapsto \mu (f^{-1}(B)).
\end{align*}
Furthermore, let us note that each $\mu\in P(X)$ gives rise to a pseudo-metric
$me_{\mu}$ on the set of
all Borel measurable real-valued functions $X$ defined by
\begin{align*}
me_{\mu}(f,g):=\inf\{\epsilon >0 \mid \mu (\{x\in X\mid \mid f(x)-g(x)\mid >\epsilon \})\le\epsilon\},
\end{align*}
for any two Borel functions $f,g:X\rightarrow\mathbb{R}$.
Let $(X, d)$ be a pseudo-metric space. Given a subset $A\subseteq X$, we abbreviate $d|_A := d|_{A\times A}$ and define ${\rm diam}(A,d):=\sup\{d(x,y)\mid x,y\in A\}$. For $x\in A\subseteq X$ and $\epsilon >0$ we set
\begin{align*}
B_d(x,\epsilon ):=\{y\in X\mid d(x,y)<\epsilon \}\ \ B_d(A,\epsilon ):=\{y\in X\mid \exists a\in A ; d(x,y)<\epsilon \}.
\end{align*}
Then the Hausdorff distance between any two subsets $A,B \subseteq X$ is given by
\begin{align*}
d_H(A,B):=\inf \{\epsilon >0\mid B\subseteq B_d(A,\epsilon ),\ A\subseteq B_d(B,\epsilon )\}.
\end{align*}
For $l,r\ge 0$, we denote by $Lip_l(X, d)$ the set of all $l$-Lipschitz real-valued functions on $(X, d)$,
and we define
\begin{align*}
Lip_l^{\infty}(X,d):=Lip_l(X,d)\cap l^{\infty}(X),\ Lip_l^{r}(X,d):=\{f\in Lip_l(X,d)\mid \|f\|_{\infty}\le r\}.
\end{align*}
Moreover, we set $Lip(X,d):=\bigcup \{Lip_l(X,d)\mid l\ge 0\}$ and $Lip(X,d)^{\infty}:=Lip(X,d)\cap l^{\infty}(X)$.\\
{
On a complete separable metric space $(X,d)$, the Wasserstein distance $W_1$
is defined on probability measures with finite first moment. The coupling
formulation and the Kantorovich--Rubinstein duality give
\[
 W_1(\mu,\nu)=\inf_{\pi\in\Pi(\mu,\nu)}\int d(x,y)\,d\pi(x,y)
 =\sup_{f\in Lip_1(X,d)}\left|\int f\,d\mu-\int f\,d\nu\right|.
\]
Restricting additionally to $\|f\|_\infty\le1$ defines instead a
bounded-Lipschitz metric, which metrizes weak convergence of probability
measures on Polish spaces. The $W_1$ topology also controls first moments;
on a compact metric space it agrees with weak convergence. See
\cite{VilBook,VilBook1}.}

\begin{tdef}{(Gromov-Milman \cite{Gro99})}\\
  A space with a metric $g$ and a measure $\mu$, or an mm-space, is a triple $(X,\mu,g)$, consisting of a set $X$, a metric $g$ on $X$ and a probability Borel measure such that $(X,g)$ is a separable complete metric space.
\end{tdef}
Moreover, an mm-space $(X,\mu ,d)$ is called compact if $(X, d)$ is compact, and fully supported if $spt \mu= X$.
Henceforth, we will denote by $\lambda$ the Lebesgue measure on $[0, 1)$.\\
A parametrization of an mm-space $(X,\mu ,d)$ is a Borel measurable map $\phi: [0, 1)\rightarrow X$ such that $\sharp\phi\lambda = \mu$. It is well known that any mm-space admits a parametrization (see, e.g.\cite{Shi}).
In the set of isomorphism classes of mm-spaces, $\mathfrak{M}$, we can define the box distance, $d_{box}$, that we are going to define (\cite{Shi}).
For two pseudo-metrics $\rho_1$ and {$\rho_2$}
on the unit interval $I$, we define $d_{box}(\rho_1,\rho_2)$ to be the infimum of $\epsilon >0$ satisfying that
there exists a Borel subset $I_0\subseteq I$ such that
\begin{enumerate}
          \item ${\vert\rho_1(s,t)-\rho_2(s,t)\vert}\le\epsilon$ for any $s,t\in I_0$,
          \item $\mathfrak{L}^1(I_0)\ge 1-\epsilon$ where $\mathfrak{L}^1$ denotes the one-dimensional Lebesgue measure.
\end{enumerate}

\begin{tdef}
  let $X$ be a topological space with a Borel probability measure $\mu_X$. A map $\varphi:I\rightarrow X$ is called a parameter of $X$ if $\varphi$ is a Borel measurable map such that

  $\sharp\varphi\mathfrak{L}^1=\mu_X$
\end{tdef}

We define box distance $d_{box}$ between two isomorphism classes of mm-spaces $X,Y$ to be the infimum of $d_{box}(\sharp\varphi d_X,\sharp\psi d_Y)$ where
$\varphi:I\rightarrow X$ and $\psi:I\rightarrow Y$ run over all parameters of $X$ and $Y$, respectively, and where $\sharp\varphi d_X(s,t):=d_X(\varphi(s),\varphi(t))$.

\begin{tdef}{(Gromov-Milman \cite{Gro99}\cite{Shi})}\\
In the set of isomorphism classes of mm-spaces we can define the following distance:
\begin{align*}
d_{conc}(X,Y):= \inf\{(me_{\lambda})_H(Lip_1(X)\circ\phi , Lip_1(Y)\circ\psi) \mid \phi\mbox{ param. of }X,\psi\mbox{ param. of }Y\}.
\end{align*}
\end{tdef}
Two mm-spaces $X$ and $Y$ are isomorphic if there exists an isomorphism between mm-spaces $(X,\mu ,d),(Y,\nu ,d')$ i.e an isometry
\begin{align*}
{f: (\operatorname{spt}\mu,d|_{\operatorname{spt}\mu\times\operatorname{spt}\mu})\rightarrow (\operatorname{spt}\nu,d'|_{\operatorname{spt}\nu\times\operatorname{spt}\nu})}
\end{align*}
such that  $\sharp f(\mu\mid {spt\mu})=\nu\mid {spt\nu}$.
A sequence of mm-spaces $(X_n,\mu_n,g_n)$ is said to concentrate to an mm-space $(X,\mu,g)$ if
\begin{align*}
\lim_n d_{conc}(X_n,X)=0.
\end{align*}

In this case, we denote $(X,\mu,g)$ as a concentration set for the sequence of mm-spaces $(X_n,\mu_n,g_n)$.

In case of a continuous sequence $(X_s,\mu_s,g_s)$ $s\in [t_0,t_1]$ similarly 
a sequence of mm-spaces $(X_n,\mu_n,g_n)$ is said to concentrate to an mm-space $(X,\mu,g)$ if
\begin{align*}
\lim_s d_{conc}(X_s,X)=0.
\end{align*}

Finally, let us recall the definition of Concentration Locus as defined in \cite{CacUrs22}

\begin{tdef}\label{def2}
 Let $\{X_n, \mu_n\}_{n\in \mathbb N}$ be a family of metric spaces with metrics $d_n$, and Borel's measures $\mu_n$ w.r.t. which nonempty open set have non-vanishing measure. Assume the measures to be normalized, $\mu_n(X_n)=1$.
 Let $\{S_n\}_{n\in \mathbb N}$ be a family of proper closed subsets, $S_n\subset X_n$. Fix a sequence $\{\varepsilon_n\}_{n\in \mathbb N}$ such that $\varepsilon_n>0$,  $\lim_{n\to\infty}\varepsilon_n=0$, and let
 $\{U^{\varepsilon_n}_n\}_{n\in \mathbb N}$ be the sequence of tubular neighbourhoods of $S_n$ of radius $\varepsilon_n$. We say that the family
 $\{S_n\}$ is a Concentration Locus if
 \begin{align}
 \lim_{n\to\infty} \mu_n(X_n-U^{\varepsilon_n}_n)=0.
 \end{align}

 Moreover, if such a sequence $\varepsilon_n$ converges to 0 at rate $k$ (so that $\lim_{n\to\infty} n^k\varepsilon_n=c$ for some constant $c$), we say that the family $\{S_n\}$ is a Concentration Locus at least at rate $k$.
\end{tdef}

\subsection{{Tube concentration in increasing dimension}}
\label{sec:discrete}
Let $(N_j,d_j,\mu_j)$ be compact metric--measure spaces with
normalized Riemannian volume, and let $Z_j\subsetneq N_j$ be closed
embedded submanifolds.  We call $(Z_j)$ a \emph{concentration locus}
if there exist $r_j\downarrow0$ such that
\[
 \mu_j\bigl(N_j\setminus T_{d_j}(Z_j,r_j)\bigr)\longrightarrow0,
 \qquad T_{d_j}(Z_j,r)=\{x:d_j(x,Z_j)<r\}.
\]
The relevant distance on $Z_j$ for the projection arguments below is
initially the restriction of the ambient distance; it need not agree
with the intrinsic distance of the induced Riemannian metric.

The normal Jacobi system developed in Section~\ref{sec:tube-jacobi} records how the
geometry of $Z_j$ and the ambient curvature enter the volume of its
tubes.  It is not enough, in general, to estimate dimension, rank or
diameter alone: one also needs the normal cut range and control of the
remaining Jacobian factors.  The round sphere admits a direct estimate.

\begin{teo}[Great subspheres]\label{proptube}
Let $N_j=S_R^{n_j+q_j}$ and $Z_j=S_R^{n_j}$ be a great subsphere of
codimension $q_j\ge1$, with $R>0$ fixed.  If $0<r_j<\pi R/2$, then
\begin{equation}\label{eq:spherical-initial-estimate}
 \mu_j(N_j\setminus T(Z_j,r_j))\le
 \frac{q_j}{(n_j+q_j+1)\sin^2(r_j/R)}.
\end{equation}
In particular, a sufficient condition for concentration is
$r_j\to0$ together with
$q_j/((n_j+q_j+1)\sin^2(r_j/R))\to0$.
\end{teo}
\begin{proof}
Write a point of $N_j$ as $R(y,z)$, where
$y\in\mathbb R^{n_j+1}$, $z\in\mathbb R^{q_j}$ and
$|y|^2+|z|^2=1$.  Then $d(R(y,z),Z_j)=R\arcsin|z|$ and,
by rotational invariance,
$\int|z|^2\,d\mu_j=q_j/(n_j+q_j+1)$.
Apply Markov's inequality to $|z|^2$.
\end{proof}

In Section~\ref{sec:dimensional-obstruction} we prove that, when
$n_j\to\infty$, a sequence of vanishing concentration radii exists
\emph{if and only if} $q_j=o(n_j)$.  {Throughout, any positive sequence of radii tending to zero
may be replaced by its decreasing tail supremum when monotone radii
are required; enlarging tubes preserves the concentration assertion.}
The metric--measure consequences
of nearest-point projection are considered in Section~\ref{box}.

\subsection{{Normal Jacobi geometry and tube volumes}}\label{sec:tube-jacobi}
Let $M\subseteq N$ be a compact embedded Riemannian submanifold without boundary, of codimension $q$ in a Riemannian manifold $N$. We call $M_{\epsilon}$ the tube generated by all geodetic segments of length $\epsilon$, outgoing perpendicularly from $M$. It is a well-known result established by Weyl that the volume of the tube can be expressed in terms of the Lipschitz-Killing curvatures $K_{2j}$ of $M$ and the ambient space, and the codimension $q$. More precisely:
\begin{teo}[Weyl, \cite{Weyl}]
 Let $M^n$ be a smooth compact embedded Riemannian submanifold without boundary of $\mathbb R^N$, of codimension $q=N-n\ge1$, with the induced metric. Let $M^\varepsilon$ denote its tubular neighbourhood of radius $\varepsilon$. Then for all $\varepsilon>0$ below the normal injectivity radius,
 \begin{align}
 {\rm Vol}_{\mathbb R^N}(M^\varepsilon)=\frac {\pi^{\frac q2}\varepsilon^q}{\Gamma(\frac q2+1)} \left(K_0(M) +\sum_{j=1}^{\lfloor n/2\rfloor} \frac {K_{2j}(M)\varepsilon^{2j}}{(q+2)(q+4)\cdots (q+2j)}\right),
 \end{align}
 where
 \begin{align}
 K_{2j}(M) =\int_M k_{2j}(\Omega),\qquad K_0(M)=\operatorname{Vol}_M(M),
 \end{align}
 are the integrated Lipschitz-Killing curvatures, and $\Omega$ is the curvature 2-form of $M$.
\end{teo}
With respect to a local oriented orthonormal frame $V_a$, $a=1,\ldots,n$, and its dual coframe $e^a$, the Lipschitz--Killing densities are represented by the following top-degree forms (on a nonorientable $M$, the corresponding densities are integrated):
\begin{align}
 k_{2j}(\Omega)=\frac 1{2^j j! (n-2j)!} \sum_{\sigma\in S_n} \epsilon_\sigma \Omega_{\sigma(1)\sigma(2)}\wedge \cdots \wedge \Omega_{\sigma(2j-1)\sigma(2j)} e^{\sigma(2j+1)}\wedge \cdots \wedge e^{\sigma(n)},
\end{align}
where $S_n$ is the set of permutations of $n$ elements and $\epsilon_\sigma$ the sign of the permutation. In particular,
\begin{align}
 k_0(\Omega)=d{\rm Vol}_M,
\end{align}
is the volume form on $M$,
\begin{align}
 k_2(\Omega)=&\frac 12 R\ d{\rm Vol}_M,\qquad n\ge2, \\
 k_n(\Omega)=& {\rm Pf}(\Omega),\qquad n\text{ even},
\end{align}
where $R$ is the scalar curvature of $M$ and Pf the Pfaffian. Finally, notice that
\begin{align}
 \frac {\pi^{\frac q2}\varepsilon^q}{\Gamma(\frac q2+1)}={\rm Vol}_{\mathbb R^q} (D_\varepsilon)
\end{align}
is the volume of the $q$ dimensional disc of radius $\varepsilon$ in $\mathbb R^q$. So, if we define the {\it mean Lipschitz-Killing curvatures} $\kappa_{2j}$ as
\begin{align}
\kappa_{2j}=\frac {K_{2j}(M) }{{\rm Vol}_M(M)},
\end{align}
then we can rewrite Weyl's formula as
 \begin{align}
 {\rm Vol}_{\mathbb R^N}(M^\varepsilon)={\rm Vol}_M(M) {\rm Vol}_{\mathbb R^q} (D_\varepsilon)  \left(1 +\sum_{j=1}^{\lfloor n/2\rfloor} \frac {\kappa_{2j}(M)\varepsilon^{2j}}{(q+2)(q+4)\cdots (q+2j)}\right).\label{formula}
 \end{align}
For fixed dimension, the Euclidean tube formula is a polynomial
identity valid below the normal cut radius.  Its coefficients are
intrinsic Lipschitz--Killing integrals, even though the normal
exponential map is governed by the extrinsic geometry.  In
particular, a totally geodesic subsphere of a round sphere has
nonzero intrinsic curvature; vanishing second fundamental form
does \emph{not} make its higher Lipschitz--Killing curvatures
vanish.  The ambient sphere also contributes to the normal Jacobi
equation, so the Euclidean expression cannot simply be substituted
for a spherical tube formula when the dimension grows.

For a fixed smooth embedding and fixed ambient dimension, a
small-radius tube can be analysed by Taylor expansion of its
Jacobi determinant, with an error depending on uniform bounds for
curvature and its derivatives.  Such estimates are \emph{not}
uniform in dimension without corresponding dimension-dependent
hypotheses.  We therefore use the exact probability distribution
of the normal coordinates on a round sphere in
Theorem~\ref{thm:fixed-varying-dimension}, rather than
extrapolating Weyl's Euclidean formula.

\subsubsection{The general case}

Let $Z^d\subset(N^m,g)$ be a closed embedded submanifold of
codimension $q=m-d\ge1$.  Below its normal injectivity radius,
write the normal exponential map as
\[
 (z,\nu,r)\longmapsto\exp_z(r\nu),
 \qquad z\in Z,\quad\nu\in S(N_zZ),\quad r>0.
\]
Along the corresponding unit-speed normal geodesic, choose a
parallel orthonormal frame perpendicular to $\dot\gamma$.
For tangential variations let $J_i$ denote the Jacobi fields
with $J_i(0)=e_i$ and $J_i'(0)=-S_\nu e_i$; for normal-angular
variations let $Y_a$ satisfy $Y_a(0)=0$ and $Y_a'(0)=\eta_a$.
Put their $m-1$ components into the square matrix
\[
 \mathsf J(r)=(J_1(r),\ldots,J_d(r),Y_1(r),\ldots,Y_{q-1}(r)).
\]
The matrix solves the exact Jacobi initial-value problem
\begin{equation}\label{Jacobian-equation}
 \mathsf J''(r)+\mathcal R(r)\mathsf J(r)=0,
 \quad \mathcal R(r)v=R(v,\dot\gamma)\dot\gamma.
\end{equation}
The curvature endomorphism is symmetric; \emph{compactness of $N$
does not imply that it is positive}.  Its coefficients are smooth,
but a Taylor series for $\mathsf J$ need not converge in the smooth
category.  No analyticity assumption is required for the
Jacobi initial-value problem itself.

\begin{prop}[Exact normal Jacobi density]
\label{prop:exact-normal-jacobi-density}
For $0<r<\operatorname{ninj}_g(Z)$, the Riemannian
volume element in normal polar coordinates is
\begin{equation}\label{formulone}
 dV_g=|\det\mathsf J(r)|\,dr\,d\nu\,dV_Z.
\end{equation}
Consequently the tube volume is obtained by integrating this
Jacobi determinant over the unit normal sphere bundle and
$0<r<\varepsilon$. For a family of pairs $(N_j^m,Z_j^d)$
of fixed dimensions, assume a common positive normal radius,
a uniform bound on $\|\mathrm{II}_{Z_j}\|_\infty$, and uniform
bounds on the ambient curvature tensor and its covariant
derivatives through the finite order needed in the expansion,
on the corresponding tubes. Then the Jacobi matrix and its
determinant have uniform finite-order Taylor remainder bounds
on any fixed smaller normal radius. These are pointwise density
bounds; integrated remainder bounds also involve the volume
of the unit normal sphere bundle.
\end{prop}
\begin{proof}
The fields $J_i,Y_a$ are the differential of the normal
exponential map in the tangential and angular-normal directions.
The Gauss lemma makes $\dot\gamma$ orthogonal to all of them,
and its length is one.  The absolute determinant is therefore
the exact radial Jacobian.  Integrating in the normal polar chart
is legitimate below the normal injectivity radius.  Finite-order
Taylor estimates follow from repeated differentiation of the
Jacobi ODE and ordinary ODE remainder estimates. In the family
version the initial matrices have uniformly bounded norms:
$\mathsf J(0)$ consists of orthonormal tangential vectors and
zero angular columns, and $\mathsf J'(0)$ consists of
$-S_\nu e_i$ and orthonormal angular vectors. Their norms are
therefore controlled by the fixed dimensions and the bound on
$\mathrm{II}$. Gronwall's inequality for the first-order system
for $(\mathsf J,\mathsf J')$ gives a common bound on the fixed
radial interval. Differentiating the ODE then bounds the
successive radial derivatives using the assumed curvature
derivative bounds. The determinant is a polynomial of fixed
degree in the entries, so the same uniformity holds for its
Taylor remainder (with the radial Jacobian sign fixed by the
initial frame orientation below the normal injectivity radius).
\end{proof}

This formula is the starting point for the fixed-dimensional
uniform obstruction proved later.  It is valid for smooth
metrics of either curvature sign, without assuming convergence
of an infinite analytic series.

\subsubsection{Compact symmetric spaces}
Let $\Sigma=G/H$ be a compact Riemannian symmetric space of compact
type and let $Z\subset\Sigma$ be a closed embedded submanifold.  In
a parallel orthonormal frame along a normal geodesic, the complete
transverse curvature operator $\mathcal R$ is constant because
$\nabla\operatorname{Rm}=0$.  Write $J_0$ and $\dot J_0$ for the
initial transverse matrices of the normal Jacobi fields, including
the angular normal directions.  Below the normal cut radius their
matrix of values is
\begin{equation}\label{eq:symmetric-jacobi-matrix}
 J(t)=\cos(t\sqrt{\mathcal R})J_0+
 \frac{\sin(t\sqrt{\mathcal R})}{\sqrt{\mathcal R}}\dot J_0,
\end{equation}
where the quotient is interpreted as $t$ on the kernel of
$\mathcal R$.  Hence the tubular volume density is
$|\det J(t)|\,dV_Z\,d\nu\,dt$.  This includes possible coupling
between tangential and angular-normal Jacobi fields.

\begin{prop}[Symmetric-space tubular Jacobian]
For every radius below the normal injectivity radius of $Z$,
formula \eqref{eq:symmetric-jacobi-matrix} gives the full
Jacobian of the normal exponential map.  If $Z$ is totally geodesic
and the curvature operator preserves the tangential/normal-angular
splitting, the tangential determinant is the product
$\prod_{i=1}^{\dim Z}\cos(d_it)$, where $d_i^2$ are the eigenvalues
of the tangential curvature block.
\end{prop}
\begin{proof}
The transverse Jacobi matrix solves $J''+\mathcal RJ=0$ with
initial conditions $J_0,\dot J_0$.  Functional calculus for the
nonnegative symmetric operator $\mathcal R$ yields the unique
solution.  The Gauss lemma identifies its determinant with the
Jacobian of the normal exponential map.  Under the additional
splitting assumption the matrix is block diagonal and the tangent
block has initial values $I,0$, giving the stated cosine factors.
\end{proof}

The local formula does not ensure a global injective parametrization
of the ambient manifold by normal coordinates.  Nor does the product
of cosine factors alone give concentration: the cut radius, the
other Jacobian factors and total volume must be estimated.  In
particular the sphere statement in Theorem~\ref{proptube} is not
asserted for an arbitrary sequence of compact symmetric spaces by
replacing dimension with dimension minus rank.  This is precisely why
we distinguish the exact spherical result from the general
fixed-dimensional obstruction below.

\subsection{{Nearest-point transport, Wasserstein distance and box convergence}}
\label{box}

Let $(N,d)$ be a compact Riemannian manifold with a Borel
probability measure $\mu$, and let $Z\subset N$ be a nonempty
closed submanifold.  Compactness gives a nearest point
$\pi(x)\in Z$ for each $x\in N$; a Borel choice of nearest point
can be made, including on the cut locus.  Put
$\sigma=\pi_\#\mu$, and endow $Z$ with the
\emph{restricted ambient distance} $d|_{Z\times Z}$.
The latter need not equal the intrinsic Riemannian distance on $Z$.
For a Wasserstein comparison, $\sigma$ is viewed as a measure on
$N$ supported on $Z$.

\begin{prop}[Projection and Wasserstein cost]\label{WassConcLoc}
For $1\le p<\infty$ one has, whenever the cost is finite,
\begin{equation}\label{eq:projected-Wp}
 W_p(\mu,\pi_\#\mu)^p=\int_N d(x,Z)^p\,d\mu(x).
\end{equation}
If $W_2(\mu_j,\pi_{j\#}\mu_j)\to0$ for a sequence of
compact spaces and closed subsets $Z_j$, then $(Z_j)$ is a
concentration locus.  Conversely, if $\operatorname{diam}N_j\le D$
uniformly and $(Z_j)$ is a concentration locus, then
$W_1(\mu_j,\pi_{j\#}\mu_j)\to0$.
\end{prop}
\begin{proof}
The coupling $(\operatorname{id},\pi)_\#\mu$ has cost
$\int d(x,Z)^p\,d\mu(x)$.  Every coupling whose second marginal
is supported on $Z$ has at least this cost pointwise, proving
\eqref{eq:projected-Wp}.  Let
$a_j=W_2(\mu_j,\pi_{j\#}\mu_j)$ and choose
$r_j=\max\{\sqrt{a_j},j^{-1}\}$; then $r_j\to0$ and
\[
 \mu_j\{d(x,Z_j)\ge r_j\}\le a_j^2/r_j^2\le a_j\to0.
\]
For the converse let $r_j\to0$ witness concentration, put
$\delta_j=\mu_j\{d(x,Z_j)\ge r_j\}\to0$ and split
\eqref{eq:projected-Wp} for $p=1$ over the tube and its
complement.  This gives $W_1\le r_j+D\delta_j\to0$.
\end{proof}

The following direct comparison avoids a diameter restriction.
We use the convention for $d_{\mathrm{box}}$ in \cite{Shi}, in
which distances between two measure parametrizations are compared
on a common subset of $[0,1)$ whose complement has small Lebesgue
measure.

\begin{prop}[Direct box bound for nearest-point projection]
\label{concbox}
For a compact metric--measure space $(N,\mu,d)$, a nonempty
closed subset $Z$ and a Borel nearest-point projection $\pi$,
write $\sigma=\pi_\#\mu$.  For every $r>0$ let
$\delta=\mu\{d(x,Z)\ge r\}$.  Then
\begin{equation}\label{eq:box-bound-revised}
 d_{\mathrm{box}}\bigl((N,\mu,d),
 (Z,\sigma,d|_{Z\times Z})\bigr)
 \le\max\{2r,\delta\}.
\end{equation}
Consequently, if $(Z_j)$ is a concentration locus, then
$d_{\mathrm{box}}(N_j,Z_j)\to0$ with the indicated measures and
distances.
\end{prop}
\begin{proof}
Choose a measure parametrization $\phi:[0,1)\to N$ of $\mu$,
and put $\psi=\pi\circ\phi$.  Then
$\psi_\#\mathcal L^1=\sigma$.  On the measurable set
$I_0=\{s:d(\phi(s),\psi(s))<r\}$ one has
$\mathcal L^1(I_0)\ge1-\delta$ and, for $s,t\in I_0$,
\[
 \bigl|d(\phi(s),\phi(t))-d(\psi(s),\psi(t))\bigr|<2r
\]
by the triangle inequality.  The definition of box distance
gives \eqref{eq:box-bound-revised}.
\end{proof}

\begin{teo}[Transfer of concentration convergence]
\label{ConcLocDconc}
Let $Z_j\subset N_j$ be concentration loci with nearest-point
projection measures $\sigma_j$ and restricted ambient distances.
If
\[
 (Z_j,\sigma_j,d_j|_{Z_j\times Z_j})
 \longrightarrow_{d_{\mathrm{conc}}}(X,\nu,d_X),
\]
then $(N_j,\mu_j,d_j)\to_{d_{\mathrm{conc}}}(X,\nu,d_X)$.
\end{teo}
\begin{proof}
By Proposition~\ref{concbox}, the box distance between the ambient
space and its projected counterpart tends to zero.  Since
$d_{\mathrm{conc}}\le d_{\mathrm{box}}$ (with the adopted
conventions) and $d_{\mathrm{conc}}$ satisfies the triangle
inequality, the assertion follows.
\end{proof}

\begin{remark}[Intrinsic distances and equivariance]
The proof uses the restriction of the ambient distance on $Z_j$.
Replacing it by the intrinsic Riemannian distance requires an
additional quantitative comparison of the two distances, even for
totally geodesic submanifolds.  Likewise, equivariant conclusions
in topological dynamics require hypotheses beyond ordinary
metric--measure concentration; see \cite{Sc,Pe1}.
\end{remark}

\subsection{{A quantitative bridge between fixed-dimensional degeneration and dimensional concentration}}
\label{sec:dimensional-obstruction}

For a closed Riemannian manifold $(M,g)$ and a nonempty closed embedded
submanifold $Z\subsetneq M$, write
\[
 T_g(Z,r)=\{x\in M:d_g(x,Z)<r\},\qquad
 \mu_g=\frac{dV_g}{\operatorname{Vol}_g(M)}.
\]
The following theorem makes precise the different geometric mechanisms
available in fixed and increasing dimensions. Here the size of a second
fundamental form is the supremum of the operator norms of its shape operators
in unit normal directions, and $\operatorname{ninj}_g(Z)$ is the normal
injectivity radius.

\begin{theorem}[Uniform tube obstruction and a dimensional alternative]
\label{thm:fixed-varying-dimension}
\textup{(i) Fixed dimension.}
Fix integers $m\geq2$ and $1\leq q\leq m$, and constants
$\kappa,B\geq0$ and $\rho>0$.
Let $(M_j^m,g_j)$ be closed Riemannian manifolds, and let
$Z_j\subset M_j$ be nonempty closed embedded submanifolds of codimension
$q$, satisfying
\begin{equation}\label{eq:bounded-tube-geometry}
 |\sec_{g_j}|\leq\kappa,\qquad
 |\mathrm{II}_{Z_j}|\leq B,\qquad
 \operatorname{ninj}_{g_j}(Z_j)\geq\rho.
\end{equation}
Set $d=m-q$ and
\[
 s_\kappa(r)=
 \begin{cases}
  \sinh(\sqrt\kappa\,r)/\sqrt\kappa,&\kappa>0,\\
  r,&\kappa=0,
 \end{cases}
 \quad
 c_\kappa(r)=\cosh(\sqrt\kappa\,r),
 \quad
 \omega_q=\frac{\pi^{q/2}}{\Gamma(q/2+1)}.
\]
For every $0<r<\rho$ one has
\begin{equation}\label{eq:master-tube-upper}
 \mu_{g_j}\bigl(T_{g_j}(Z_j,r)\bigr)
 \leq C_{m,q,\kappa,B,\rho}\,
 \frac{\operatorname{Vol}_{g_j}(Z_j)}
      {\operatorname{Vol}_{g_j}(M_j)}\,r^q,
\end{equation}
where an admissible constant is
\begin{equation}\label{eq:explicit-constant}
 C_{m,q,\kappa,B,\rho}=\omega_q
 \bigl(c_\kappa(\rho)+B s_\kappa(\rho)\bigr)^d
 \left(\frac{s_\kappa(\rho)}{\rho}\right)^{q-1}.
\end{equation}
Consequently, if
\begin{equation}\label{eq:normal-area-controlled}
\sup_j\frac{\operatorname{Vol}_{g_j}(Z_j)}
               {\operatorname{Vol}_{g_j}(M_j)}<\infty,
\end{equation}
then for every sequence $r_j\downarrow0$,
\[
 \mu_{g_j}\bigl(T_{g_j}(Z_j,r_j)\bigr)\longrightarrow0.
\]
In particular, such $Z_j$ cannot be concentration loci.
Conversely, assume \eqref{eq:bounded-tube-geometry} and
\[
 r_j\downarrow0,\qquad \mu_{g_j}(T_{g_j}(Z_j,r_j))\to1.
\]
Then necessarily
\begin{equation}\label{eq:necessary-normalized-area}
 \frac{\operatorname{Vol}_{g_j}(Z_j)}
      {\operatorname{Vol}_{g_j}(M_j)}
 \geq \frac{1-o(1)}{C_{m,q,\kappa,B,\rho}\,r_j^q}
 \longrightarrow\infty.
\end{equation}

\smallskip
\noindent\textup{(ii) Increasing dimension.}
Fix $R>0$. Let $n_j\to\infty$, $q_j\geq1$, and consider the
round sphere $M_j=S_R^{n_j+q_j}$ with normalized Riemannian volume,
and a great subsphere $Z_j=S_R^{n_j}$ of codimension $q_j$.
There exists a sequence $r_j\downarrow0$ such that
\begin{equation}\label{eq:spherical-concentration}
\mu_j\bigl(T(Z_j,r_j)\bigr)\longrightarrow1
\end{equation}
if and only if
\begin{equation}\label{eq:codimension-criterion}
\frac{q_j}{n_j+q_j+1}\longrightarrow0,
\end{equation}
or, equivalently, $q_j=o(n_j)$.
Writing $a_j=q_j/(n_j+q_j+1)$, whenever $a_j\to0$ one may choose
\begin{equation}\label{eq:radius-choice}
 r_j=R\arcsin(a_j^{1/4}),\qquad
 \mu_j\bigl(M_j\setminus T(Z_j,r_j)\bigr)
 \leq a_j^{1/2}.
\end{equation}
For a fixed $q_j=q$, the normalized volume ratio obeys
\begin{equation}\label{eq:area-ratio-sphere}
 \frac{\operatorname{Vol}(S_R^{n_j})}
      {\operatorname{Vol}(S_R^{n_j+q})}
 \sim\left(\frac{n_j}{2\pi R^2}\right)^{q/2}.
\end{equation}
Thus dimensional concentration is consistent with uniform sectional
curvature, vanishing second fundamental form and a uniform normal
injectivity radius, because the dimensional and normalized-volume
factors in the tube estimate are no longer uniformly controlled.
\end{theorem}

\begin{proof}
\emph{Part (i).}
Fix $j$ and a unit normal $\nu$ to $Z_j$ at $z\in Z_j$.
Along $\gamma(t)=\exp_z(t\nu)$, choose a parallel orthonormal frame
of $\dot\gamma(t)^\perp$.
Horizontal variations of the normal exponential map along an
orthonormal basis $e_1,\ldots,e_d$ of $T_zZ_j$ give Jacobi fields
$J_i$ with
\[
 J_i(0)=e_i,\qquad J_i'(0)=-A_\nu e_i.
\]
Variations of the unit normal direction along an orthonormal basis
$\eta_1,\ldots,\eta_{q-1}$ of $N_zZ_j\cap\nu^\perp$ give Jacobi
fields $Y_a$ satisfying
\[
 Y_a(0)=0,\qquad Y_a'(0)=\eta_a.
\]
Since $|\sec_{g_j}|\leq\kappa$, the Jacobi curvature operator has
norm at most $\kappa$. The integral form of the Jacobi equation and
comparison with $y''=\kappa y$ imply, for $0\leq t\leq\rho$,
\[
 |J_i(t)|\leq c_\kappa(t)+B s_\kappa(t),\qquad
 |Y_a(t)|\leq s_\kappa(t).
\]
Hadamard's determinant inequality therefore bounds the transverse
Jacobian of the normal exponential map by
\[
 \bigl|\det(J_1,\ldots,J_d,Y_1,\ldots,Y_{q-1})\bigr|
 \leq\bigl(c_\kappa(t)+B s_\kappa(t)\bigr)^d
       s_\kappa(t)^{q-1}.
\]
The radial vector has unit length and is perpendicular to these
variations. The normal exponential map is injective for $t<\rho$;
thus integration over $Z_j$, its normal unit spheres and $0<t<r$
gives
\[
 \operatorname{Vol}_{g_j}(T_{g_j}(Z_j,r))
 \leq \operatorname{Vol}_{g_j}(Z_j)|S^{q-1}|
 \int_0^r\bigl(c_\kappa(t)+B s_\kappa(t)\bigr)^d
 s_\kappa(t)^{q-1}\,dt.
\]
The two functions $c_\kappa(t)+B s_\kappa(t)$ and
$s_\kappa(t)/t$ are nondecreasing for $t>0$, so the last integral
is at most
\[
 \bigl(c_\kappa(\rho)+B s_\kappa(\rho)\bigr)^d
 \left(\frac{s_\kappa(\rho)}{\rho}\right)^{q-1}
 \frac{r^q}{q}.
\]
Since $|S^{q-1}|/q=\omega_q$, division by the ambient volume
proves \eqref{eq:master-tube-upper}--\eqref{eq:explicit-constant}.
The remaining assertions in (i) follow immediately.

\smallskip
\emph{Part (ii).}
Write a uniformly distributed point of $S_R^{n_j+q_j}$ as
$R(y,z)$, with
$y\in\mathbb R^{n_j+1}$, $z\in\mathbb R^{q_j}$ and
$|y|^2+|z|^2=1$. The great subsphere is $z=0$, and
\[
 d(R(y,z),Z_j)=R\arcsin|z|.
\]
For $X_j=|z|^2\in[0,1]$, rotational invariance yields
\begin{equation}\label{eq:expectation-X}
 \mathbb E X_j=\frac{q_j}{n_j+q_j+1}=a_j.
\end{equation}
If $a_j\to0$, the radius in \eqref{eq:radius-choice} tends to
zero and Markov's inequality gives
\[
 \mu_j(M_j\setminus T(Z_j,r_j))
 =\mathbb P(X_j\geq a_j^{1/2})\leq a_j^{1/2}.
\]
Conversely, if $r_j\to0$ and the tube probabilities tend to one,
then, using $0\leq X_j\leq1$,
\[
 a_j=\mathbb E X_j
 \leq \sin^2(r_j/R)+
 \mu_j(M_j\setminus T(Z_j,r_j))\longrightarrow0.
\]
As $n_j\to\infty$, $a_j\to0$ is equivalent to $q_j/n_j\to0$.
Finally, the round-sphere volume formula gives
\[
 \frac{\operatorname{Vol}(S_R^{n_j})}
      {\operatorname{Vol}(S_R^{n_j+q})}
 =\frac{\Gamma((n_j+q+1)/2)}
        {\pi^{q/2}R^q\,\Gamma((n_j+1)/2)}.
\]
Stirling's formula, for fixed $q$, gives \eqref{eq:area-ratio-sphere}.
\end{proof}

\section{{Geometric realizations and curvature obstructions}}\label{EX}
\subsection{{Exact tube degeneration and spacetime curvature}}\label{sec:gowdy-bridge}
We now apply the fixed-dimensional obstruction of
Section~\ref{sec:dimensional-obstruction} to the reference family
\eqref{eq:refU}--\eqref{eq:reflambda}.  In the following theorem,
$c=2b/\pi$, $v=a-c>1$, $\alpha=(v^2-1)/4$,
$\beta=(v^2+3)/4$, and $Z_t=\{\theta=\pi\}$.
The second fundamental form $\mathrm{II}_{Z_t}$ is computed
\emph{inside} the Riemannian Cauchy leaf, not in spacetime.

\begin{theorem}[Exact loss of tube control and quantitative curvature coupling]\label{thm:gowdy-tube-curvature}
For the reference vacuum family \eqref{eq:refU}--\eqref{eq:reflambda}, put
\[
 L=|\log t|,\qquad B_*:=\lim_{t\downarrow0}
 \bigl[\lambda(t,\pi)-v^2\log t\bigr],\qquad
 \rho_t:=\operatorname{ninj}_{\gamma_t}(Z_t).
\]
Then $B_*$ is finite and the following assertions hold as $t\downarrow0$.

\noindent\textup{(i) Exact tube geometry.} For every $t>0$,
\begin{equation}\label{eq:gowdy-tube-exact}
 \rho_t=\frac12\int_0^{2\pi}A(t,\theta)\,d\theta,\qquad
 \frac{\Vol_{\gamma_t}(Z_t)}{\Vol_{\gamma_t}(\Sigma_t)}
 =\frac{1}{2\rho_t},\qquad
 \mu_t\bigl(T_{\gamma_t}(Z_t,r)\bigr)=\frac r{\rho_t}
 \quad(0<r<\rho_t).
\end{equation}
In particular $\mathrm{II}_{Z_t}=0$ identically.

\noindent\textup{(ii) Sharp loss of normal injectivity and normalized area control.}
\begin{equation}\label{eq:gowdy-rho-asymptotic}
 \rho_t\sim e^{B_*/4}\sqrt{\frac{\pi}{vc}}\,
 \frac{t^\alpha}{\sqrt L}\longrightarrow0,
 \qquad
 \frac{\Vol_{\gamma_t}(Z_t)}{\Vol_{\gamma_t}(\Sigma_t)}
 \sim\frac{e^{-B_*/4}}2\sqrt{\frac{vc}{\pi}}\,
 t^{-\alpha}\sqrt L\longrightarrow\infty.
\end{equation}
Taking $r_t=(1-L^{-1})\rho_t$ for $L>1$, one has
$r_t\to0$ and $\mu_t(T_{\gamma_t}(Z_t,r_t))=1-L^{-1}\to1$.

\noindent\textup{(iii) Spatial curvature and the tube radius.} Let $k_t$ denote the sectional curvature of $(\Sigma_t,\gamma_t)$ at $Z_t$ in the plane spanned by the unit $\theta$- and $x$-directions. Then
\begin{equation}\label{eq:gowdy-sec-asymptotic}
 k_t\sim\frac{c\,e^{-B_*/2}}2\,
 L\,t^{-2\alpha}\longrightarrow+\infty,
 \qquad
 \boxed{\ \lim_{t\downarrow0}\rho_t^2k_t=\frac{\pi}{2v}\ }.
\end{equation}
Thus the uniform ambient sectional-curvature bound, the uniform normal-injectivity bound, and the normalized area bound in the fixed-dimensional tube obstruction all fail, while the second fundamental form of $Z_t$ remains zero.

\noindent\textup{(iv) Spacetime curvature and finite proper time.} Let
$\mathcal K=\mathbf R_{\alpha\beta\gamma\delta}
\mathbf R^{\alpha\beta\gamma\delta}$ be the Kretschmann scalar.
Along the timelike normal geodesic $\Gamma(t)=(t,\pi,x_0,y_0)$,
\begin{align}
 \mathcal K(t,\pi)&\sim\frac{(v^2-1)^2(v^2+3)}4
             e^{-B_*}\,t^{-(v^2+3)}\longrightarrow+\infty,
             \label{eq:gowdy-Kretschmann-sharp}\\
 \tau(t)&:=\int_0^t A(s,\pi)\,ds
      \sim\frac{e^{B_*/4}}{\beta}t^\beta<\infty.
      \label{eq:gowdy-proper-time-sharp}
\end{align}
The spatial and spacetime curvatures satisfy the further sharp identity
\begin{equation}\label{eq:gowdy-curvature-link}
 \boxed{\ \lim_{t\downarrow0}
 \frac{t^4L^2\,\mathcal K(t,\pi)}{k_t^2}
 =\frac{(v^2-1)^2(v^2+3)}{c^2}>0\ }.
\end{equation}
Consequently $\Gamma$ is future incomplete and its curvature diverges at finite remaining proper time.
\end{theorem}

\begin{proof}
\emph{Vacuum equations and preliminary asymptotics.}
The Bessel identities $Y_0'=-Y_1$ and
$Y_1'=Y_0-Y_1/t$ give \eqref{eq:gowdy} by direct
substitution in \eqref{eq:refU}--\eqref{eq:reflambda}.
The small-argument expansions of $Y_0,Y_1$ imply, in $C^2(S^1)$,
\begin{align}\label{eq:gowdy-sharp-asymptotic}
 U(t,\theta)&=(a+c\cos\theta)\log t+w(\theta)+O(t^2L),\\
 \lambda(t,\theta)&=(a+c\cos\theta)^2\log t+B(\theta)
             +O(t^2L^2),\nonumber
\end{align}
with smooth periodic $w,B$; the same bounds hold for the needed
$t$-derivatives after differentiation of the explicit formula.
Write $v(\theta)=a+c\cos\theta$ and $f(\theta)=v(\theta)^2$.
The unique nondegenerate global minimum is at $\theta_*=\pi$,
with $f(\pi)=v^2$ and $f''(\pi)=2vc>0$.
In particular,
\begin{equation}\label{eq:gowdy-Astar}
 A(t,\pi)=e^{B_*/4}t^\alpha(1+O(t^2L^2)).
\end{equation}

\emph{Exact tubes.}
On a spatial slice let $s$ be the arclength coordinate
on the base circle, $ds=A(t,\theta)d\theta$.
Then
\[
 \gamma_t=ds^2+F(s)^2dx^2+G(s)^2dy^2,\quad
 F=\sqrt t\,e^{U/2},\quad G=\sqrt t\,e^{-U/2},\quad FG=t.
\]
For each point, the distance to $Z_t$ is exactly the minimum of the
two arclength distances around the base circle: every path has length
at least its projected base-circle length, and equality is attained by
the curve with $x,y$ fixed. The normal geodesics to $Z_t$ are precisely
these constant-$(x,y)$ curves. Their differential in the torus
directions never vanishes since $F,G>0$. They first meet at the
opposite point of the base circle, whose distance is one half the
circumference. This proves the first identity in \eqref{eq:gowdy-tube-exact}.
Moreover $dV_{\gamma_t}=FG\,ds\,dx\,dy=t\,ds\,dx\,dy$.
The torus area is $(2\pi)^2t$, the total volume is $(2\pi)^2t(2\rho_t)$,
and for $r<\rho_t$ the tube volume is $(2\pi)^2t(2r)$.
The remaining identities in \eqref{eq:gowdy-tube-exact} follow.
The spatial shape operator of $Z_t$ has principal curvatures
$\pm U_\theta(t,\pi)/(2A(t,\pi))=0$, proving
$\mathrm{II}_{Z_t}=0$.

\emph{Radius and volume ratio.}
The Laplace method applied to
\[
 2\rho_t=\int_0^{2\pi}e^{B(\theta)/4}
      t^{(f(\theta)-1)/4}(1+O(t^2L^2))\,d\theta
\]
and to the nondegenerate minimum of $f$ yields
\[
 2\rho_t\sim e^{B_*/4}t^\alpha
     \sqrt{\frac{8\pi}{2vc\,L}},
\]
which is \eqref{eq:gowdy-rho-asymptotic}. The exact tube formula gives the
concentration assertion, including its rate.

\emph{Spatial sectional curvature.}
For the doubly warped metric $ds^2+F(s)^2dx^2+G(s)^2dy^2$,
\[
 \operatorname{Sec}(\partial_s,F^{-1}\partial_x)=-\frac{F_{ss}}F
 =-\frac12 U_{ss}-\frac14 U_s^2,
\quad
 U_s=\frac{U_\theta}{A},\quad
 U_{ss}=\frac{U_{\theta\theta}-(A_\theta/A)U_\theta}{A^2}.
\]
At $\theta=\pi$, both $U_\theta$ and $A_\theta$ vanish, and
$U_{\theta\theta}(t,\pi)=bY_0(t)=c\log t+O(1)$.
Hence
\[
 k_t=-\frac{bY_0(t)}{2A(t,\pi)^2}
 \sim\frac{c\,e^{-B_*/2}}2 L\,t^{-2\alpha}.
\]
Together with \eqref{eq:gowdy-rho-asymptotic} this gives
$\rho_t^2k_t\to\pi/(2v)$.

\emph{Four-dimensional curvature and proper time.}
At $\theta=\pi$, reflection symmetry implies $A_\theta=U_\theta=0$;
thus the coordinate line $\Gamma$ with fixed $x,y,\theta=\pi$ is a
timelike geodesic after reparametrization by proper time.
Equation \eqref{eq:gowdy-Astar} gives \eqref{eq:gowdy-proper-time-sharp}.
In the orthonormal coframe
\[
 \vartheta^0=A\,dt,\qquad \vartheta^1=A\,d\theta,
 \qquad\vartheta^2=\sqrt t\,e^{U/2}dx,
 \qquad\vartheta^3=\sqrt t\,e^{-U/2}dy,
\]
the time-derivative connection coefficients are of order
$(At)^{-1}$, whereas the terms involving first or second $\theta$-
derivatives of $U,\lambda$ are smaller in curvature by
$O(t|\log t|^2)$ or $O(t^2|\log t|^2)$, respectively, relative to
$(At)^{-2}$; these follow directly by differentiating
\eqref{eq:refU}--\eqref{eq:reflambda}.
Thus the leading orthonormal Riemann components along $\Gamma$
are those of a Kasner metric with proper time $\tau$ and exponents
\[
 p_1=\frac{v^2-1}{v^2+3},\qquad
 p_2=\frac{2(1+v)}{v^2+3},\qquad
 p_3=\frac{2(1-v)}{v^2+3}.
\]
They satisfy $\sum p_i=\sum p_i^2=1$. In an orthonormal Kasner
frame the leading terms are $R_{0i0i}=-p_i(p_i-1)\tau^{-2}$ and
$R_{ijij}=p_ip_j\tau^{-2}$, so
\[
 \mathcal K=4\tau^{-4}\left(
  \sum_i p_i^2(p_i-1)^2+
  \sum_{i<j}p_i^2p_j^2+o(1)\right),
\quad
 \sum_i p_i^2(p_i-1)^2+\sum_{i<j}p_i^2p_j^2
 =\frac{16(v^2-1)^2}{(v^2+3)^3}.
\]
Inserting $\tau\sim e^{B_*/4}t^\beta/\beta$ and
$\beta=(v^2+3)/4$ yields \eqref{eq:gowdy-Kretschmann-sharp}.
Finally, dividing \eqref{eq:gowdy-Kretschmann-sharp} by the square of
\eqref{eq:gowdy-sec-asymptotic} gives \eqref{eq:gowdy-curvature-link}.
Since $\tau(t_0)<\infty$ and $\mathcal K\to\infty$ on $\Gamma$,
there is a finite-proper-time curvature singularity along this geodesic.
\end{proof}

\begin{remark}[Intrinsic versus coordinate concentration]
For $0<r<\rho_t$, the probability of the intrinsic metric tube is
exactly $r/\rho_t$.  Thus normalized volume is uniform in
transverse arclength, although its density relative to the fixed
$\theta$ coordinate localizes near $\theta=\pi$.
Since the full transverse circle shrinks, any fixed coordinate torus
$\{\theta=\theta_0\}$ is also a tubular concentration locus for
suitable radii tending to zero.  The torus $Z_t$ is distinguished by
the limiting density in the areal identification, not by uniqueness
of the intrinsic tubular concentration locus.
The spatial second fundamental form of $Z_t$ vanishes, whereas its
normal injectivity radius, the normalized area ratio and the ambient
spatial curvature violate the uniform bounds of
Theorem~\ref{thm:fixed-varying-dimension}.  The curvature relation
\eqref{eq:gowdy-curvature-link} is proved for the reference Gowdy
family; it does not give a converse or an implication for arbitrary
vacuum initial data.
\end{remark}

\subsection{Stationary and homogeneous examples}
In a static vacuum spacetime with metric
$\mathbf g=-N(x)^2ds^2+\gamma_{ij}(x)dx^i dx^j$, the spatial
metric is independent of $s$.  Its slices have $K=0$; hence their
normalized spatial volume does not redistribute under the static
flow.  Stationarity alone does not imply $K=0$.  For a stationary
Kerr exterior, for example, constant Boyer--Lindquist time slices
carry a time-independent spatial metric (and hence a fixed volume
measure under the stationary identification), although their
extrinsic curvature need not vanish.  These observations are
foliation-specific and make no assertion about a singular endpoint.

Another instructive contrast is the homogeneous interior of the
Oppenheimer--Snyder dust solution.  Its comoving metric is
$-d\tau^2+a(\tau)^2g_{\mathrm{round}}$ on the dust region.  The
spatial volume scales as $a^3$ while the normalized comoving volume
measure is constant.  The dust mass in a material region is also
conserved; collapse of the scale factor alone therefore does not
produce transfer of normalized probability between comoving regions.
It may nevertheless produce collapse in the evolving spatial metric.

\subsection{Matter concentration and an additional curvature obstruction}
The role of a material measure is distinct from that of normalized
volume.  For example, let $u$ be the four-velocity of pressureless
dust, with $T_{\mu\nu}=\varrho u_\mu u_\nu$, and suppose a material
region $S_\tau$ is transported by the dust flow before caustics.
{For this estimate assume the slices are orthogonal to $u$
(for example, a regular irrotational comoving dust foliation). On an
arbitrary spacelike slice with future unit normal $n$, the conserved
mass instead has integrand $\varrho(-g(u,n))$; without control of this
Lorentz factor the following rest-density bound is not justified.}
The continuity equation gives a conserved positive mass
$M_S=\int_{S_\tau}\varrho\,dV_{\gamma_\tau}$.  If the proper
volume $V_S(\tau)$ tends to zero, then
\[
 \sup_{S_\tau}\varrho\ge\frac{M_S}{V_S(\tau)}.
\]
In geometrized units and with vanishing cosmological constant,
the Einstein equations for dust imply
$\operatorname{Ric}_{\mu\nu}\operatorname{Ric}^{\mu\nu}
=64\pi^2\varrho^2$.  Hence
\begin{equation}\label{eq:matter-ricci-obstruction}
 \sup_{S_\tau}\operatorname{Ric}_{\mu\nu}
 \operatorname{Ric}^{\mu\nu}
 \ge 64\pi^2\frac{M_S^2}{V_S(\tau)^2}.
\end{equation}
This is a quantitative curvature obstruction from conservation of
material mass and absolute-volume collapse.

\section*{Acknowledgments}
We are extremely grateful to Alessio Figalli for his insightful suggestions and precise remarks. We also thank Daniele Semola and Takashi Shioya for some helpful discussions.

P.Ursino acknowledges the financial support from Universit\`a degli Studi di Catania, PIAno di inCEntivi per la RIcerca di Ateneo 2024/2026 - Linea di intervento 1. We used ChatGPT as artificial-intelligence-assisted tools in preparing
this manuscript.

\end{document}